\documentclass[10pt]{article}

\usepackage{amsmath}
\usepackage{amsthm}
\usepackage{graphicx}
\usepackage{amsfonts}
\usepackage{amssymb}%
\usepackage{latexsym}
\usepackage[normalem]{ulem}

\usepackage{marvosym}

\usepackage{caption}
\usepackage{subcaption}
\usepackage[a4paper]{geometry}
\newtheorem{theorem}{Theorem}[section]
\newtheorem{corollary}[theorem]{Corollary}
\newtheorem{definition}[theorem]{Definition}
\newtheorem{lemma}[theorem]{Lemma}
\newtheorem{proposition}[theorem]{Proposition}

\numberwithin{equation}{section}
\numberwithin{table}{section}
\numberwithin{figure}{section}

\usepackage[a4paper]{geometry}
\newtheorem{remark}[theorem]{Remark}
\newtheorem*{remark*}{Remark}
\newtheorem{example}[theorem]{Example}

\usepackage{multirow}
\usepackage{tikz}
\usepackage{caption}
\usetikzlibrary{
  decorations.pathmorphing,%
  decorations.markings,%
  decorations.pathreplacing,%
  decorations.text,%
  calc,%
  patterns,%
  shapes.geometric,%
  arrows,
  decorations.shapes,
  positioning,
  plotmarks,
  shadings,
  math,
  intersections
}

\definecolor{jeffColor}{RGB}{102, 0, 204}
\definecolor{yaizaColor}{RGB}{0, 153, 153}

\definecolor{periodColor}{RGB}{255, 167, 105}
\definecolor{dark-green}{RGB}{135, 194, 130}
\tikzset{>=latex} 
\tikzset{font=\small}
\tikzset{mark size=1.5pt, mark options=thin}
\tikzset{pin distance=4pt,
  every pin edge/.style={<-, thin, shorten <= -2pt}}
\usepackage{array,booktabs,tabularx}
\usepackage{graphicx}
\usepackage{amssymb}
\usepackage{enumerate}
\usepackage{color}
\usepackage{hyperref}
\usepackage{esint}
\usepackage{mathtools}
\usepackage{dsfont}
\usepackage{ esint }
\usepackage{placeins}

\usepackage{comment}
\usepackage{enumitem}
\usepackage{soul}
\usepackage{needspace}
\usepackage[normalem]{ulem}

\usepackage{multicol}
\usepackage{tabto}

\usepackage{dsfont}

\newcommand{\x}{\ensuremath{\times}}

\newcommand{\mc}[1]{\mathcal{#1}}
\newcommand{\WFh}{\operatorname{WF_{\h}}}

\newcommand{\comp}{\operatorname{comp}}

\newcommand{\w}{\omega}

\renewcommand{\a}{{\bf a}}

\def\XXint#1#2#3{{\setbox0=\hbox{$#1{#2#3}{\int}$} \vcenter{\hbox{$#2#3$}}\kern-.5\wd0}}

\DeclareMathOperator{\rank}{rank}

\DeclareMathOperator{\supp}{supp}

\newcommand{\e}{\varepsilon}

\newcommand{\tr}{{\operatorname{tr}}}
\newcommand{\tc}{{\cL_1}}

\renewcommand{\Re}{\operatorname{Re}}
\renewcommand{\Im}{\operatorname{Im}}

\renewcommand{\r}{{\bf{r}}}

\newcommand{\h}{\hbar}
\definecolor{uipoppy}{RGB}{221,128,71}
\definecolor{uipaleblue2}{RGB}{179,196,215}
\definecolor{uiviolet}{RGB}{86,86,99}
\definecolor{uiblack}{RGB}{0, 0, 0}
\definecolor{azul}{RGB}{0,128,255}
\definecolor{verde}{RGB}{50,180,50}
\definecolor{uipaleblue}{RGB}{108,199,220}
\definecolor{light-gray}{gray}{0.8}
\definecolor{light-blue}{rgb}{0.53,.8,98}
\definecolor{green1}{RGB}{50,180,50}

\definecolor{jeffColor}{RGB}{102, 0, 204}
\definecolor{yaizaColor}{RGB}{0, 153, 153}
\definecolor{pale-verde}{RGB}{155,207,145}

\definecolor{periodColor}{RGB}{255, 167, 105}
\definecolor{dark-green}{RGB}{135, 194, 130}

\DeclareMathOperator{\Oph}{Op_{\h}}

\newcommand{\loc}{\operatorname{loc}}
\usepackage{hyperref}

\newcommand{\beq}{\begin{equation}}
\newcommand{\eeq}{\end{equation}}
\newcommand{\beqs}{\begin{equation*}}
\newcommand{\eeqs}{\end{equation*}}
\newcommand{\bit}{\begin{itemize}}
\newcommand{\eit}{\end{itemize}}
\newcommand{\ben}{\begin{enumerate}}
\newcommand{\een}{\end{enumerate}}
\newcommand{\bal}{\begin{align}}
\newcommand{\eal}{\end{align}}
\newcommand{\bals}{\begin{align*}}
\newcommand{\eals}{\end{align*}}
\newcommand{\bse}{\begin{subequations}}
\newcommand{\ese}{\end{subequations}}
\newcommand{\bpr}{\begin{proposition}}
\newcommand{\epr}{\end{proposition}}
\newcommand{\bre}{\begin{remark}}
\newcommand{\ere}{\end{remark}}
\newcommand{\bpf}{\begin{proof}}
\newcommand{\epf}{\end{proof}}
\newcommand{\ble}{\begin{lemma}}
\newcommand{\ele}{\end{lemma}}
\newcommand{\bco}{\begin{corollary}}
\newcommand{\eco}{\end{corollary}}
\newcommand{\bex}{\begin{example}}
\newcommand{\eex}{\end{example}}
\newcommand{\bth}{\begin{theorem}}
\newcommand{\enth}{\end{theorem}}

\newcommand{\Rea}{\mathbb{R}}

\newcommand{\tendi}{\rightarrow \infty}
\newcommand{\tendo}{\rightarrow 0}
\newcommand{\cH}{{\mathcal H}}

\newcommand{\Otr}{{\Omega_{\rm tr}}}
\newcommand{\Oi}{{\Omega_-}}

\newcommand{\Gtr}{\Gamma_{\rm tr}}

\newcommand{\pdiff}[2]{\frac{\partial #1}{\partial #2}}

\newcommand{\ri}{{\rm i}}
\newcommand{\re}{{\rm e}}
\newcommand{\bx}{x}
\newcommand{\hatx}{\widehat{\bx}}

\newcommand{\DtN}{\mathcal{D}}
\newcommand{\eps}{\varepsilon}
\newcommand{\cO}{\mathcal{O}}
\newcommand{\cL}{\mathcal{L}}

\newcommand{\QMC}{\epsilon}

\newcommand{\indicatorR}{1^{\rm res}_{\Otr}}
\newcommand{\indicatorE}{1^{\rm ext}_{\Otr}}

\newcommand{\newP}{{\mathsf{P}}}

\newcommand{\projx}{\pi_{\Rea}}

\newcommand{\klower}{k_j^-}
\newcommand{\kupper}{k_j^+}

\newcommand{\Op}{\operatorname{Op}}

\usepackage{hyperref}
\definecolor{myblue}{rgb}{0,0,0.6}
\hypersetup{colorlinks=true,
linkcolor=myblue,citecolor=myblue,filecolor=myblue,urlcolor=myblue}
\newcommand*{\N}[1]{\left\|#1\right\|}

\allowdisplaybreaks[4]

\newcommand{\tfa}{\text{ for all }}
\newcommand{\tfor}{\text{ for }}

\newcommand{\tin}{\text{ in }}
\newcommand{\ton}{\text{ on }}
\newcommand{\tas}{\text{ as }}
\newcommand{\tand}{\text{ and }}
\newcommand{\tst}{\text{ such that }}

\newcommand{\mythmname}[1]{\textbf{\emph{(#1.)}}}

\usepackage{soul}
\definecolor{escol}{rgb}{0,0,0.8}
\definecolor{estcol}{rgb}{0,0.5,0}
\definecolor{esnewcol}{rgb}{0,0.5,0}

\begin{document}

\title{Eigenvalues of the truncated Helmholtz solution operator under strong trapping}

\author{Jeffrey Galkowski\footnotemark[1]\,\,, Pierre Marchand\footnotemark[2]\,\,, Euan A.~Spence\footnotemark[3]}

\renewcommand{\thefootnote}{\fnsymbol{footnote}}

\footnotetext[1]{Department of Mathematics, University College London, 25 Gordon Street, London, WC1H 0AY, UK, \tt J.Galkowski@ucl.ac.uk}
\footnotetext[2]{Department of Mathematical Sciences, University of Bath, Bath, BA2 7AY, UK, \tt pfcm20@bath.ac.uk}
\footnotetext[3]{Department of Mathematical Sciences, University of Bath, Bath, BA2 7AY, UK, \tt E.A.Spence@bath.ac.uk }

\date{\today}

\maketitle

\begin{abstract}
For the Helmholtz equation posed in the exterior of a Dirichlet obstacle, we prove that if there exists a family of quasimodes (as is the case when the exterior of the obstacle has stable trapped rays), then there exist near-zero eigenvalues of the standard variational formulation of the exterior Dirichlet problem (recall that this formulation involves truncating the exterior domain and applying the exterior Dirichlet-to-Neumann map on the truncation boundary). 

Our motivation for proving this result is that 
a) the finite-element method for computing approximations to solutions of the Helmholtz equation is based on the standard variational formulation, and 
b) the location of eigenvalues, and especially near-zero ones, plays a key role in understanding how
iterative solvers such as the generalised minimum residual method (GMRES) behave when
used to solve linear systems, in particular those arising from the finite-element method.
The result proved in this paper is thus the first step towards rigorously understanding how GMRES behaves when applied to 
discretisations of high-frequency Helmholtz problems under strong trapping (the subject of 
the companion paper \cite{MaGaSpSp:21}). 



\end{abstract}

\section{Introduction}

\subsection{Preliminary definitions}

Let $\Omega_- \subset\Rea^d, d\geq 2$ be a bounded open set such that its open complement $\Omega_+:= \Rea^d\setminus \overline{\Omega_-}$ is connected. Let $\Gamma_D:= \partial \Omega_-$, where the subscript $D$ stands for ``Dirichlet''. Let $\Omega_1$ be another bounded open set with connected open complement and such that $\operatorname{conv}(\Omega_-)\Subset \Omega_1$, where 
$\operatorname{conv}$ denotes the convex hull and $\Subset$ denotes compact containment.
Let $\Otr:= \Omega_1 \setminus \Omega_-$, and $\Gamma_{\tr}:= \partial \Omega_1$, where the subscript $\tr$ stands for ``truncated''.
{We assume throughout that $\Gamma_D$ and $\Gamma_\tr$ are both $C^\infty$.}
Let $\gamma_0^D$ and $\gamma_0^{\rm tr}$ denote the Dirichlet traces on $\Gamma_D$ and $\Gtr$ respectively, and let 
$\gamma_1^D$ and $\gamma_1^{\rm tr}$ denote the respective Neumann traces, where the normal vector points out of $\Otr$ on both $\Gamma_D$ and $\Gamma_\tr$.
 Let
\beqs
H_{0,D}^1(\Otr):= \big\{ v\in H^1(\Otr) : \gamma_0^D v=0 \big\}.
\eeqs

Let $\DtN(k): H^{1/2}(\Gtr) \rightarrow H^{-1/2}(\Gtr)$ be the Dirichlet-to-Neumann map for the equation $\Delta u+k^2 u=0$ posed in the exterior of $\Omega_1$ with the Sommerfeld radiation condition
\beq\label{eq:src}
\pdiff{u}{r}(\bx) - \ri k u(\bx) = o \left( \frac{1}{r^{(d-1)/2}}\right)
\eeq
as $r:= |\bx|\tendi$, uniformly in $\hatx:= \bx/r$. We say that a function satisfying \eqref{eq:src} is \emph{$k$-outgoing}.
 When $\Gtr = \partial B_R$, for some $R>0$,  the definition of $\DtN(k)$ in terms of Hankel functions and polar coordinates (when $d=2$)/spherical polar coordinates (when $d=3$) is given in, e.g., \cite[Equations 3.7 and 3.10]{MeSa:10}.

\begin{definition}[Eigenvalues of the truncated exterior Dirichlet problem]\label{def:eigenvalue}
We say $\mu_\ell$ is an \emph{eigenvalue of the truncated exterior Dirichlet problem at frequency $k_\ell>0$}, with corresponding eigenfunction $u_\ell$, if $u_\ell\in H_{0,D}^1(\Otr)\setminus\{0\}$ and $\mu_\ell\in \mathbb{C}$ satisfies
\beqs
(\Delta+k_\ell^2)u_\ell=\mu_\ell u_\ell\quad\text{ in }\Omega_{\tr}\quad\tand\quad
\gamma_1^{\tr} u_\ell=\DtN(k_\ell) (\gamma_0^{\tr} u_\ell).
\eeqs
\end{definition}

\begin{definition}[Quasimodes]\label{def:quasimodes}
A \emph{family of quasimodes of quality $\QMC(k)$}
is a sequence $\{(u_\ell,k_\ell)\}_{\ell=1}^\infty\subset H^2(\Otr)\cap H_{0,D}^1(\Otr)\times \mathbb{R}$ such that the frequencies $k_\ell\tendi$ as $\ell \tendi$ and there is a compact subset $\mc{K}\Subset \Omega_1$ such that, for all $\ell$, $\supp\, u_\ell \subset \mc{K}$,
\beqs
\N{(\Delta +k_\ell^2) u_\ell}_{L^2(\Otr)} \leq \QMC(k_\ell) \quad\tand\quad\N{u_\ell}_{L^2(\Otr)}=1.
\eeqs
\end{definition}

\bre
\label{r:lowerBoundQuasi}
By~\cite[Theorem 2]{Bu:98}, we can assume that there exist $S_1,S_2>0$ such that $\QMC(k)\geq S_1\exp(-S_2 k)$. 
\ere

\begin{definition}[Quasimodes with multiplicity]\label{def:quasimodes_mult}
Let $\{(u_\ell,k_\ell)\}_{\ell=1}^\infty$ be a quasimode with quality $\QMC(k)$ and let $\{(m_j,k_j^-,k_j^+)\}_{j=1}^\infty\subset \mathbb{N}\times \mathbb{R}^2$ be such that $k_j^-\to \infty$ and $k_j^-\leq k_j^+$. 
Define  
$$
\mathcal{W}_j:=\big\{\ell\,:\, k_\ell\in[k_j^-,k_j^+]\big\}.
$$
We say $u_\ell$ has multiplicity $m_j$ in the window $[k_j^-,k_j^+]$ if
\begin{gather*}
|\mathcal{W}_j|=m_j,\qquad |\langle u_{\ell_1},u_{\ell_2}\rangle_{L^2(\Otr)}|\leq \epsilon(k_j^-)\quad\tfor \ell_1\neq \ell_2,\,\,\,\ell_1,\ell_2\in \mathcal{W}_j.
\end{gather*}
\end{definition}
We assume throughout that the quality, $\QMC(k)$, of a quasimode is a decreasing function of $k$; this can always be arranged by replacing $\epsilon(k)$ by $\tilde{\epsilon}(k):=\sup_{\tilde{k}\geq k} \epsilon(\tilde{k})$.

We use the notation that 
$A = \cO(k^{-\infty})$ as $k\tendi$ if, given $N>0$, there exists $C_N$ and $k_0$ such that $|A|\leq C_N k^{-N}$ for all $k\geq k_0$, i.e.~$A$ decreases superalgebraically in $k$.

\subsection{The main results}

\begin{theorem}[From quasimodes to eigenvalues]\label{thm:main1}
Let $\alpha> 3(d+1)/2$.
Suppose there exists a family of quasimodes of quality $\QMC(k)$ with 
\beqs
\QMC(k) \ll k^{1-\alpha}.
\eeqs
Then there exists $k_0>0$ (depending on $\alpha$) such that, if $\ell$ is such that $k_\ell\geq k_0$, then there exists 
an eigenvalue of the truncated exterior Dirichlet problem at frequency $k_\ell$ satisfying 
\beqs
|\mu_\ell| \leq k_\ell^\alpha\QMC(k_\ell),
\eeqs
\end{theorem}

We now give three specific cases when the assumptions of Theorem \ref{thm:main1} hold. The first two cases are via the quasimode constructions of 
\cite[Theorem 2.8, Equations 2.20 and 2.21]{BeChGrLaLi:11} and \cite[Theorem 1]{CaPo:02} for obstacles whose exteriors support elliptic-trapped rays. The third case is via the ``resonances to quasimodes'' result of \cite[Theorem 1]{St:00}; recall that the resonances of the exterior Dirichlet problem are the poles of the meromorphic continuation of the solution operator 
from $\Im k\geq 0$ to $\Im k<0$; see, e.g., \cite[Theorem 4.4. and Definition 4.6]{DyZw:19}.

\begin{lemma}[Specific cases when the assumptions of  Theorem \ref{thm:main1} hold]\label{lem:specific}

\

(i) Let $d=2$. Given $a_1>a_2>0$, let 
\beq\label{eq:ellipse}
E:= \left\{(x_1,x_2) \, : \, \left(\frac{x_1}{a_1}\right)^2+\left(\frac{x_2}{a_2}\right)^2<1\right\}.
\eeq
If $\Gamma_D$ coincides with the boundary of $E$ in the neighborhoods of the points \((0,\pm a_2)\), 
and if $\Omega_+$ contains the convex hull of these neighbourhoods, 
then the assumptions of Theorem \ref{thm:main1} hold with 
\beqs
\QMC(k)= \exp( - C_1 k)
\eeqs
for some $C_1>0$ (independent of $k$).\footnote{
In \cite[Theorem 2.8]{BeChGrLaLi:11}, $\Omega_+$ is assumed to contain the whole ellipse $E$.
However, inspecting the proof, we see that the result remains unchanged if $E$ is replaced with the convex hull of the neighbourhoods of 
$(0,\pm a_2)$. Indeed, the idea of the proof is to consider a family of eigenfunctions of the ellipse localising around the periodic orbit $\{(0,x_2) : |x_2|\leq a_2\}$.
}

(ii) Suppose $d\geq 2$, $\Gamma_D\in C^\infty$, and $\Omega_+$ contains an elliptic-trapped ray such that (a) $\Gamma_D$ is analytic in a neighbourhood of the ray and (b) the ray satisfies the stability condition \cite[(H1)]{CaPo:02}. If  $q >11/2$ when $d=2$ and $q>2d+1$ when $d\geq 3$, then the assumptions of Theorem \ref{thm:main1} hold with 
\beqs
\QMC(k)= \exp( - C_2 k^{1/q})
\eeqs
for some $C_2>0$ (independent of $k$).

(iii) Suppose there exists a sequence of resonances $\{\lambda_\ell\}_{\ell=1}^\infty$ of the exterior Dirichlet problem with
\beq\label{eq:resonances}
0\leq -\Im \lambda_\ell = \mathcal{O}\big(|\lambda_\ell|^{-\infty}\big)  \quad\tand \quad \Re \lambda_\ell \tendi \quad\tas\quad \ell \tendi.
\eeq
Then there exists a family of quasimodes of quality $\QMC(k)=\mathcal{O}(k^{-\infty})$ and thus the assumptions of Theorem \ref{thm:main1} hold.
\end{lemma}

\bre[Resonances $\iff$ quasimodes $\iff$eigenvalues]
Part (iii) of Lemma \ref{lem:specific} is the ``resonances to quasimodes'' result of \cite[Theorem 1]{St:00}. The converse implication, i.e.~that a family of quasimodes of quality $\QMC(k)=\mathcal{O}(k^{-\infty})$ implies a sequence of resonances satisfying \eqref{eq:resonances}, was proved in \cite{TaZw:98}, \cite{St:99} (following \cite{StVo:95, StVo:96}), see also \cite[Theorem 7.6]{DyZw:19}. Therefore the ``quasimodes to eigenvalues'' result of Theorem \ref{thm:main1} is equivalent to a ``resonances to eigenvalues'' result.
In fact, in Appendix~\ref{a:eToQ} we show that the existence of $\cO(k^{-\infty})$ eigenvalues implies the existence of quasimodes of quality $\cO(k^{-\infty})$. We therefore have that resonances $\iff$ quasimodes $\iff$ eigenvalues.
\ere

With $\{\mu_j(k)\}_j$ the set of eigenvalues,  counting multiplicities, of the truncated exterior Dirichlet problem at frequency $k$ (with $\mu_j(k)$ depending continuously on $k$ for each $j$), let
\beq\label{eq:mathcalE}
\mc{E}(\e_1,\e_0,k_-,k_+):=\Big\{j\,:\, \mu_j(k)\in (-2\e_1,2\e_1)-\ri(0,2\e_0)\text{ for some }k\in [k_-,k_+]\Big\};
\eeq
$|\mc{E}|$ is therefore the counting function of the eigenvalues, $\mu_j(k)$, that pass through a rectangle next to zero in $\mu$  as $k$ varies in the interval $[k_-,k_+]$; see 
Figure \ref{fig:box1}. \footnote{In Figure \ref{fig:box1} we have drawn the paths of the eigenvalues as arbitrary curves. We see later in Figure \ref{fig:flow} an example where the paths appear to be horizontal lines; this is consistent with the intuition that eigenvalues should be shifted resonances.}

\begin{figure}
\begin{center}
\begin{tikzpicture}
\def \w{1.8};
\def \h{1.2};
\draw[->] (-{1.5*\w},0)--({1.5*\w},0)node[right]{$\mathrm{Re}\, \mu$};
\draw[->] (0,{-2*\h})--(0,1)node[above]{$\mathrm{Im}\, \mu$};
\draw[dashed](0,-\h)--({1.8*\w},-\h)node[right]{$-2\e_0(k)$};
\draw[fill =light-gray,opacity=.5] (-\w,0)--(-\w,-\h)--(\w,-\h)--(\w,0)--cycle;
\draw (\w,0)node[above]{$2\e_1(k)$};
\draw [verde, opacity=.6] plot [smooth] coordinates { (-{1.5*\w},-{2*\h}) ({-.9*\w}, -{.75*\h}) (\w,-{1.5*\h}) };
\draw[verde] (-{1.5*\w},-{2*\h})node[below]{$\mu_1(k_-)$};
\draw [verde](\w,-{1.5*\h})node[below]{$\mu_1(k_+)$};
\draw [verde,opacity=.6] plot [smooth] coordinates { (-{1.5*\w},-{.3*\h}) ({-.1*\w}, -{.3*\h}) (.5*\w,-{2*\h}) };
\draw[verde] ({-1.5*\w},-{.3*\h})node[left]{$\mu_2(k_-)$};
\draw[verde] (.5*\w,-{2*\h})node[below]{$\mu_2(k_+)$};
\draw [verde,opacity=.6] plot [smooth] coordinates { ({1.3*\w},-{1.3*\h}) ({-.1*\w}, -{.3*\h}) (-1*\w,-{2.5*\h}) };
\draw [verde]({1.3*\w},-{1.3*\h})node[right]{$\mu_3(k_-)$};
\draw[verde](-1*\w,-{2.5*\h})node[below]{$\mu_3(k_+)$};
\draw [blue,opacity=.6] plot [smooth] coordinates { ({1.3*\w},-{.5*\h}) ({1*\w}, -{1.3*\h}) (0*\w,-{1.8*\h})  (-{.4*\w},{-2.5*\h})};
\draw[blue] ({1.3*\w},-{.5*\h})node[right]{$\mu_4(k_-)$};
\draw[blue](-{.4*\w},{-2.5*\h})node[below]{$\mu_4(k_+)$};
\end{tikzpicture}
\end{center}
\caption{Paths of the eigenvalues, $\mu_j$, of the truncated problem are shown as functions of $k \in [k_-,k_+]$. Those eigenvalues shown in green correspond to members of the box $\mc{E}$  defined by \eqref{eq:mathcalE} (shaded), while {the eigenvalue} in blue is not in $\mc{E}$.}\label{fig:box1}
\end{figure}
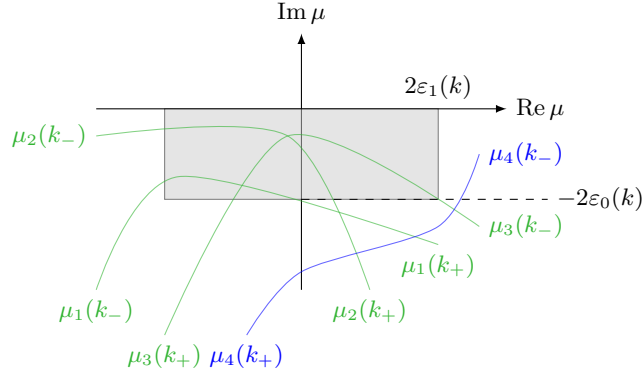

\begin{theorem}[From quasimodes to eigenvalues, with multiplicities]\label{thm:main2}
Let $k_j^-,k_j^+\to \infty$ such that there is $C>0$ satisfying $k_j^-\leq k_j^+\leq Ck_j^-$. Suppose there exists a family of quasimodes of quality $\QMC(k)\ll k^{-(5d+3)/2}$ and multiplicity $m_j$ in the window $[k_j^-,k_j^+]$ (in the sense of Definition \ref{def:quasimodes_mult}).
If $\QMC_0(k)$ is such that, for some $S>0$, 
\beqs
\QMC_0(k)\leq S k^{-(d+1)/2}\tfa k \quad\tand\quad \QMC_0(k) \gg k^{2d+1}\QMC(k) \tas k\tendi,
\eeqs
then there exists $k_0>0$ such that if $k_j^-\geq k_0$, 
$$
\Big|\mathcal{E}\Big((\klower)^{(d+1)/2}\QMC_0(\klower)\,,\,\QMC_0(\klower)\,,\,\klower\,,\,\kupper\Big)\Big|\geq m_j.
$$
\end{theorem}
Observe that if $\kupper = \klower$, then (up to algebraic powers of $k$) Theorem \ref{thm:main2} reduces to Theorem \ref{thm:main1}, except that now multiplicities are counted; therefore the ``quasimodes to eigenvalues'' result holds with multiplicities (just as the ``quasimodes to resonances'' result of \cite{St:99} includes multiplicities).

The ideas used in the proof of Theorems \ref{thm:main1} and \ref{thm:main2} are discussed in \S\ref{sec:idea} below.

\begin{remark*} The reason why both the constant $\alpha$ in Theorem \ref{thm:main1} and the exponent in the bound on the quality Theorem \ref{thm:main2} depend on $d$ is because the right-hand side of the bound \eqref{eq:bound1new} below on the solution operator of the truncated problem depends on $d$, which in turn comes from the fact that the trace-class norm of compactly-supported pseudodifferential operators depends on $d$.
\end{remark*}

\subsection{Numerical experiments illustrating the main results}\label{sec:num}

{
\paragraph{Description of the obstacles $\Oi$.}
In this section, $\Oi$ is one of the two ``horseshoe-shaped'' 2-d domains shown in Figure \ref{fig:geometries}. 
We define the \emph{small cavity} as the region between the two elliptic arcs 
\begin{align*}
&(\cos (t), 0.5 \sin (t)), \quad t\in [-\phi_0,\phi_0]
\quad\tand \quad 
(1.3\cos (t), 0.6 \sin (t)),\quad t\in [-\phi_1,\phi_1] \\
&\qquad\qquad\text{ with } \phi_0=7\pi/10\quad\tand \quad \phi_1= \arccos \left(\frac{1}{1.3}\cos (\phi_0)\right);
\end{align*}
this corresponds to the interior of the solid lines in Figure \ref{fig:geometries}.
We define the \emph{large cavity} as the region between the two arcs now with \(\phi_0=9\pi/10\).
(Note that our small cavity is the same as the cavity considered in the numerical experiments in \cite[Section IV]{BeChGrLaLi:11}.)
Recall that Theorems \ref{thm:main1} and \ref{thm:main2} require $\Gamma_D$ to be smooth, and thus these results do not strictly apply to the small and large cavities; however they do apply to smoothed versions of these.

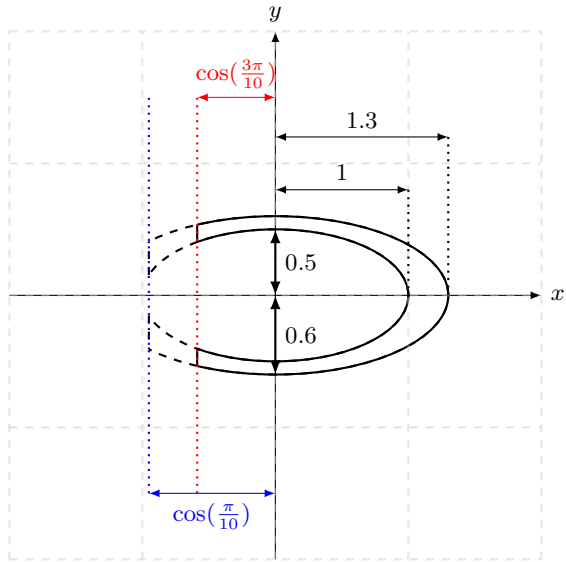
\begin{figure}
    \centering
    
    \begin{tikzpicture}
    \def \a{1.3};
    \def \b{.6};
    \begin{scope}[scale=1.75 ]
    \draw[thick,black] (0,0) ellipse (1 and .5);
    \draw[thick,black] (0,0) ellipse ({\a} and {\b});
	\fill[white] ({cos(126)},-2)-- ({cos(126)},2)-- (-2,2)--(-2,-2)--cycle;
    \draw[thick,black] ({cos(126)},{.5*sqrt(1-cos(126)^2)})--({cos(126)},{\b*sqrt(1-cos(126)^2/(\a)^2)});
    \draw[thick,black] ({cos(126)},-{.5*sqrt(1-cos(126)^2)})--({cos(126)},-{\b*sqrt(1-cos(126)^2/(\a)^2)});
    \draw[dashed,thick,black] (0,0) ellipse (1 and .5);
    \draw[dashed,thick,black] (0,0) ellipse ({\a} and {\b});
         \fill[white] ({cos(162)},-2)-- ({cos(162)},2)-- (-2,2)--(-2,-2)--cycle;
             \draw[dashed,thick,black] ({cos(162)},{.5*sqrt(1-cos(162)^2)})--({cos(162)},{\b*sqrt(1-cos(162)^2/(\a)^2)});
    \draw[dashed,thick,black] ({cos(162)},-{.5*sqrt(1-cos(162)^2)})--({cos(162)},-{\b*sqrt(1-cos(162)^2/(\a)^2)});

     \draw[blue,thick,dotted] ({cos(162)},-1.5)--({cos(162)},1.5);
     \draw[red,thick,dotted] ({cos(126)},-1.5)--({cos(126)},1.5);
     \draw [->,red] ({cos(126)/2},1.5)node[above]{$\cos(\tfrac{3\pi}{10})$}--({cos(126)},1.5);
     \draw [->,red] ({cos(126)/2},1.5)--(0,1.5);
     \draw [->,blue] ({cos(162)/2},-1.5)node[below]{$\cos(\tfrac{\pi}{10})$}--({cos(162)},-1.5);
     \draw [->,blue] ({cos(162)/2},-1.5)--(0,-1.5);
     \draw [->] (-2,0)--(2,0)node[right]{$x$};
    \draw [->] (0,-2)--(0,2)node[above]{$y$};
  \foreach \x in {-2,...,2}
    \draw[light-gray,opacity=.5,dashed,thick] (\x,-2)--(\x,2) (-2,\x)--(2,\x);
        \draw [->,thick] (0,.25)node[right]{$0.5$}--(0,.5);
     \draw [->,thick] (0,.25)--(0,0);
      \draw [->,thick] (0,-{\b/2})node[right]{$0\b$}--(0,-{\b});
     \draw [->,thick] (0,-{\b/2})--(0,0);
     \draw [->] ({\a/2},1.2)node[above]{$\a$}--(0,1.2);
     \draw [->]({\a/2},1.2)--(\a,1.2);
     \draw[dotted, thick] (\a,1.2)--(\a,0);
     \draw [->] (.5,.8)node[above]{$1$}--(0,.8);
     \draw [->](.5,.8)--(1,.8);
     \draw[dotted, thick] (1,.8)--(1,0);
    \end{scope}
    \end{tikzpicture}
    \caption{The two obstacles $\Oi$ considered in the numerical experiments}\label{fig:geometries}
\end{figure}

For both the small and large cavities, \(\Gamma_D\) coincides with the boundary of the ellipse \(E\) \eqref{eq:ellipse} with \(a_1=1\) and \(a_2=0.5\) in the neighbourhood of its minor axis. 
Part (i) of Lemma \ref{lem:specific} (i.e., the results of \cite{BeChGrLaLi:11}) then implies that there exist quasimodes with exponentially-small quality. 

We choose these particular $\Oi$ because we can compute the frequencies $k_\ell$ in the quasimode. Indeed, the functions $u_\ell$ in the quasimode construction in \cite{BeChGrLaLi:11} are based on 
the family of eigenfunctions of the ellipse localising around the periodic orbit $\{(0,x_2) : |x_2|\leq a_2\}$; 
when the eigenfunctions are sufficiently localised, the eigenfunctions multiplied by a suitable cut-off function form a quasimode, with frequencies $k_\ell$ equal to the square roots of eigenvalues of the ellipse. By separation of variables, $k_\ell$ can be expressed as the 
solution of a multiparametric spectral problem involving Mathieu functions; see see \cite[Appendix A]{BeChGrLaLi:11} and \cite[Appendix E]{MaGaSpSp:21}.

When giving specific values of $k_\ell$ below, we use the notation from \cite[Appendix A]{BeChGrLaLi:11} and \cite[Appendix E]{MaGaSpSp:21} that $k_{m,n}^e$ and $k_{m,n}^o$ are the frequencies associated with the eigenfunctions of the ellipse that are even/odd, respectively, in the angular variable, with $m$ zeros in the radial direction (other than at the centre or the boundary) and $n$ zeros in the angular variable in the interval $[0,\pi)$.

\begin{figure}[h!]
  \centering
  \includegraphics[width=0.99\textwidth]{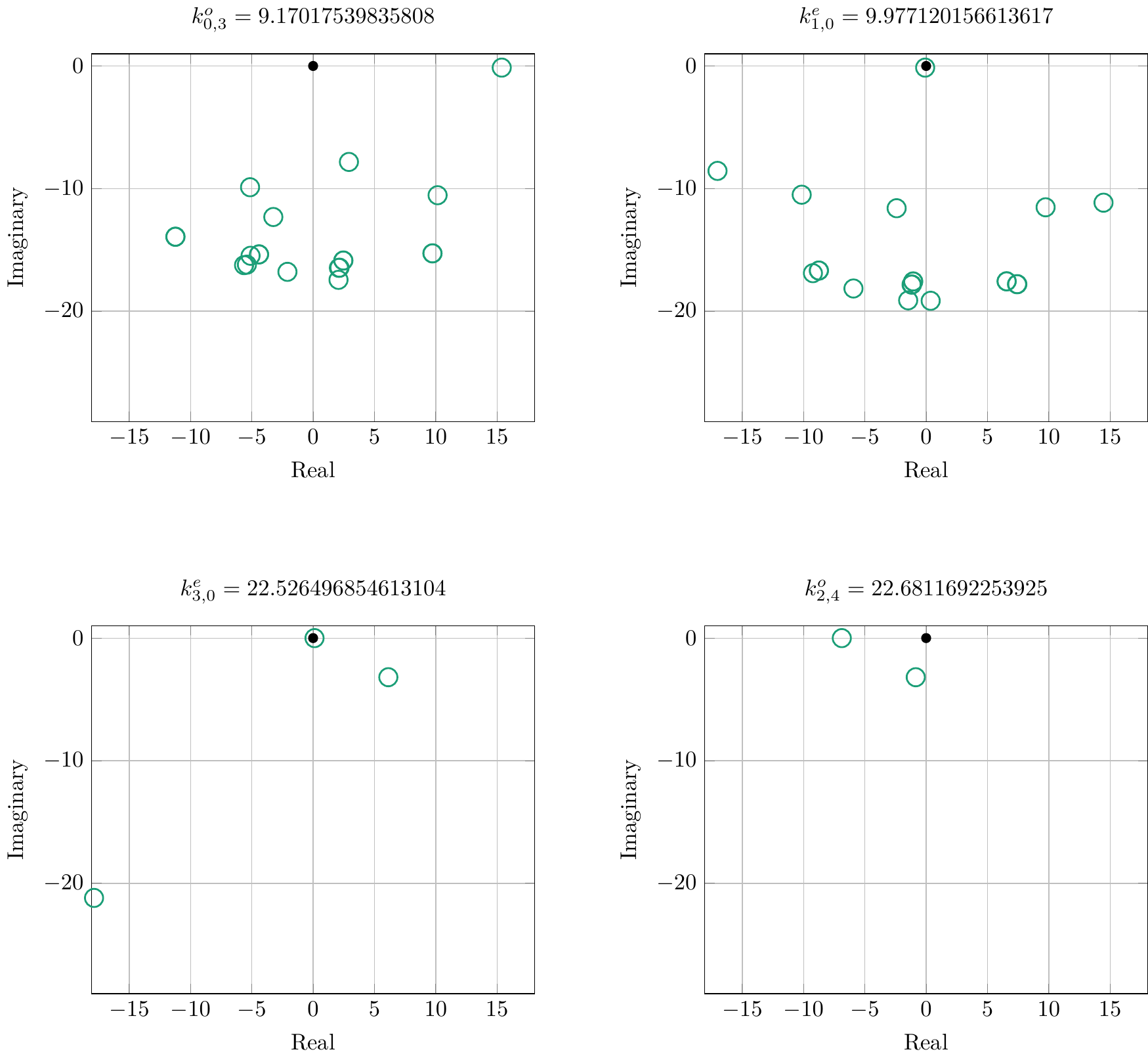}
  \caption{
  The eigenvalues of the truncated exterior Dirichlet problem (Definition~\ref{def:eigenvalue}) near the origin when $\Gamma_D$ is equal to the small cavity. The eigenvalues are plotted at several frequencies, $k$, corresponding to eigenvalues of the ellipse. 
 In each plot, the origin is marked with a black dot, and the eigenvalues are shown as green circles.}
 \label{fig:spectra_illustrations_elliptic_cavity_2D_dir}
\end{figure}

\begin{figure}[h!]
  \centering
  \includegraphics[width=0.99\textwidth]{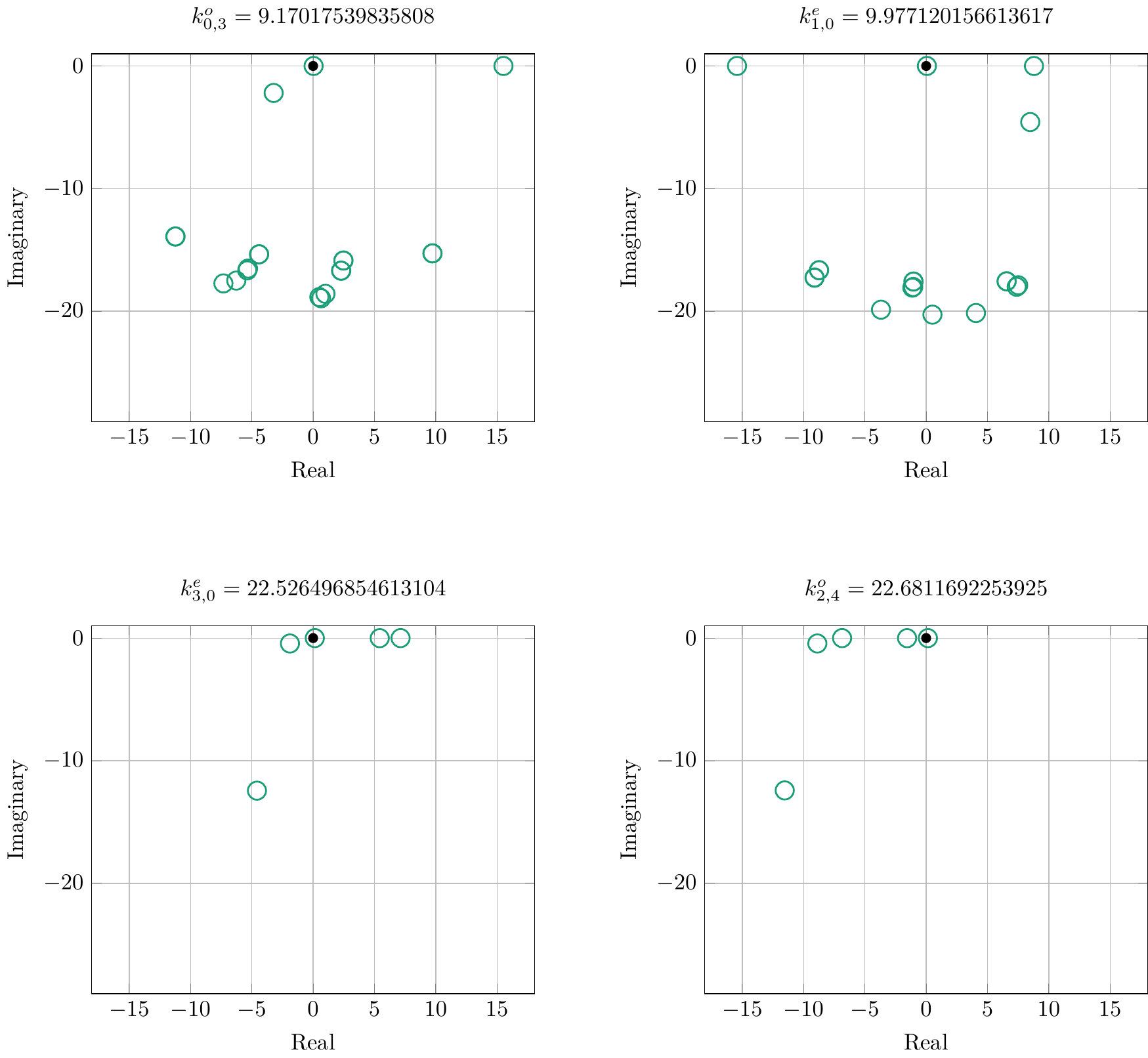}
  \caption{
   The eigenvalues of the truncated exterior Dirichlet problem (Definition~\ref{def:eigenvalue}) near the origin when $\Gamma_D$ is equal to the large cavity. The eigenvalues are plotted at several frequencies, $k$, corresponding to eigenvalues of the ellipse. 
 In each plot, the origin is marked with a black dot, and the eigenvalues are shown as green circles.}
\label{fig:spectra_illustrations_elliptic_cavity_bigger_2D_dir}
\end{figure}

\begin{figure}[h!]
  \begin{subfigure}[t]{0.45\textwidth}
      \includegraphics[width=\textwidth]{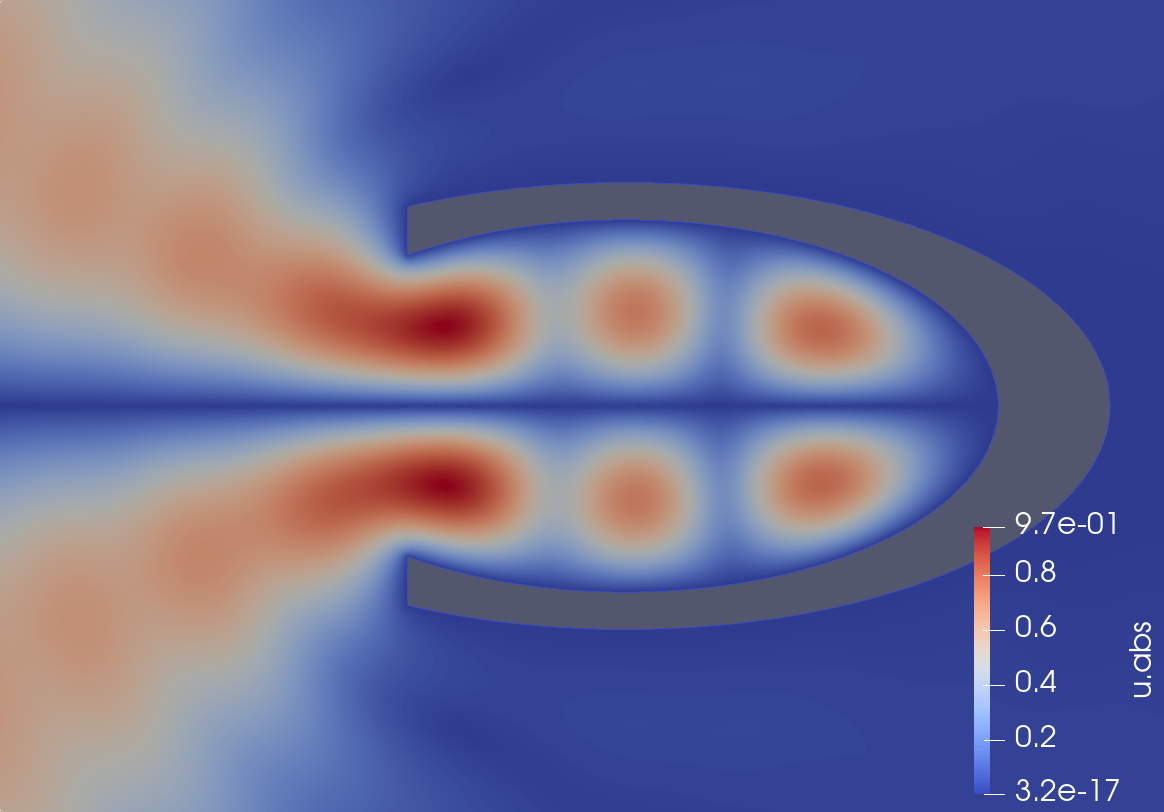}
      \subcaption{\(k^o_{0,3}=9.17017539835808\)}
  \end{subfigure}\hfill%
  \begin{subfigure}[t]{0.45\textwidth}
      \includegraphics[width=\textwidth]{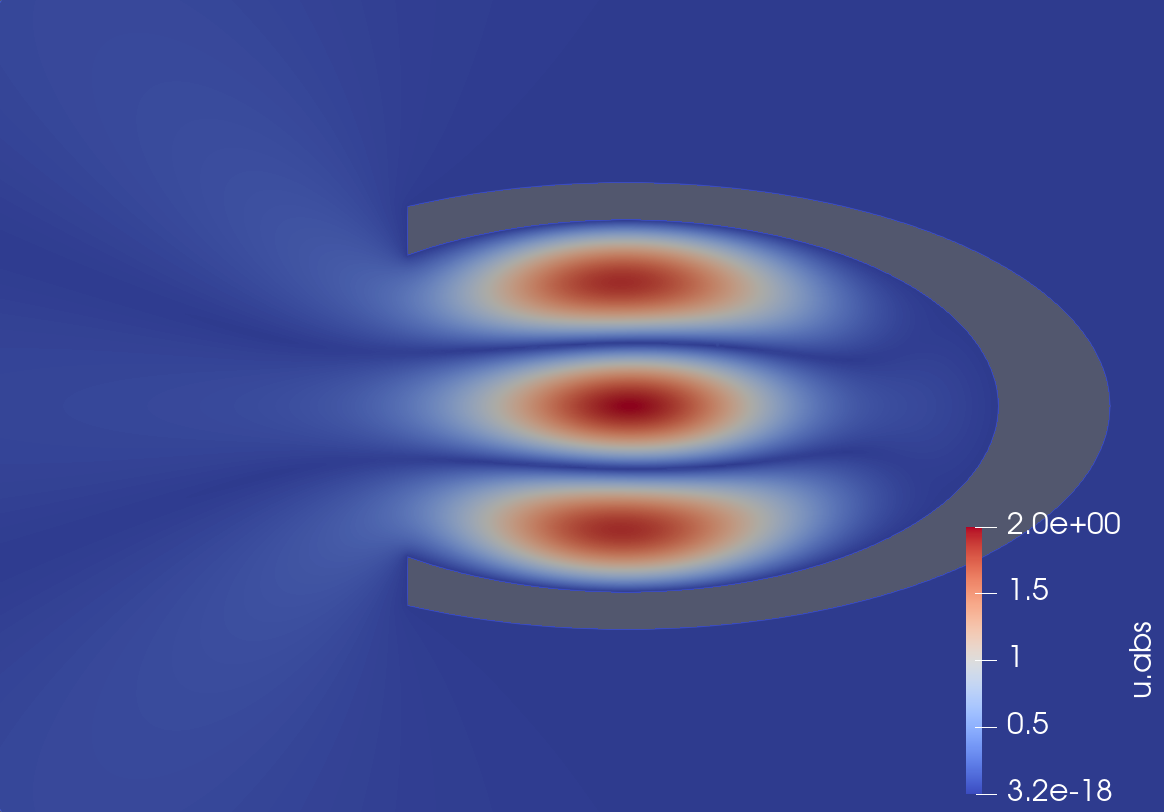}
      \subcaption{\(k^e_{1,0}=9.977120156613617\)}
  \end{subfigure}%
  \par\bigskip
  \begin{subfigure}[t]{0.45\textwidth}
      \includegraphics[width=\textwidth]{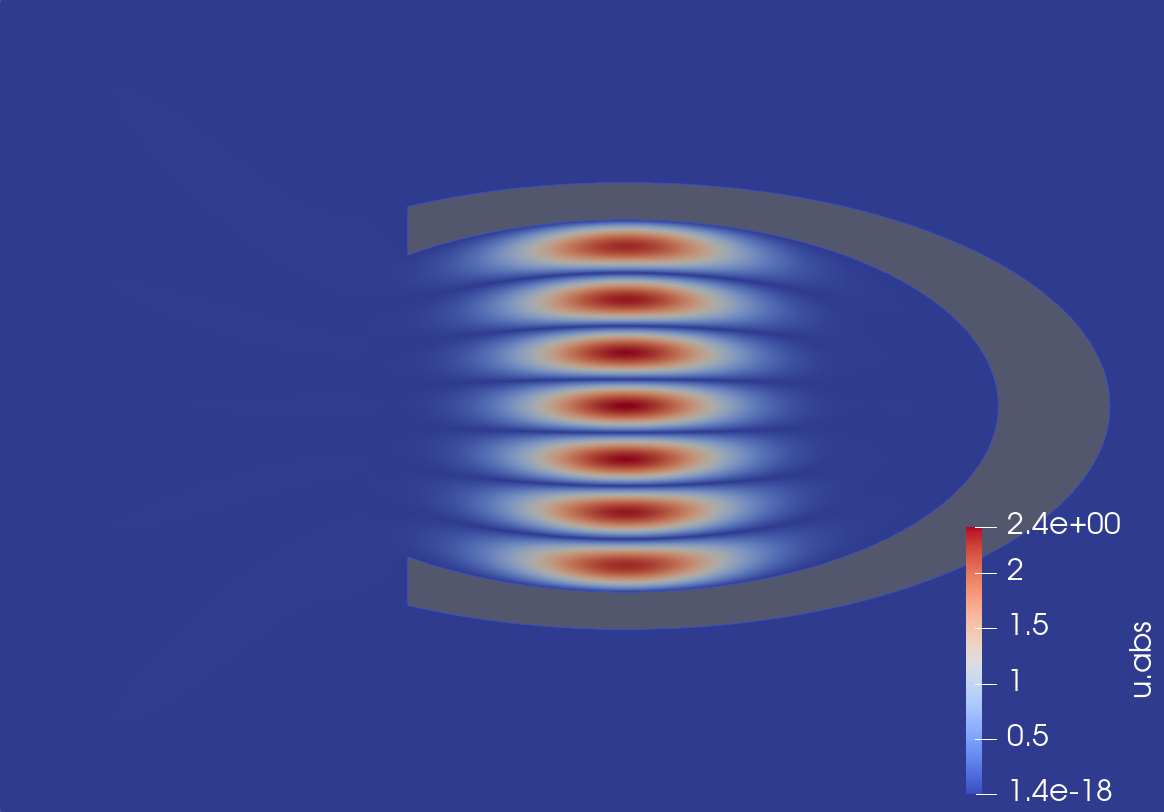}
        \subcaption{\(k^e_{3,0}=22.526496854613104\)}
  \end{subfigure}\hfill%
  \begin{subfigure}[t]{0.45\textwidth}
      \includegraphics[width=\textwidth]{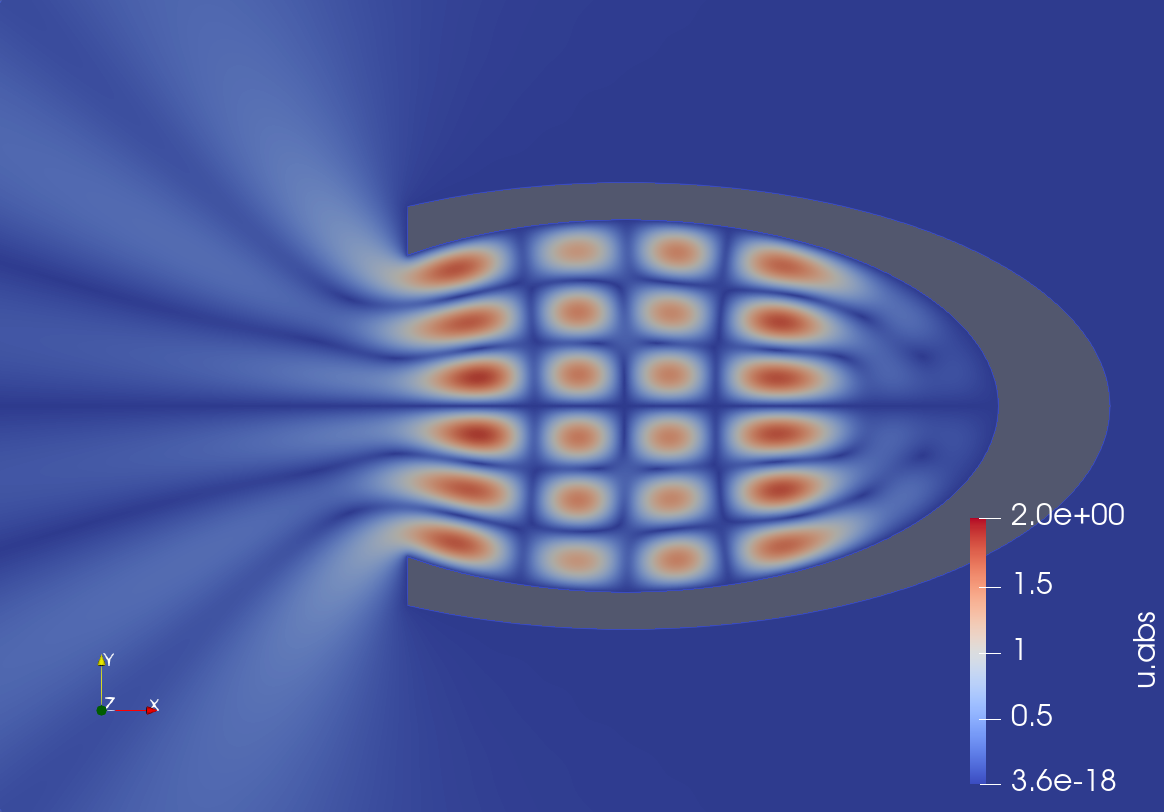}
      \subcaption{\(k^o_{2,4}=22.6811692253925\)}
  \end{subfigure}%
  \caption{Absolute value of the eigenfunction of the truncated exterior Dirichlet problem associated with the smallest eigenvalue 
  the small cavity.} 
 \label{fig:eig_illustrations_elliptic_cavity_2D_dir}
\end{figure}

\begin{figure}[h!]
  \begin{subfigure}[t]{0.45\textwidth}
      \includegraphics[width=\textwidth]{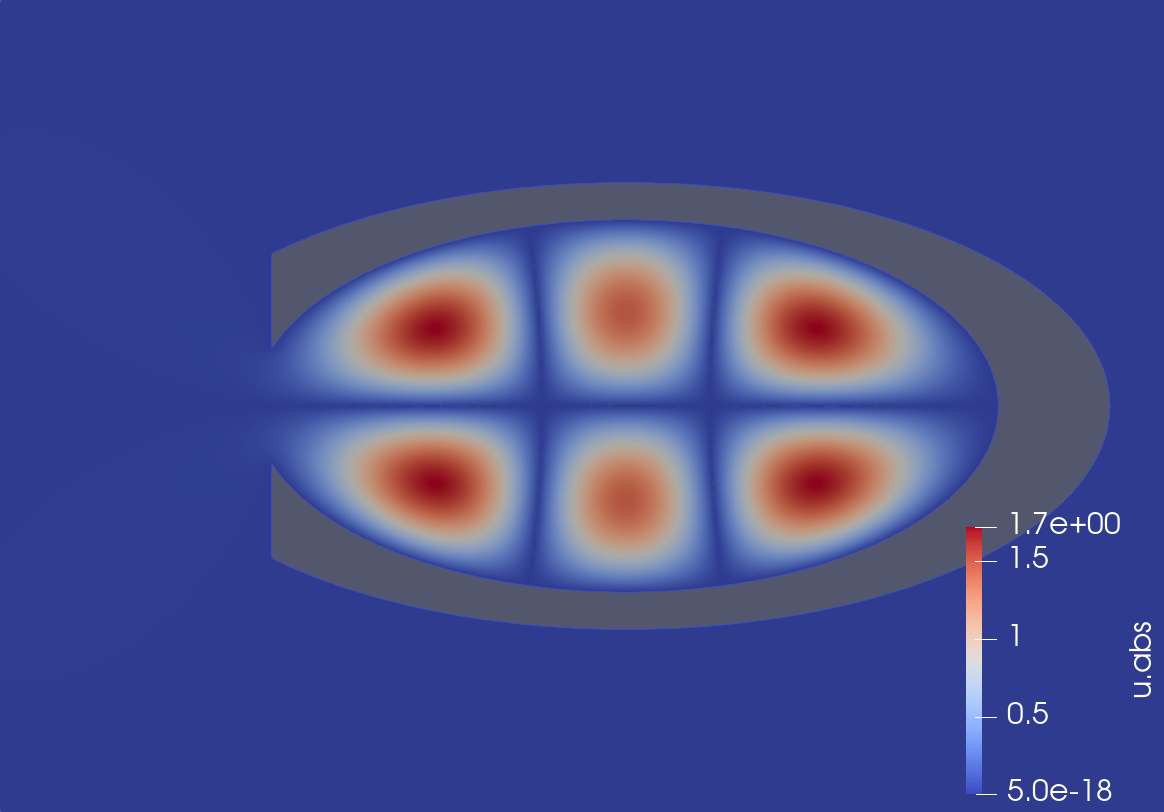}
      \subcaption{\(k^o_{0,3}=9.17017539835808\)}
  \end{subfigure}\hfill%
  \begin{subfigure}[t]{0.45\textwidth}
      \includegraphics[width=\textwidth]{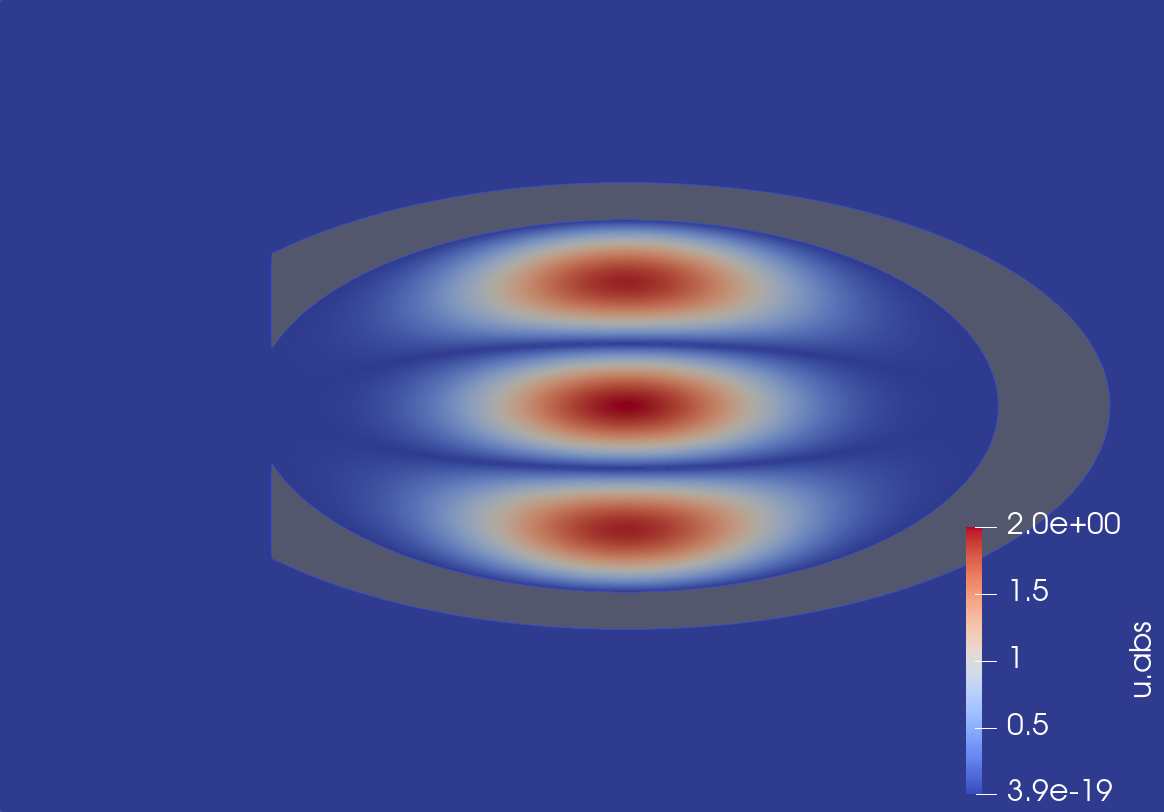}
      \subcaption{\(k^e_{1,0}=9.977120156613617\)}
  \end{subfigure}%
  \par\bigskip
  \begin{subfigure}[t]{0.45\textwidth}
      \includegraphics[width=\textwidth]{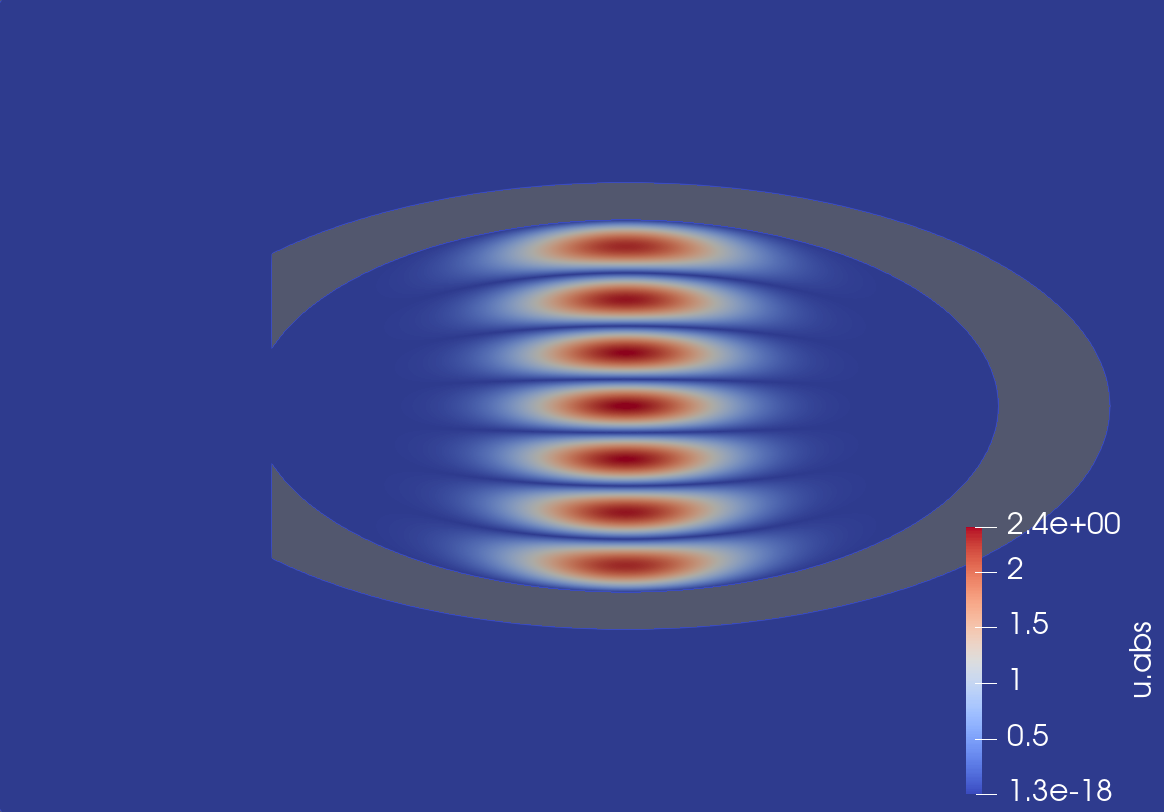}
      
      \subcaption{\(k^e_{3,0}=22.526496854613104\)}
  \end{subfigure}\hfill%
  \begin{subfigure}[t]{0.45\textwidth}
      \includegraphics[width=\textwidth]{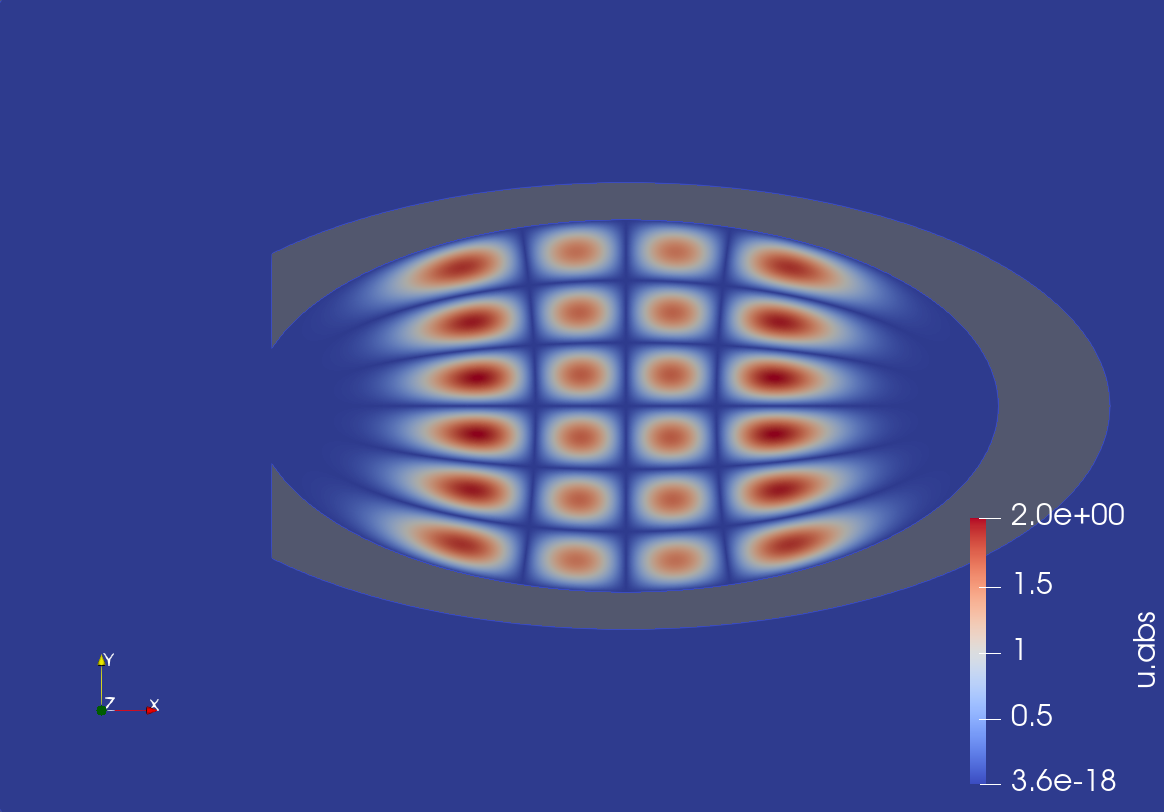}
      \subcaption{\(k^o_{2,4}=22.6811692253925\)}
  \end{subfigure}%
  \caption{Absolute value of the eigenfunction of the truncated exterior Dirichlet problem 
  associated with the smallest eigenvalue for the large cavity.
  }\label{fig:eig_illustrations_elliptic_cavity_bigger_2D_dir}
\end{figure}

\begin{figure}[h!]
  \centering
  \begin{minipage}{0.49\textwidth}
    \centering
    \includegraphics[width=\textwidth]{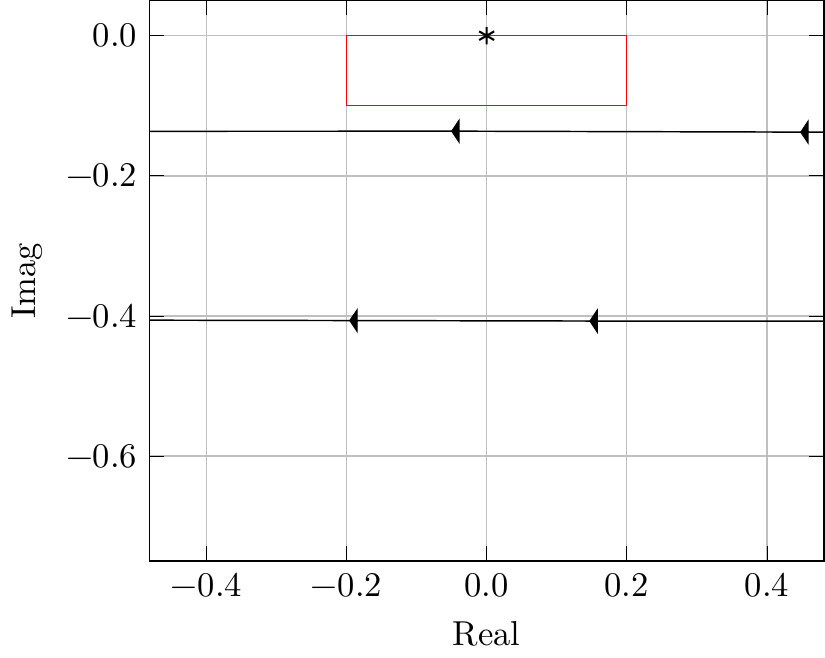}
  \end{minipage}%
  \hfill
  \begin{minipage}{0.49\textwidth}
    \centering
    \includegraphics[width=\textwidth]{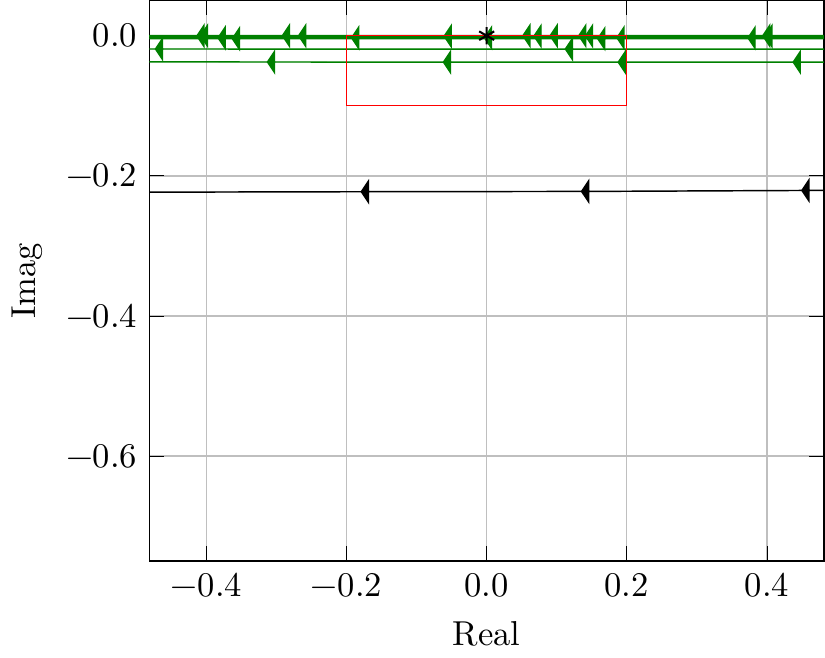}
  \end{minipage}
  \caption{Paths of the eigenvalues for \(k\in (2.5,12.5)\) for the small cavity (left) and the large cavity (right). 
  The eigenvalues that enter the red rectangle 
  are coloured green.}
  \label{fig:flow}
\end{figure}

\paragraph{Plots of the eigenvalues and eigenfunctions}

Figures~\ref{fig:spectra_illustrations_elliptic_cavity_2D_dir} and~\ref{fig:spectra_illustrations_elliptic_cavity_bigger_2D_dir} plot
the near-zero eigenvalues of the truncated exterior Dirichlet problem for the small and large cavity, respectively, at frequencies corresponding to eigenvalues of the ellipse. Figures~\ref{fig:eig_illustrations_elliptic_cavity_2D_dir} and~\ref{fig:eig_illustrations_elliptic_cavity_bigger_2D_dir} plot the corresponding eigenfunctions.
In all these figures $\Gtr = \partial B(0,2)$.

Figure \ref{fig:spectra_illustrations_elliptic_cavity_bigger_2D_dir} shows that the large cavity has an eigenvalue very close to zero at each of the four frequencies considered, qualitatively illustrating Theorem \ref{thm:main1}.
In contrast, Figure~\ref{fig:spectra_illustrations_elliptic_cavity_2D_dir} shows that the small cavity only has
an eigenvalue very close to zero at the frequencies $k_{1,0}^e$ and $k_{3,0}^e$ (top right and bottom left in the figures) and not at 
$k_{0,3}^e$ and $k_{2,4}^o$ (top left and bottom right). The reason for this is clear from the plots of the eigenfunctions
of the truncated exterior Dirichlet problem: 
looking at Figure~\ref{fig:eig_illustrations_elliptic_cavity_2D_dir}, we see that at $k_{0,3}^e$ and $k_{2,4}^o$ the eigenfunctions are not well localised around the minor axis of the ellipse to be inside the small cavity -- in the top left and bottom right of Figure \ref{fig:eig_illustrations_elliptic_cavity_2D_dir} we see them ``leaking out'' of the small cavity. However, 
looking at Figure~\ref{fig:eig_illustrations_elliptic_cavity_bigger_2D_dir}, we see that the corresponding 
eigenfunctions are localised sufficiently to be inside the large cavity, and thus generate an eigenvalue very close to zero. 
In these plots, the eigenfunctions are normalised so that their $L^2(\Otr)$ norm equals one.

Figure \ref{fig:flow} plots the trajectories of the near-zero eigenvalues as functions of $k$ for 
both the small cavity (left plot) and large cavity (right plot) for $k\in (2.5,12.5)$, with the spectra computed every $0.025$.
For Figure \ref{fig:flow}, $\Gtr=\partial B(0,1.5)$; this change (compared to $\Gtr=\partial B(0,2)$ for the earlier figures) is to reduce the cost of each eigenvalue solve, because each of the two plots in Figure \ref{fig:flow} requires $400$ such solves. 
Since we use the exact  (up to discretisation error) Dirichlet-to-Neumann map on $\Gtr$, we expect there to be no difference between choosing $\Gtr=\partial B(0,1.5)$ and $\Gtr=\partial B(0,2)$ (in particular 
Figures~\ref{fig:spectra_illustrations_elliptic_cavity_2D_dir} and~\ref{fig:spectra_illustrations_elliptic_cavity_bigger_2D_dir} are unchanged when $\Gtr$ is changed from $\partial B(0,2)$ to $\partial B(0,1.5)$).

The eigenvalues that enter the red rectangle in Figure \ref{fig:flow} are coloured green; these are members of $\mathcal{E}(0.2,0.05, 2.5, 12.5)$, where $\mathcal{E}$ is defined by \eqref{eq:mathcalE}.
Similar to the eigenvalues plots 
in Figures \ref{fig:spectra_illustrations_elliptic_cavity_2D_dir} and 
\ref{fig:spectra_illustrations_elliptic_cavity_bigger_2D_dir},
Figure \ref{fig:flow} shows that the large cavity has more near-zero eigenvalues for the range of $k$ considered than the small 
cavity. This is expected since a larger number of the eigenfunctions of the ellipse are localized in the large cavity than in the small cavity.

\paragraph{How the eigenvalues and eigenfunctions were computed.}

Definition \ref{def:eigenvalue} (of the eigenvalues of the truncated Dirichlet problem) implies that if $\mu_\ell$ is an eigenvalue at frequency $k_\ell$, and with corresponding eigenfunction $u_\ell$, then 
\beq\label{eq:eigenfunction}
a(u_\ell,v)= \mu_\ell (u_\ell,v)_{L^2(\Otr)} \quad\tfa v \in H^1_{0,D}(\Otr),
\eeq
where the sesquilinear form $a(\cdot,\cdot)$ is that appearing in the standard variational (i.e.~weak) formulation of the Helmholtz exterior Dirichlet problem. 

\begin{definition}[Variational formulation of Helmholtz exterior Dirichlet problem]\label{def:EDP}
Given $k>0$, $\Omega_-$ as above, and $F\in (H_{0,D}^1(\Otr))^*$,  let $u\in H_{0,D}^1(\Otr)$ be the solution of the variational problem
\beq\label{eq:vf}
\text{ find } u \in H_{0,D}^1(\Otr) \quad\tst \quad a(u,v)=F(v)
\quad \tfa v\in H_{0,D}^1(\Otr),
\eeq
where
\begin{align}\label{eq:sesqui}
a(u,v)&:= \int_{\Otr} 
\Big(\nabla u\cdot\overline{\nabla v}
 - k^2  u\overline{v}\Big) - \big\langle \DtN(k) (\gamma_0^{\tr} u),\gamma_0^{\tr} v\big\rangle_{\Gtr},
\end{align}
where $\langle\cdot,\cdot\rangle_{\Gtr}$ denotes the duality pairing on $\Gtr$ that is linear in the first argument and antilinear in the second.
\end{definition}

The figures above were created by solving the eigenvalue problem \eqref{eq:eigenfunction} using the finite-element method with continuous piecewise-linear elements (i.e.~the polynomial degree, $p$, equals one) and meshwidth $h$, equal $(2\pi/30) k^{-3/2}$.
The Dirichlet-to-Neumann map, $\DtN(k)$, in $a(\cdot,\cdot)$ was computed using boundary integral equations -- see Appendix \ref{sec:DtN} for details. 
The accuracy, uniform in frequency, of the finite-element applied the variational problem \eqref{eq:vf} with $p=1$ and $hk^{3/2}$ 
 sufficiently small has been known empirically for a long time, and was recently proved in \cite{LaSpWu:19} for the case when the Dirichlet-to-Neumann map is realised exactly.

Since computing the Dirichlet-to-Neumann map is relatively expensive, in practice one often approximates it using a perfectly-matched layer (PML) or an absorbing boundary condition (such as the impedance boundary condition). 
The plots of the eigenfunctions and near-zero eigenvalues 
of the corresponding truncated exterior Dirichlet problems 
are very similar to those above; this too is expected since the quasimode is supported in a neighbourhood of the obstacle.
}

\subsection{Implications of the main results for numerical analysis of the Helmholtz exterior Dirichlet problem}\label{sec:num2}

Theorems \ref{thm:main1} and \ref{thm:main2} are the first step towards rigorously understanding how 
iterative solvers such as the generalised minimum residual method (GMRES) behave when applied to discretisations of high-frequency Helmholtz problems under strong trapping (the subject of the companion paper \cite{MaGaSpSp:21}).
We now explain this in more detail.

As we saw in \eqref{eq:eigenfunction}, 
the eigenvalues of truncated exterior Dirichlet problem (in sense of Definition \ref{def:eigenvalue}) correspond to eigenvalues of sesquilinear form of standard variational formulation (Definition \ref{def:EDP}).
The standard variational formulation is the basis of the finite-element method for computing approximations to the solution of the variational problem \eqref{eq:vf}. Indeed, the finite-element method consists of choosing a piecewise-polynomial subspace of $H^1_{0,D}(\Otr)$ and solving the variational problem \eqref{eq:vf} in this subspace.

A very popular way of solving the linear systems resulting from the finite-element method applied to the Helmholtz scattering problems is via iterative solvers such as GMRES 
\cite{SaSc:86}; this choice is made because the linear systems are (i) large 
and (ii) non-self-adjoint. Regarding (i): the systems are large since the number of degrees of freedom must be $\gtrsim k^d$ to resolve the oscillations in the solution, see, e.g., the literature review in \cite[\S1.1]{LaSpWu:19}.
Regarding (ii): non-self-adjointness of the linear systems arises directly from the non-self-adjointness of the underlying Helmholtz scattering problem; GMRES is applicable to such systems, unlike the conjugate gradient method.

There is currently large research interest in understanding how iterative methods behave when applied to Helmholtz linear systems, and in designing good preconditioners for these linear systems; see the literature reviews \cite{Er:08,ErGa:12,GaZh:19}, \cite[\S1.3]{GrSpZo:20}.

The location of eigenvalues, especially near-zero ones, is crucial in understanding 
the behaviour of iterative methods.
In the Helmholtz context, eigenvalue analyses of iterative methods applied to nontrapping problems include, for finite-element discretisations, \cite{ElOl:99,ElErOl:01,ErVuOo:04,VaErVu:07,ErGa:12,ViGe:14,CoGa:17,LiXiSaDe:20}, and, for boundary-element discretisations, \cite{ChHa:01,DaDaLa:13,ChDaLe:17}.

The paper \cite{MaGaSpSp:21} analyses GMRES applied to discretisations of Helmholtz problems with strong trapping, 
using the ``cluster plus outliers'' GMRES convergence theory from \cite{CaIpKeMe:96} (with this idea arising in the context of the conjugate gradient method \cite{Je:77} and used subsequently in, e.g., \cite{ElSiWa:02}).
The paper \cite{MaGaSpSp:21} obtains bounds on how the number of GMRES iterations depends on the frequency, under various assumptions about the eigenvalues. In particular, Theorem \ref{thm:main1} 
rigorously justifies \cite[Observation O2(b)]{MaGaSpSp:21} for the standard variational formulation of the truncated exterior Dirichlet problem.
We highlight that, although the results in \cite{MaGaSpSp:21} are about unpreconditioned systems, they give insight into the design of preconditioners. Indeed, a successful preconditioner for Helmholtz problems with strong trapping will need to 
specifically deal with the near-zero eigenvalues created by trapping. 
Theorem \ref{thm:main1} and \ref{thm:main2} give information about the location and multiplicities of these eigenvalues, and \cite{MaGaSpSp:21} shows how these locations and multiplicities affect GMRES.

\subsection{The ideas behind the proof of Theorem \ref{thm:main1}}\label{sec:idea}

\paragraph{Semiclassical notation.}
Instead of working with the parameter $k$ and being interested in the
large-$k$ limit, the semiclassical literature usually works with a
parameter $h:= k^{-1}$ and is interested in the small-$h$ limit. So
that we can easily recall results from this literature, we also work
with the small parameter $k^{-1}$, but to avoid a notational clash
with the meshwidth of the FEM, we let $\h:= k^{-1}$ (the notation
$\h$ comes from the fact that the semiclassical parameter is sometimes related
to Planck's constant, which is written as $2\pi \h$; see, e.g., \cite[\S1.2]{Zw:12}).
Theorem \ref{thm:main1} is then restated in semiclassical notation as Theorem \ref{thm:main1h} below.

\paragraph{The solution operator of the truncated problem.}
Let $R_{\Otr}(\lambda,z):L^2(\Omega_{\tr})\to L^2(\Omega_{\tr})$ be the solution operator for the truncated problem
\begin{equation}
\label{e:prob}
\begin{cases}
(-\h^2\Delta-\lambda^2 -{z})u=f\quad\tin\Omega_{\tr}\\
\gamma_0^D u=0,\\
\gamma_1^\tr u=
\DtN(\lambda/\h)\gamma_0^\tr u; 
\end{cases}
\end{equation}
that is, $R_{\Otr}(\lambda,z)$ satisfies
\beqs
\begin{cases}
(-\h^2\Delta-\lambda^2 -z)R_{\Otr}(\lambda,z)f=f\quad\tin\Omega_{\tr}\\
\gamma_0^D R_{\Otr}(\lambda,z)f=0\\
\gamma_1^\tr
R_{\Otr}(\lambda,z)f=
\DtN(\lambda/\h)\gamma_0^DR_{\Otr}(\lambda,z) f.
\end{cases}
\eeqs
Note that, at this point, it is not clear that the problem~\eqref{e:prob} is well posed and that the family of operators $R_{\Otr}(\lambda,z)$ is well defined. We address this  in Lemma~\ref{l:inverseForm} below.

We study $R_{\Otr}(\lambda,z)$ by relating it to the solution operator of a more-standard scattering problem. Namely, let $V\in L^\infty(\Omega_+)$ with $\supp V \Subset \Rea^d$, 
and consider the problem
\begin{equation}
\label{e:probClassical}
\begin{cases}
(-\h^2\Delta-\lambda^2 +V)u=f \quad\ton\Omega_+,\\
\gamma_0^D u=0, \\
u\text{ is }\lambda/\h\text{ outgoing}.
\end{cases}
\end{equation}
By, e.g., \cite[Chapter 4]{DyZw:19}, the inverse of~\eqref{e:probClassical} is a meromorphic family of operators (for $\lambda\in \mathbb{C}$ when $d$ is odd or $\lambda$ in the logarithmic cover of $\mathbb{C}\setminus \{0\}$ when $d$ is even) $R_{V}(\lambda) :L^2_{\comp}(\Omega_+)\to L^2_{\loc}(\Omega_+)$ with finite-rank poles satisfying
\begin{equation}
\label{e:probClassical2}
\begin{cases}
(-\h^2\Delta-\lambda^2 +V)R_{V}(\lambda) f=f\quad\text{ in }\Omega\\
\gamma_0^D R_{V}(\lambda) f=0, \\
R_{V}(\lambda) f\text{ is }\lambda/\h\text{ outgoing}.
\end{cases}
\end{equation}
Observe that, although both $R_{\Otr}(\lambda,z)$ and $R_{V}(\lambda) $ depend on $\h$, we omit this dependence in the notation to keep expressions compact.

The following two lemmas (proved in \S\ref{sec:proofssolution}) relate $R_{\Otr}(\lambda,z)$ and $R_{V}(\lambda)$ and then characterise the eigenvalues of the truncated exterior Dirichlet problem as poles of $R_{\Otr}(\lambda,z)$ as a function of $z$. 

We use three indicator functions: $1_\Otr$ denotes the function in $L^\infty(\Omega_+)$ that is one on $\Otr$ and zero otherwise,
$\indicatorR$ denotes the restriction operator $L^2(\Omega_+)\rightarrow L^2(\Otr)$, and $\indicatorE$ denotes the extension-by-zero operator $L^2(\Otr)\rightarrow L^2(\Omega_+)$.

\begin{lemma}
\label{l:inverseForm}
Define
\beq\label{eq:Rdef}
R(\lambda,z):= R_V(\lambda) \quad\text{ with } V(z) = -z 1_{\Otr}.
\eeq
Then
\beq\label{eq:twoRs}
R_{\Otr}(\lambda,z)=\indicatorR
 R(\lambda,z)
\indicatorE ,
\eeq
and thus $R_{\Otr}(\lambda,z)$ is a meromorphic family of operators in $\lambda$ for $\lambda\in \mathbb{C}$ when $d$ is odd and $\lambda$ in the logarithmic cover of $\mathbb{C}\setminus \{0\}$ when $d$ is even.
\end{lemma}

\begin{lemma}\label{lem:meromorphic}
For $\lambda \in \Rea\setminus\{0\}$, $z\mapsto  R(\lambda,z)$ is a meromorphic family of operators $L^2_{\comp}(\Omega_+)\to L^2_{\loc}(\Omega_+)$ with finite rank poles.
\end{lemma}

\begin{corollary}
If $z_j$ is a pole of $z\mapsto  R_{\Otr}(1,z)$, then $\mu_\ell:= -\h_j^{-2} z_j$ is an eigenvalue of the truncated exterior Dirichlet problem (in the sense of Definition \ref{def:eigenvalue}).
\end{corollary}

The key point is that we are interested in $R_{\Otr}(\lambda,z)$ as a meromorphic family in the variable $z$, in contrast to the more-familiar study of $R_{V}(\lambda) $ as a meromorphic family in the variable $\lambda$.

\paragraph{Recap of ``from quasimodes to resonances''.}

Recall that resonances of $-\h^2 \Delta +V$ are defined as poles of the meromorphic continuation of $R_V(w)$ into $\Im w <0$, see \cite[\S4.2, \S7.2]{DyZw:19}. The ``quasimodes to resonances'' argument of \cite{TaZw:98} (following \cite{StVo:95, StVo:96}; see also \cite[Theorem 7.6]{DyZw:19})
shows that existence of quasimodes (as in Definition \ref{def:quasimodes}) implies existence of resonances close to the real axis;
the additional arguments in \cite{St:99} then prove the corresponding result with multiplicities.

These arguments use the \emph{semiclassical maximum principle} (a consequence of the maximum principle of complex analysis, see Theorem \ref{thm:scmp} below) combined with the bounds
\beq\label{eq:bound1}
\N{\chi R_V(\lambda)\chi}_{L^2\rightarrow L^2} \leq C \exp\Big(C\h^{-d}\log \delta^{-1}\Big),\qquad \lambda^2\in \Omega
\,\,\Big\backslash
 \bigcup_{w \in {\rm Res}
(-\h^2 \Delta +V)}B(w,\delta),
\eeq
for $\Omega\Subset \{\Re w >0\}$,
and
\beq\label{eq:bound2}
\N{R_V(\lambda)}_{L^2\rightarrow L^2} \leq \frac{1}{
\Im (\lambda^2)}
 \quad\tfor \Im( \lambda^2 )>0;
\eeq
see \cite[Lemma 1]{TaZw:98},  \cite[Proposition 4.3]{TaZw:00}, \cite[Theorem 7.5]{DyZw:19}.

\paragraph{From quasimodes to eigenvalues.}

Theorems \ref{thm:main1} and \ref{thm:main2} are proved using the same ideas as in the quasimodes to resonances arguments, except that now we work in the complex $z$-plane (with real $\lambda$) instead of the complex $\lambda$-plane.
The analogue of the bounds \eqref{eq:bound1} and \eqref{eq:bound2} are given in the following lemma.

\begin{lemma}[Bounds on $R_{\Otr}(\lambda,z)$]\label{l:mainEstimate}
Let $0<a<b$ and let $z_j(\h,\lambda)$ be the poles of $R_{\Otr}(\lambda,z)$ (as a meromorphic function of $z$). Then there exist $C_1,\e_1>0$ such that for all $0<\h<1$, $\lambda^2\in[a,b]$ and $\delta>0$,
\beq\label{eq:bound1new}
\|R_{\Otr}(\lambda,z)\|_{L^2(\Otr)\to L^2(\Otr)}\leq \exp\Big(C_1\h^{-d}\log \delta^{-1}\Big)\quad\tfor z\in B(0,\e_1 \h)
\Big\backslash \bigcup_{j}B(z_j(\h,\lambda),\delta).
\eeq
Furthermore, there exists $C_2>0$ such that
\beq\label{eq:bound2new}
\|R_{\Otr}(\lambda,z)\|_{L^2(\Otr)\to L^2(\Otr)}\leq C_2\frac{ \langle z\rangle}{\Im z}
\quad\tfor  \Im z>0,
\eeq
where $\langle z\rangle:= (1+|z|^2)^{1/2}$.
\end{lemma}

The bound \eqref{eq:bound1new} is proved by finding a parametrix for 
$-\h^2\Delta-\lambda^2 -z 1_{\Otr}$ (i.e.~an approximation to $R_\Otr(\lambda,z)$) via a boundary complex absorbing potential. While parametrices based on complex absorption are often used in scattering theory (see, e.g.,~\cite{DyZw:16, DyGa:17}~\cite[Theorem 7.4]{DyZw:19}), parametrices based on boundary complex absorption appear to be new in the literature.  One of the main features of the argument below is that it relies on a comparison of the (in principle, trapping) billiard flow with the non-trapping free flow to obtain estimates on the parametrix. A similar argument should work for boundaries in \emph{any} non-trapping background.

We also highlight that, while we consider the scattering by Dirichlet obstacles in this paper and therefore must use boundary complex absorption,  smooth compactly-supported perturbations of $-\Delta $, e.g.~metric perturbations or semiclassical Schr\"odinger operators, can be handled similarly. Indeed, for these problems, the parametrix based on boundary absorption could be replaced by one based on simpler complex absorbing potentials.

\subsection{Outline of the rest of the paper}

In \S\ref{sec:prelim} we prove Lemmas \ref{l:inverseForm} and \ref{lem:meromorphic} and then collect preliminary results about the generalized bicharacteristic flow (\S\ref{sec:flow}), the geometry of trapping (\S\ref{sec:trapping}), complex scaling (\S\ref{sec:complexscaling}), and defect measures (\S\ref{sec:defect}).
In \S\ref{sec:parametrix} we find a parametrix for $R_\Otr(\lambda,z)$ via a boundary complex absorbing potential. 
In \S\ref{sec:proofbounds} we prove Lemma \ref{l:mainEstimate}.
In \S\ref{sec:mainproof} we prove Theorems \ref{thm:main1} and \ref{thm:main2} using Lemma \ref{l:mainEstimate} and the semiclassical maximum principle.

\section{Preliminary results}\label{sec:prelim}

\subsection{Restatement of Theorems \ref{thm:main1} and \ref{thm:main2} in semiclassical notation}

\begin{definition}[Quasimodes in $\h$ notation]\label{def:quasimodesh}
A \emph{family of quasimodes of quality $\eps(\h)$}
is a sequence $\{(u_\ell,\h_\ell)\}_{\ell=1}^\infty\subset H^2(\Otr)\cap H_{0,D}^1(\Otr)\times \mathbb{R}$ such that  $\h_\ell\tendo$ as $\ell \tendi$ and there is a compact subset $\mc{K}\Subset \Omega_1$ such that, for all $\ell$, $\supp\, u_\ell \subset \mc{K}$,
\beqs
\N{(-\h^2\Delta -1) u_\ell}_{L^2(\Otr)} \leq \eps(\h_\ell) \quad\tand\quad\N{u_\ell}_{L^2(\Otr)}=1.
\eeqs
\end{definition}

Let 
\beq\label{eq:epsepsilon}
\eps(\h) := \h^2 \QMC(\h^{-1}).
\eeq
Remark \ref{r:lowerBoundQuasi} implies that we can assume that there exist $\widetilde{S}_1,\widetilde{S}_2>0$ such that
\beq\label{eq:lowerboundBurq}
\eps(\h) \geq \widetilde{S}_1 \exp( - \widetilde{S}_2/\h).
\eeq
Theorem \ref{thm:main1} is then equivalent to the following result in the sense that the following result holds if and only if Theorem \ref{thm:main1} holds with $\mu_\ell := \h_\ell^{-2}{z_\ell}$.

\begin{theorem}[Analogue of Theorem \ref{thm:main1} in $\h$ notation]\label{thm:main1h}
Let $\alpha> 3(d+1)/2$.
Suppose there exists a family of quasimodes {in the sense of Definition \ref{def:quasimodesh}} 
such that the quality $\eps(\h)$ satisfies
\beq\label{eq:epslowerbound}
\eps(\h) \ll \h^{1+\alpha}.
\eeq
Then there exists $\h_0>0$ (depending on $\alpha$) such that, if $\ell$ is such that $\h_\ell\leq \h_0$ then there exists $z_\ell\in\mathbb{C}$ and $0\neq u_\ell \in H_{0,D}^1(\Otr)$ with 
\begin{equation}
\label{e:semiEigenvalue}
(-\h_\ell^2\Delta-1 +{z_\ell})u_\ell=0\tin\Omega_{\tr},  \quad
\gamma_1^{\tr} u_\ell=\DtN(\h_\ell^{-1})(\gamma_0^{\tr} u_\ell),
\quad\tand\quad|{z_\ell}| \leq \h_\ell^{-\alpha} \eps(\h_\ell).
\end{equation}
\end{theorem}

\begin{definition}[Quasimodes with multiplicity in $\h$ notation]\label{def:quasimodes_multh}
Let $0\leq a(\h)\leq b(\h)<\infty$ be two functions of $\h$.
A \emph{family of quasimodes of quality $\eps(\h)$ and multiplicity $m(\h)$} in the window $[a(\h),b(\h)]$ is a sequence 
$\{\h_j\}_{j=1}^\infty$ 
such that $\h_j\tendo$ as $j \tendi$ and for every $j$ there exist $\{(u_{j,\ell},E_{j,\ell})\}_{\ell=1}^{m(\h_j)} \subset H^2(\Otr)\cap H^1_{0,D}(\Otr)\times [a(\h_j),b(\h_j)] $ with 
\beqs
\N{(-\h_j^2\Delta -E_{j,\ell}) u_{j,\ell}}_{L^2(\Otr)} = \eps(\h_j),\,\,\N{u_{j,\ell}}_{L^2(\Otr)}=1, \,\,\big| \langle u_{j, \ell_1}, u_{j,\ell_2}\rangle_{L^2(\Otr)}\big| \leq \h_j^{-2}\eps(\h_j ) \tfor \ell_1 \neq \ell_2,
\eeqs
and $\supp\, u_{j,\ell} \subset \mc{K}$ for all $j$ and $\ell$, where $\mc{K}\Subset \Omega_1$.
\end{definition}

With $\{z_p(\h,\lambda)\}_p$ the set of poles of $z\mapsto R_\Otr(\lambda,z)$ counting multiplicities
(with $z_p(\h, \lambda)$ depending continuously on $\lambda$ for each $p$), let
\beq\label{eq:mathcalZ}
\mc{Z}(\e_1,\e_0,a,b;\h):=\Big\{p\,:\, z_p(\h,\lambda)\in (-2\e_1,2\e_1)-\ri(0,2\e_0)\text{ for some }\lambda^2\in[a,b]\Big\};
\eeq
$|\mc{Z}|$ is therefore the counting function of the poles of $z\mapsto R_\Otr(\lambda,z)$ that enter a rectangle next to zero in $z$  as $\lambda^2$ varies from $a$ to $b$.

\begin{theorem}[Analogue of Theorem \ref{thm:main2} in $\h$ notation]\label{thm:main2h}
Let $0<a_0\leq a(\h)\leq b(\h)<b_0<\infty$ and suppose there exists a family of quasimodes with quality 
\beq\label{eq:qualitysmall}
\eps(\h)\ll \h^{(5d+3)/2}
\eeq
 and multiplicity $m(\h)$ in the window $[a(\h),b(\h)]$ (in the sense of Definition \ref{def:quasimodes_multh}). If $\eps_0(\h)$ is such that, for some $\widetilde{S}>0$, 
\beq\label{eq:conditioneps0}
\eps_0(\h) \leq \widetilde{S}\,\h^{(d+1)/2} \quad\tfa \h,\,\,\tand\quad\eps_0(\h) \gg  \h^{-2d-1}\eps(\h)\quad \tas \h\tendo,
\eeq
then there exists $\h_0>0$ such that if $\h_j \leq \h_0$, then 
$$
\Big|\mathcal{Z}\Big(\h_j^{-(d+1)/2}\e_0(\h_j)\,,\,\e_0(\h_j)\,,\,a(\h_j)\,,\,b(\h_j)\,;\,\h_j\Big)\Big|\geq m(\h_j).
$$
\end{theorem}

\bpf[Proof of Theorem \ref{thm:main2} from Theorem \ref{thm:main2h}]
{We first show that if there exists a family of quasimodes $u_j$ with multiplicity $m_\ell$ in the window $[k_\ell^-,k_\ell^+]$ in $k$ notation (i.e.~in the sense of Definition \ref{def:quasimodes_mult}), then there exists a family of quasimodes in $\h$ notation (in the sense of Definition \ref{def:quasimodes_multh}).}

{Without loss of generality, each $k_\ell \in [\klower,\kupper]$ for some $j$ (if necessary by adding a window with $\klower=\kupper=k_\ell$), i.e.~given $\ell$ in the index set of the quasimode, there exists $j$ such that $\ell\in W_j$. We now index the quasimode with the index $j$ describing the windows $ [\klower,\kupper]$.} Let
\beqs
\begin{gathered}
\h_j:= (\klower)^{-1}, \,\, \quad m(\h_j):=m_j,\,\, \quad a(\h_j):=1, \,\, \quad b(\h_j):=\frac{(\kupper)^2}{(\klower)^2},\,\, \\
\quad\eps(\h_j):= \h^2_j \QMC(\h_j ^{-1}), \quad\tand\quad
E_{j,\ell}:= \frac{(k_{\ell})^2}{(\klower)^2} \, \tand\, {u_{j,\ell}:= u_\ell \tfor \ell \in \mathcal{W}_j}.
\end{gathered}
\eeqs
Then, 
\beqs
\N{ (\h_j^2 \Delta + E_{j,\ell})u_{j,\ell} }_{L^2(\Otr)} =( \klower)^{-2} \N{ (\Delta + k_\ell^2) u_\ell}_{L^2(\Otr)} = (\klower)^{-2} \QMC(k_\ell) \leq (\klower)^{-2} \QMC(\klower) = \eps(h_j),
\eeqs
where we have used that $\epsilon(k)$ is a decreasing function of $k$. Therefore, we have shown that 
there exists a family of quasimodes with multiplicity $m(\h)$ in the window $[a(\h),b(\h)]$ in $\h$ notation (i.e.~in the sense of Definition \ref{def:quasimodes_multh}).

The result of Theorem \ref{thm:main2} then follows from the 
result of Theorem \ref{thm:main2h} since (a) if $\lambda^2 \in [a(\h),b(\h)]$ and $\lambda/\h=k$, then $k\in [\klower,\kupper]$, and 
(b) if 
\beqs
z\in\mathcal{Z}\Big(\h_j^{-(d+1)/2}\e_0(\h_j)\,,\,\e_0(\h_j)\,,\,a(\h_j)\,,\,b(\h_j)\,;\,\h_j\Big),
\eeqs
then 
\beqs
\mu:= \h^{-2}_j z \in \mathcal{E}\Big((\klower)^{(d+1)/2}\QMC_0(\klower)\,,\,\QMC_0(\klower)\,,\,\klower\,,\,\kupper\Big).
\eeqs
\epf

\subsection{Results about meromorphic continuation}
\label{sec:proofssolution}

\bpf[Proof of Lemma \ref{l:inverseForm}]
Once we show \eqref{eq:twoRs}, the meromorphicity of $R_{\Otr}(\lambda,z)$ in $\lambda$ follows from the corresponding result for $R_{V}(\lambda) $ \cite[Theorem 4.4]{DyZw:19}.

We first show that the appropriate extension of a solution of \eqref{e:prob} is a solution of \eqref{e:probClassical} with $V(z)=-z1_\Otr$. We then show that the appropriate restriction of the solution of  \eqref{e:probClassical} with $V(z)=-z1_\Otr$ is a solution of \eqref{e:prob}.

Given $f\in L^2(\Otr)$, suppose that $u$ solves~\eqref{e:prob}.  Then, by the definition of the operator $\DtN$, there exists a $\lambda/\h$-outgoing function $v\in H^2_{\loc}(\mathbb{R}^d\setminus \Omega_1)$ such that 
$$
(-\h^2\Delta-\lambda^2)v=0\quad \text{on } \mathbb{R}^d\setminus \overline{\Omega_{1}},\quad \tand\quad \gamma_0^\tr v= \gamma_0^\tr u,\quad  \gamma_1^\tr  v= \gamma_1^\tr  u.
$$
Therefore, 
$$
\widetilde{v}:=\indicatorE u+1^{\rm ext}_{\mathbb{R}^d\setminus \Omega_1}v
$$
is in $H^2_{\loc}(\Omega_+)$ (since both its Dirichlet and Neumann traces match across $\partial \Omega_1$) and 
\begin{align*}
(-\h^2\Delta-\lambda^2)\widetilde{v}=z1_{\Otr}\widetilde{v} +\indicatorE f\quad\text{on }\Omega_+.
\end{align*}
By the definition of $R(\lambda,z)$ as the solution of \eqref{e:probClassical2} with $V(z) =- z 1_{\Otr}$,
$$
\widetilde{v}= R(\lambda,z)\indicatorE f,\qquad\text{ which implies that }\qquad u=\indicatorR R(\lambda,z)\indicatorE f.
$$

Now suppose $\widetilde{f}\in L^2(\Omega_+)$. Then, by \eqref{eq:Rdef} and \eqref{e:probClassical},
\beq\label{eq:Monday1}
\begin{cases}
(-\h^2\Delta-\lambda^2-z1_{\Otr} ) R(\lambda,z)\widetilde{f}=\widetilde{f}&\text{ in }\Omega,\\
 R(\lambda,z)\widetilde{f}=0 &\ton \Gamma_D,\\
 R(\lambda,z)\widetilde{f}\text{ is }\lambda/\h\text{-outgoing}.
\end{cases}
\eeq
Therefore, if $\widetilde{f}=\indicatorE f$ and  $v:= R(\lambda,z)\widetilde{f}$, then $(-\h^2\Delta-\lambda^2) R(\lambda,z)\widetilde{f}=0$ in $\mathbb{R}^d\setminus \overline{\Omega_{1}}$ and $v$ is $\lambda/\h$-outgoing. This last fact implies that
\beq\label{eq:Monday2}
\gamma_1^\tr (1^{\rm res}_{\mathbb{R}^d\setminus \overline{\Omega_{\tr}}}v)=\mc{D}(\lambda/\h)\gamma_0^\tr (1^{\rm res}_{\mathbb{R}^d\setminus \overline{\Omega_{\tr}}}v).
\eeq
Since $v=R(\lambda,z)\widetilde{f}\in H^2_{\loc}(\Omega_+)$,  the Dirichlet and Neumann traces of $v$ across $\Gamma_{\tr}$ do not have jumps, so that \eqref{eq:Monday2} implies that 
\beq\label{eq:Monday3}
\gamma_1^\tr (\indicatorR v)=\mc{D}(\lambda/\h)\gamma_0^\tr (\indicatorR v).
\eeq
Then, by \eqref{eq:Monday1} and \eqref{eq:Monday3}, $u:= \indicatorR v$ solves \eqref{e:prob} and the proof is complete.
 \end{proof}

\

\begin{proof}[Proof of Lemma \ref{lem:meromorphic}]
Since
\beqs
(-\h^2 \Delta-\lambda^2 - z1_{\Otr}) R(\lambda,0) = I - z 1_{\Otr} R(\lambda,0),
\eeqs
the definition of $R (\lambda,z)$ \eqref{eq:Rdef} implies that
\beq\label{eq:R3}
R (\lambda,z) = R (\lambda,0) \big(I-z1_{\Otr} R(\lambda,0)\big)^{-1}.
\eeq
We now claim that,  for any $\rho  \in C^\infty(\overline{\Omega_+})$ with $\supp\, \rho \Subset \Rea^d$ and $\rho\equiv 1$ on $\Omega_{\tr}$
\beq\label{eq:rho}
\big(I-z1_{\Otr} R(\lambda,0)\big)^{-1}=\big(I-z1_{\Otr} R (\lambda,0)\rho\big)^{-1}\big(I+z1_{\Otr} R (\lambda,0)(1-\rho)\big).
\eeq
{Indeed, 
\beqs
I- z1_{\Otr} R(\lambda,0) = \Big( I- z1_{\Otr}R(\lambda,0)(1-\rho)\big(I- z1_{\Otr} R(\lambda,0)\rho\big)^{-1}\Big)\big(I- z1_{\Otr} R(\lambda,0)\rho\big).
\eeqs
and thus 
\beq\label{eq:rho1}
\big(I- z1_{\Otr} R(\lambda,0)\big)^{-1} = \big(I- z1_{\Otr} R(\lambda,0)\rho\big)^{-1}\Big( I- z1_{\Otr}R(\lambda,0)(1-\rho)\big(I- z1_{\Otr} R(\lambda,0)\rho\big)^{-1}\Big)^{-1}.
\eeq
Observe that since $\rho R(\lambda,0)\rho :L^2(\Omega_+)\to L^2(\Omega_+)$ is compact, $1_{\Otr} R(\lambda,0)\rho:L^2(\Omega_+)\to L^2(\Omega_+)$ is compact, and the analytic Fredholm theorem~\cite[Theorem C.8]{DyZw:19} implies that 
\begin{equation}
\label{e:mero1}
z\mapsto (I-z1_{\Otr}R(\lambda,0)\rho )^{-1}\text{ is a meromorphic family of operators for $z\in \mathbb{C}$}
\end{equation}
 with finite rank poles.

Now, since $(1-\rho)1_{\Otr}=0$, for $|z|$ small enough,
\beq\label{eq:R1}
(1-\rho)\big(I- z1_{\Otr} R(\lambda,0)\rho\big)^{-1} = (1-\rho)\sum_{j=0}^\infty (z1_{\Otr}R(\lambda,0)\rho)^k=(1-\rho).
\eeq
However, by~\eqref{e:mero1} both the left- and right-hand sides of \eqref{eq:R1} are meromorphic for $z\in \mathbb{C}$. Therefore,  \eqref{eq:R1} holds for all $z\in \mathbb{C}$ and hence
\beq\label{eq:R2}
\big(I- z1_{\Otr} R(\lambda,0)(1-\rho)\big)^{-1} = I+ z1_{\Otr} R(\lambda,0)(1-\rho).
\eeq
Using \eqref{eq:R1} and \eqref{eq:R2} in \eqref{eq:rho1}, we obtain \eqref{eq:rho}.
}
Therefore, for $\chi\equiv 1$ on $\Omega_{\tr}$ and $\rho\equiv 1$ on $\supp \chi$,~\eqref{eq:R3},~\eqref{eq:rho} and~\eqref{eq:R1} imply that
$$
\chi  R(\lambda,z)\chi =\chi R(\lambda,0)\rho\big(I-z 1_{\Otr} R(\lambda,0)\rho\big)^{-1}\chi.
$$
Using~\eqref{e:mero1} again completes the proof.
\end{proof}

\

With $z_0(\h,\lambda)$ a pole of $R_\Otr(\lambda,z)$, let
\beq\label{eq:projection1}
\Pi_{z_0(\h,\lambda)}:= -\frac{1}{2\pi \ri} \oint_{z_0(\h,\lambda)} R_\Otr(\lambda,z)\, d z\quad\tand\quad
m_R\big(z_0(\h,\lambda)\big):= \rank \Pi_{z_0(\h,\lambda)},
\eeq
where $\oint_{z_0(\h,\lambda)}$ denotes integration over a circle containing $z_0$ and no other pole of $R_\Otr(\lambda,z)$. 

The following result then holds by, e.g., \cite[Theorem C.9]{DyZw:19}.

\ble\label{lem:projection}
For $\lambda \in \Rea\setminus\{0\}$, $\Pi_{z_0(\h,\lambda)}:L^2(\Otr) \rightarrow L^2(\Otr)$ is a bounded projection with finite rank.
 \ele

The next result concerns the singular behaviour of $R_\Otr(\lambda,z)$ near its poles in $z$, and is analogous to (parts of) \cite[Theorem 4.7]{DyZw:19} concerning the singular behaviour of $R_V(\lambda)$ near its poles in $\lambda$.

\ble\label{lem:singular}
For $\lambda \in \Rea\setminus\{0\}$, if $z_0= z_0(\h,\lambda)$ and $m_R(z_0)>0$, then there exists $M_{z_0}>0$ such that 
\beqs
R_\Otr(\lambda,z) =- \sum_{\ell=1}^{M_{z_0}}  \Pi_{z_0}\frac{ 
(-\h^2 \Delta-\lambda^2 -z)^{\ell-1}}{(z-z_0)^{\ell}}
+ A(z,z_0,\lambda)
\eeqs
where $z\mapsto A(z,z_0,\lambda)$ is holomorphic near $z_0$. 
\ele

\bpf
By Lemma \ref{lem:meromorphic}, for $\lambda \in \Rea\setminus\{0\}$, $z\mapsto R_\Otr(\lambda,z)$ is a meromorphic family of operators  (in the sense of \cite[Definition C.7]{DyZw:19}) from $L^2(\Otr) \rightarrow L^2(\Otr)$ and thus there exists $M_{z_0}>0$, finite-rank operators $A_\ell(\lambda): L^2(\Otr) \rightarrow L^2(\Otr)$, $\ell=1,\ldots,M_{z_0}$, and a family of operators $z\mapsto A(z,z_0,\lambda)$ from $L^2(\Otr) \rightarrow L^2(\Otr)$, holomorphic near $z_0$, such that 
\beqs
R_{\Otr}(\lambda,z) = \sum_{\ell=1}^{M_{z_0}} \frac{A_\ell(\lambda)}{(z-z_0)^\ell} + A(z,z_0,\lambda).
\eeqs
By integrating around $z_0$ and using the residue theorem, we have $A_1 = -\Pi_{z_0}$. Then, with $\equiv$ denoting equality up to holomorphic operators,
\begin{align*}
 R_\Otr(\lambda,z)(-\h^2 \Delta - \lambda^2 -z) &\equiv \sum_{\ell=1}^{M_{z_0}} \bigg(
\frac{A_\ell(-\h^2 \Delta - \lambda^2 -z_0)}{(z-z_0)^\ell} - \frac{A_\ell}{(z-z_0)^{\ell-1}}
\bigg),\\
& 
=\sum_{\ell=1}^{M_{z_0}} \frac{ A_\ell(-\h^2 \Delta - \lambda^2 -z_0) - A_{\ell+1}}{(z-z_0)^\ell},
\end{align*}
where we define $A_{M_{z_0}+1}:=0$. Since $R_\Otr(\lambda,z)(-\h^2 \Delta -\lambda^2 - z)= I$ on $H^2(\Otr)\cap H_{0}^1(\Omega_+)$, $A_{\ell+1} = A_\ell (-\h^2 \Delta -\lambda^2-z)$, $\ell=1,\ldots, M_{z_0}$, and the result follows from density of $H^2(\Otr)\cap H_{0}^1(\Omega_+)$ in $L^2(\Otr)$.
\epf

\subsection{The semiclassical maximum principle}

The following result is the semiclassical maximum principle of \cite[Lemma 2]{TaZw:98}, \cite[Lemma 4.2]{TaZw:00} (see also \cite[Lemma 7.7]{DyZw:19}).

\begin{theorem}[Semiclassical maximum principle]
\label{thm:scmp}
Let $\cH$ be an Hilbert space and $z\mapsto Q(z,\h)\in\mathcal{L}(\cH)$
an holomorphic family of operators in a neighbourhood of
\beq\label{eq:box}
\Omega(\h):=\big(w-2\beta(\h),w+2\beta(\h)\big)+\ri\big(-\delta(\h)\h^{-L},\delta(\h)\big),
\eeq
where 
\beq\label{eq:restrict1}
0<\delta(\h)<1,\qquad \tand \quad \beta(\h)^{2}\geq C\h^{-3L}\delta(\h)^{2}
\eeq
for some $L>0$ and $C>0$. Suppose that
\begin{align}\label{eq:expbound}
\Vert Q(z,\h)\Vert_{\cH\rightarrow\cH}&\leq\exp(C\h^{-L}),\qquad z\in\Omega,\\ \label{eq:absbound}
\Vert Q(z,\h)\Vert_{\cH\rightarrow\cH}&\leq  
\frac{C}{\Im \,z},
\qquad \Im \,z>0,\quad z\in\Omega.
\end{align}
Then
\beq\label{eq:scmp1}
\Vert Q(z,\h)\Vert_{\cH\rightarrow\cH}\leq \frac{C}{\delta(\h)}
\exp(C+1)\quad \tfa z\in\big[w-\beta(\h),w+\beta(\h)\big].
\eeq
\end{theorem}

\begin{proof}[References for proof]
Let $f,g\in \cH$ with $\|f\|_{\cH}=\|g\|_{\cH}=1$, and let
\[
F(z,\h):=\big\langle Q(z+w,h)g,f\big\rangle_{\cH}.
\]
The result \eqref{eq:scmp1} follows from the ``three-line theorem in a rectangle'' (a consequence of the maximum principle) stated as \cite[Lemma D.1]{DyZw:19} applied to the holomorphic family $(F(\cdot,h))_{0<h\ll1}$
with
\begin{align*}
R=2\beta(\h),\qquad \delta_{+}=\delta(\h),\qquad \delta_{-}=\delta(\h)h^{-L},\\
M=M_{-}=\exp(C\h^{-L}),\qquad M_{+}=C\,\delta(\h)^{-1}.
\end{align*}
\end{proof}

\subsection{The generalized bicharacteristic flow}\label{sec:flow}

Recall that 
\beqs
T_{\overline{\Omega_+}}^*\mathbb{R}^d:=  \big\{(x,\xi) \in T^*\mathbb{R}^{d}, x\in \overline{\Omega_+}\big\}=
\big\{ x\in \overline{\Omega_+}, \xi \in \Rea^d\big\}
 \eeqs
 and
\beqs
S^*_{\overline{\Omega_+}}\mathbb{R}^d:= \big\{(x,\xi) \in S^*\mathbb{R}^{d}, x\in \overline{\Omega_+}\big\}=
 \big\{ x\in \overline{\Omega_+}, \xi \in \Rea^d \text{ with } |\xi|=1\big\}.
 \eeqs
We write $\varphi_t:S_{\Omega_+}^*\mathbb{R}^d\to S_{\Omega_+}^*\mathbb{R}^d$ for the generalized bicharacteristic flow associated to 
a symbol $p$ (see e.g.~\cite[\S24.3]{Ho:85}). 
Since the flow over the interior is  generated by the Hamilton vector field $H_p$,  for any symbol $b \in C^\infty_c(T^*_{\Omega_+}\Rea^d)$,
\beq\label{eq:Hpb}
\partial_t( b \circ \varphi_t)= H_p b = \{p, b\},
\eeq
where $\{\cdot,\cdot\}$ denotes the Poisson bracket; see \cite[\S2.4]{Zw:12}.

We primarily consider the case when $p$ is the semiclassical principal symbol of the Helmholtz equation, namely 
$p=|\xi|^2-1$. 
 By Hamilton's equations, away from the boundary of $\Omega_+$, the corresponding flow satisfies $\dot{x}_i= 2 \xi_i$ and $\dot{\xi}_i=0$, and thus, for $\rho=(x,\xi)$ with $x$ away from $\Gamma_D$, $\varphi_t(\rho)= x+ 2 t\xi$ for $t$ sufficiently small; 
i.e., the flow has speed two.

We let $\projx$ denote the projection operator onto the spatial variables; i.e.
\beqs
\projx: T^*_{\overline{\Omega_+}}\Rea^d \rightarrow \overline{\Omega_+}, \quad\projx((x,\xi)) = x.
\eeqs

\subsection{Geometry of trapping}\label{sec:trapping}

Let $\chi \in C^\infty(\overline{\Omega_+};[0,1])$ with $\supp \chi \Subset \Rea^d$ and $\chi\equiv 1$ near $\Omega_-$ and define $\r:T_{\Omega_+}^*\mathbb{R}^d\to \mathbb{R}$ by
$$
\r(x,\xi):=(1-\chi(x))|x|
$$
so that there is $c>0$ such that for $r_0>c$,
$$
\{x\,:\,\r>r_0\}=\mathbb{R}^d\setminus B(0,r_0).
$$
Moreover, note that $\{\r\leq c\}$ is compact for every $c$. 
Next, define the \emph{directly escaping sets},
\begin{gather*}
\mc{E}_{\pm}:=\Big\{(x,\xi)\in S^*\mathbb{R}^d\mid \r(x,\xi)\geq r_0,\quad \pm\langle x,\xi\rangle_{\mathbb{R}^d}\geq 0\Big\},\\
\mc{E}_{\pm}^o:=\Big\{(x,\xi)\in S^*\mathbb{R}^d\mid \r(x,\xi)\geq r_0,\quad \pm\langle x,\xi\rangle_{\mathbb{R}^d}> 0\Big\}.
\end{gather*}
Then, 
\begin{equation}
\label{e:directEscape1}
\rho \in \mc{E}_{\pm}\quad\text{ implies that }\quad \varphi_{\pm t}(\rho)\in \mc{E}_{\pm}\quad\tand\quad \r(\varphi_{\pm t}(\rho))\geq \sqrt{\r(\rho)^2+4t^2},\quad\text{for all }t\geq 0.
\end{equation}
Therefore, $\r(\varphi_t(\rho))\to \infty$ as $t\to \pm\infty$ and hence $\rho\in \mc{E}_\pm$ escapes forward/backward in time. This, in particular implies that
\begin{equation}
\label{e:directEscape2}
\r(\rho)\geq r_0,\,\r(\varphi_{\mp t_0}(\rho))\leq \r(\rho)\text{ for some }t_0>0\quad\Rightarrow\quad \pm\langle x(\rho),\xi(\rho)\rangle>0.
\end{equation}
We now define the \emph{outgoing tail} $\Gamma_{+}\subset S^*_\Omega\mathbb{R}^d$,
the \emph{incoming tail} $\Gamma_-\subset S^*_\Omega\mathbb{R}^d$,  and the \emph{trapped set}, $K$ by 
\begin{equation}
\label{e:trappedSets}
\Gamma_{\pm}:=\{q\in S^*_{\Omega}\mathbb{R}^d\mid \r(\varphi_t(q))\not\to \infty,\, t\to \mp\infty\},\qquad K:=\Gamma_+\cap \Gamma_-;
\end{equation}
i.e.~the outgoing tail is the set of trajectories that do not escape as $t\rightarrow -\infty$, 
the incoming tail is the set of trajectories that do not escape as $t\tendi$, 
and the trapped set is the set of trajectories that do not escape in either time direction.

We now recall some basic properties of $\Gamma_\pm$ and $K$, with these proved in a more general setting in \cite[\S6.1]{DyZw:19}.
\begin{lemma}
\label{l:trapping}

\

(i) The sets $\Gamma_{\pm},K$ are closed in $S^*_\Omega\mathbb{R}^d$ and $K\subset \{\r<r_0\}$.

(ii) Suppose that $\rho_n\in S^*_{\Omega_+}\mathbb{R}^d$ with $\rho_n\to \rho$ and there are $t_n\to \infty$ such that $ \varphi_{\pm t_n}(\rho_n)\to \rho_\infty$. Then $\rho\in \Gamma_\mp$. 
\end{lemma}

\begin{proof}
\noindent (i) 
We show that $\Gamma_-$ is closed in $S_\Omega^*\mathbb{R}^d.$ Suppose that $\rho_0\in S_\Omega^*\mathbb{R}^d\setminus \Gamma_-$. Then $\r(\varphi_t(\rho_0))\to \infty$ as $t\to \infty$.  In particular, there are $0<t_1<t_2$ such that $\r(\varphi_{t_2}(\rho_0))\geq r_0$ and $\r(\varphi_{t_1}(\rho_0))\leq \r(\varphi_{t_2}(\rho_0))$. So, applying~\eqref{e:directEscape2} with $\rho=\varphi_{t_2}(\rho_0)$, we have $\varphi_{t_2}(\rho_0)\in\mc{E}_+^o$. Since $\mc{E}_{+}^o$ is open and $\varphi_{t_2}$ is continuous we have $\varphi_{t_2}(\rho)\in\mc{E}_+^o$ for all $\rho$ sufficiently close to $\rho_0$ and hence, by~\eqref{e:directEscape1}, $\rho\notin \Gamma_-$. 
Therefore $\Gamma_-$ is closed. By an identical argument $\Gamma_+$ and hence $\Gamma_-\cap\Gamma_+$ are closed.

Now, we show that $K\subset \{\r <r_0\}$. Note that $S^*_\Omega\mathbb{R}^d\cap \{\r\geq r_0\}\subset \mc{E}_+\cup\mc{E}_-$. But, $\mc{E}_+\cap \Gamma_-=\emptyset$ and $\mc{E}_-\cap \Gamma_+=\emptyset$ and hence $S^*_\Omega\mathbb{R}^d\cap \{\r\geq r_0\}\cap \Gamma_+\cap\Gamma_-=\emptyset$ as claimed.
\medskip

(ii) We {prove the result for} $t_n\to \infty$; the proof of the other case is similar. 
{Seeking a contradiction, assume that} $\rho\notin \Gamma_-$. Then there exists $T>0$ such that $\r (\varphi_T(\rho))\in \mathcal{E}^o_{{+}}$ and hence, since $\varphi_T$ is continuous, and $ \mathcal{E}^o_{+}$ is open, for $n$ large enough, $\varphi_{T}(\rho_n)\in  \mathcal{E}^o_{+}$. But then, by~\eqref{e:directEscape1} and~\eqref{e:directEscape2} for $t\geq T$ 
$\r(\varphi_{t}(\rho_n))\geq \sqrt{r_0^2+4(t-T)^2}$. In particular, for $n$ large enough, 
$$
\r(\varphi_{T_n}(\rho_n))\geq  \sqrt{r_0^2+4(T_n-T)^2}\to \infty
$$
which contradicts the fact that $\r(\varphi_{T_n}(\rho_n))\to \rho_\infty$.
\end{proof}

\subsection{Complex scaling}\label{sec:complexscaling}
We now review the method of complex scaling following \cite[\S4.5]{DyZw:19}. We first fix a small angle of scaling, $\theta>0$, and the radius, $r_1>r_0$, where the scaling starts; without loss of generality, we assume that $\Omega_1\Subset \{x : \r \leq r_1\}$. Let $f_\theta \in C^\infty([0,\infty)$ satisfy
\beqs
\begin{gathered}f_\theta(r)\equiv 0,\quad r\leq r_1;\qquad f_\theta(r)=r\tan \theta,\quad r\geq 2r_1;\\
f_\theta'(r)\geq 0,\quad r\geq 0;\qquad \{f_\theta'(r)=0\}=\{f_\theta(r)=0\}.
\end{gathered}
\eeqs
Then, consider the totally real submanifold (see \cite[Definition 4.28]{DyZw:19})
$$
\Gamma_\theta:=\left\{ x+ \ri f_\theta(|x|)\frac{x}{|x|}\,:\, x\in \mathbb{R}^d\right\}\subset \mathbb{C}^d
$$
and note that we identify $\Omega_-$ with its image on $\Gamma_\theta$. We define the complex scaled operator $P_\theta$ on $\Omega$ by the Dirichlet realization of
$$
P_\theta:=\Big(\frac{1}{1+\ri f_\theta'(r)}\h D_r\Big)^2-\frac{(d-1)\ri}{(r+\ri f_\theta(r))(1+\ri f_\theta'(r))}\h^2D_r-\frac{\h^2\Delta_\phi}{(r+\ri f_\theta(r))^2},\qquad \{\r\geq r_0\}.
$$
where $\Delta_\phi$ denotes the Laplacian on the round sphere $S^{d-1}$. Note that $P_\theta$ is a semiclassical differential operator of second order such that on $r\leq r_1$, $P_\theta=-\h^2\Delta$ with principal symbol, $p_\theta$, satisfying
$p_\theta(x,\xi)=|\xi|^2$ on $\{\r\leq r_1\},$ and in polar coordinates $x=r\phi$,
\beq\label{eq:symbolPtheta}
p_\theta(r,\phi,\xi_r,\xi_\phi)=\frac{\xi_r^2}{(1+\ri f_\theta'(r))^2}+\frac{|\xi_\phi|^2}{(r+\ri f_\theta(r))^2}.
\eeq
Now, by e.g.~\cite[Theorems 4.36,4.38]{DyZw:19}, for $\Im (\re^{\ri \theta}\lambda)>0$, 
\begin{equation}
\label{e:Fred1}
P_\theta-\lambda^2:H^2(\Omega_+)\cap H_0^1(\Omega_+) \to L^2(\Omega_+)\text{ is a Fredholm operator of index zero}.
\end{equation}
In particular, for $V\in L^\infty(\mathbb{R}^d)$, $\supp V\subset \{r< r_1\}$, this implies that 
\begin{equation}
\label{e:Fred2}
 P_\theta-\lambda^2+V:H^2(\Omega_+)\cap H_0^1(\Omega_+)\to L^2(\Omega_+)\text{ is a Fredholm operator of index zero}.
\end{equation}
Moreover, by \cite[Theorem 4.37]{DyZw:19}, $(P_\theta-\lambda^2+V)^{-1}$ has the same poles as $R_{V}(\lambda) $ and, for
$\chi\in C_c^\infty(\{x : \r\le r_1\})$ with $\supp \chi \Subset \Rea^d$,
\begin{equation}
\label{e:resolveAgree}
\chi(P_\theta-\lambda^2+V)^{-1}\chi= \chi R_{V}(\lambda) \chi.
\end{equation}

\subsection{Semiclassical pseudodifferential operators}

For simplicity of exposition, we begin by discussing semiclassical pseudodifferential operators on $\Rea^d$, and then 
outline below how to extend the results from $\Rea^d$ to a manifold $\Gamma$ (with these results then applied with $\Gamma=\Gamma_D$ or $\Gamma=\Gamma_\tr$).

A symbol is a function on $T^*\Rea^d:=\Rea^d \times (\Rea^d)^*$ that is also allowed to depend on $\h $, and thus can be considered as an $\h $-dependent family of functions.
Such a family $a=(a_\h )_{0<\h \leq\h _0}$, with $a_\h  \in C^\infty({T^*\mathbb R^d})$, 
is a \emph{symbol
of order $m$}, written as $a\in S^m(\Rea^d)$, if for any multiindices $\alpha, \beta$
\beq\label{eq:Sm}
| \partial_x^\alpha \partial^\beta_\xi a_\h (x,\xi) | \leq C_{\alpha, \beta} 
\langle \xi\rangle^{m -|\beta|}
\quad\tfa (x,\xi) \in T^* \Rea^d \text{ and for all } 0<\h \leq \h _0,
\eeq
where $\langle\xi\rangle:= (1+ |\xi|^2)^{1/2}$ and $C_{\alpha, \beta}$ does not depend on $\h $; see \cite[p.~207]{Zw:12}, \cite[\S E.1.2]{DyZw:19}. 

For $a \in S^m$, we define the \emph{semiclassical quantisation} of $a$, denoted by $\Oph(a):\mathcal{S}(\mathbb{R}^d)\to \mathcal{S}(\mathbb{R}^d)$, by  
\beq \label{eq:quant}
\Oph (a) v(x) := (2\pi \h )^{-d} \int_{\Rea^d} \int_{\Rea^d} 
\exp\big(\ri (x-y)\cdot\xi/\h \big)\,
a(x,\xi) v(y) \,d y  d \xi;
\eeq
\cite[\S4.1]{Zw:12} \cite[\S E.1 (in particular Page 543)]{DyZw:19}. The integral in \eqref{eq:quant} need not converge, and can be understood \emph{either} as an oscillatory integral in the sense of \cite[\S3.6]{Zw:12}, \cite[\S7.8]{Ho:83}, \emph{or} as an iterated integral, with the $y$ integration performed first; see \cite[Page 543]{DyZw:19}.

Conversely, if $A$ can be written in the form above, i.\,e.\ $A =\Oph(a)$ 
with $a\in S^m$, we say that $A$ is a \emph{semiclassical pseudo-differential operator of order $m$} and
we write $A \in \Psi_{\h }^m$. We use the notation $a \in \h ^l S^m$  if $\h ^{-l} a \in S^m$; similarly 
$A \in \h ^l \Psi_\h ^m$ if 
$\h ^{-l}A \in \Psi_\h ^m$. We write $\Psi_\h^{-\infty} = \cap_m \Psi_\h^{-m}$.


Let the quotient space $ S^m/\h  S^{m-1}$ be defined by identifying elements 
of  $S^m$ that differ only by an element of $\h  S^{m-1}$. 
For any $m$, there is a linear, surjective map
$$
\sigma^m_{\h }:\Psi_\h  ^m \to S^m/\h  S^{m-1},
$$
called the \emph{principal symbol map}, 
such that, for $a\in S^m$,
\beq\label{eq:symbolone}
\sigma_\h ^m\big(\Op_\h (a)\big) = a \quad\text{ mod } \h  S^{m-1};
\eeq
see \cite[Page 213]{Zw:12}, \cite[Proposition E.14]{DyZw:19} (observe that \eqref{eq:symbolone} implies that 
$\operatorname{ker}(\sigma^m_{\h }) = \h \Psi_\h  ^{m-1}$).
When applying the map $\sigma^m_{\h }$ to 
elements of $\Psi^m_\h $, we denote it by $\sigma_{\h }$ (i.e.~we omit the $m$ dependence) and we use $\sigma_{\h }(A)$ to denote one of the representatives
in $S^m$ (with the results we use then independent of the choice of representative).
Key properties of the principal symbol that we use below is that 
\beq\label{eq:thething}
\sigma_{\h }(AB)=\sigma_{\h }(A)\sigma_{\h }(B) \quad\tand\quad
 \h^{-1} \sigma_\h\big(\big[\operatorname{Op}_\h(a),\operatorname{Op}_\h(b)\big]\big) = -\ri \{a,b\},
\eeq
where (as in \S\ref{sec:flow}) $\{\cdot,\cdot\}$ denotes the Poisson bracket; see 
\cite[Proposition E.17]{DyZw:19} and \cite[Equation E.1.44]{DyZw:19}, \cite[Page 68]{Zw:12}.              

While the definitions above are written for operators on $\Rea^d$, semiclassical pseudodifferential operators and all of their properties above have analogues on compact manifolds  (see e.g.~\cite[\S14.2]{Zw:12},~\cite[\S E.1.7]{DyZw:19}). Roughly speaking, the class of semiclassical pseudodifferential operators of order $m$ on a compact manifold $\Gamma$, $\Psi_\h^m(\Gamma)$ are operators that, in any local coordinate chart, have kernels of the form~\eqref{eq:quant} where the function $a\in S^m$ modulo a remainder operator $R$ that has the property
\begin{equation}
\label{e:remainder}
\|R\|_{H_\h ^{-N}\to H_{\h }^N}\leq C_{N} \h ^N.
\end{equation}
We say that an operator $R$ satisfying~\eqref{e:remainder} is $O(\h ^\infty)_{\Psi_\h^{-\infty}}$.

Semiclassical pseudodifferential operators on manifolds continue to have a natural principal symbol map
$$
\sigma_{\h }:\Psi_\h^m\to S^m(T^*\Gamma)/\h  S^{m-1}(T^*\Gamma)
$$
where now $S^m(T^*\Gamma)$ is a class of functions on $T^*\Gamma$, the cotangent bundle of $\Gamma$ which satisfy the estimates~\eqref{eq:Sm}. Furthermore, \eqref{eq:thething} holds as before.

Finally, there is a noncanonical quantisation map $\Op_{\h }:S^m(T^*\Gamma)\to \Psi_\h^m(\Gamma)$ that satisfies
$$
\sigma_\h  (\Op_{\h }(a))=a.
$$
and for all $A\in \Psi_\h^m(\Gamma)$, there is $a\in S^m(T^*\Gamma)$ such that 
$$
A=\Op_{\h }(a)+O(\h ^\infty)_{\Psi_\h^{-\infty}}.
$$

\subsection{Defect measures}\label{sec:defect}
We say that a sequence $\{u_{\h_n}\}_{n=1}^\infty$ with $\|u_{h_n}\|_{L^2(\Rea^d)}\leq C$ for all $n$ (with $C$ independent of $n$)
has \emph{defect measure $\mu$} if for all $a\in C_c^\infty(T^*\mathbb{R}^d)$, 
$$
\big\langle \operatorname{Op_{\h_{n}}}(a) u_{\h_{n}},u_{\h_{n}}\big\rangle_{L^2(\mathbb{R}^d)}\to \int a\,d\mu,
$$
where $\operatorname{Op}_{\h}(a)$ is defined by \eqref{eq:quant}.
By, e.g., \cite[Theorem 5.2]{Zw:12}, $\mu$ is a positive Radon measure on $T^*\mathbb{R}^d$. We say that $u_{\h_n}$ and $f_{\h_n}$ have joint defect measure $\mu^j$ if 
\beq\label{eq:joint}
\big\langle \operatorname{Op_{\h_{n}}}(a) {u_{\h_n}},{f_{\h_n}}\big\rangle_{L^2(\mathbb{R}^d)}\to \int a\,d\mu^j.
\eeq

We usually suppress the $n$ in the notation and instead write that $u_{\h}$ has defect measure $\mu$ and $u_{\h}$ and $f_\h$ have joint defect measure $\mu^j$.

\begin{lemma}\mythmname{\cite[Theorem 5.3]{Zw:12}}
\label{l:ellipticDefect}
Let $\newP \in \Psi_\h^m(\mathbb{R}^d)$ and suppose that $u_\h$ has defect measure $\mu$ and satisfies
$$
\N{\newP  u_\h}_{L^2(\Rea^d)}=o(1).
$$
Then,
$
\supp \mu\subset \{\sigma_\h(\newP )=0\},
$
where $\sigma_\h(\newP )$ is the semiclassical principal symbol of $\newP $.
\end{lemma}

The following lemma is the defect-measure analogue of the propagation of singularities result \cite[Theorem E.47]{DyZw:19}.

\begin{lemma}
\label{l:propagateDefect}
Let $\newP \in \Psi_\h^m(\mathbb{R}^d)$ with $\Im \sigma_\h( \newP) \leq 0$. There exists $C>0$ such that the following holds: suppose that $u_\h$ has defect measure $\mu$ and satisfies
$$
\newP  u_\h=\h f_\h,
$$
where $\|f_\h\|_{L^2(\Rea^d)}\leq C_1$ and $u_\h$ and $ f_\h$ have joint defect measure $\mu^j$.
Then, for all  real valued $a\in C_c^\infty(T^*\mathbb{R}^d)$,
$$
\mu(H_{\Re \sigma_\h(\newP)} a^2)\geq -2\Im \mu^j(a^2) - C \mu\big(\langle\xi\rangle^{m-1} a^2\big).
$$
\end{lemma}
\begin{proof}
Let $A=\Oph(a)$. Since $\sigma_\h(A^*)=a$ (by 
 \cite[Equation E.1.45]{DyZw:19})
and thus
$\sigma_\h(A^*A) =a^2$ (by \cite[Equation E.1.43]{DyZw:19}),
by the definition of the joint measure \eqref{eq:joint}, 
\beq
\label{e:joint}
2\h^{-1}\Im \big\langle A^*A{u_\h, \newP u_\h}\big\rangle=2\Im \mu^j(a^2)+o(1), 
\eeq
and, by \eqref{eq:thething} and \eqref{eq:Hpb},
\beqs
\h^{-1} \Im \big\langle [A^*A,\Re \newP ]u_\h,u_\h\big\rangle= \mu(H_{\Re \sigma_\h(\newP)}a^2).
\eeqs
Since $2 \Im z = \Im (z- \overline{z})$ and $\newP = \Re \newP  + \ri \Im \newP $ with $\Re \newP $ and $\Im \newP $ both self-adjoint,
\begin{align}\nonumber
 {-}2\h^{-1}&\Im\big\langle A^*A {u_\h,\newP  u_\h}\big\rangle,\\ \nonumber
&=\h^{-1} \Im\Big(\big\langle \newP  u_\h, A^*A u_\h\big\rangle-\big\langle A^*A u_\h,\newP  u_\h\big\rangle\Big),\\ \nonumber
&=\h^{-1} \Im\Big(\big\langle (A^*A\Re \newP -\Re \newP A^*A) u_\h, u_\h\big\rangle+\ri\big\langle( A^*A\Im \newP  +\Im \newP A^*A) u_\h, u_\h\big\rangle\Big),\\ \nonumber
&=\h^{-1} \Im\big\langle (A^*A\Re \newP -\Re \newP A^*A) u_\h, u_\h\big\rangle+2 \h^{-1}\Re\big\langle A^*A\Im \newP  u_\h, u_\h\big\rangle),\\ \nonumber
&= \mu(H_{\Re \sigma_\h(\newP)} a^2) +o(1) +2\h^{-1}\Re \big\langle \Im \newP  Au_\h,Au_\h\big\rangle +2\h^{-1}\Re \big\langle A^*[A,\Im \newP ]u_\h,u_\h\big\rangle,\\
&\leq \mu(H_{\Re \sigma_\h(\newP)}a^2)+o(1)+2\h^{-1}\Re \big\langle A^*[A,\Im \newP ]u_\h,u_\h\big\rangle + C \| Au\|^2_{H^{(m-1)/2}},
\label{e:commutator}
\end{align}
where the last line follows from the sharp Garding inequality (see, e.g., \cite[Proposition E.32]{DyZw:19}) and the fact that $\Im\sigma_\h( \newP) \leq 0$.
By \eqref{eq:thething},
$$
\Re \h^{-1}\sigma_\h \big(A^*[A,\Im \newP ]\big)=\Re \big(-\ri a\{a,\Im \sigma_\h(\newP)\}\big)=0,
$$
and therefore, since the kernel of $\sigma_\h: \Psi_\h^{-\infty}\to S^{-\infty}/\h S^{-\infty}$ is $\h\Psi_\h^{-\infty}$, $\h^{-1}\Re A^*[A,\Im \newP]\in \h\Psi_\h^{-\infty}$ and, in particular,
\beq\label{e:commutator2}
\Re \langle A^*[A,\Im \newP ]u_\h,u_\h\rangle=\cO(\h^2).
\eeq
The lemma follows from combining \eqref{e:commutator} with \eqref{e:commutator2} and \eqref{e:joint}, and sending $\h\tendo$.
\end{proof}

\begin{corollary}\label{cor:forwards}
Let $\Xi\geq 0$ and suppose the assumptions of Lemma \ref{l:propagateDefect} hold and, in addition, $\mu^j=0$. Then, with $\varphi_t$ the bicharacteristic flow corresponding to the symbol $\Re \sigma_\h(\newP)$, for any $B\subset T^*\Rea^d\cap \{|\xi|\leq\Xi\}$,
\beq\label{eq:forwards}
\mu\big( \varphi_t(B)\big) \leq e^{Ct\langle \Xi\rangle^{m-1}}\mu(B) \quad\tfor t\geq 0.
\eeq
\end{corollary}

Corollary \ref{cor:forwards} shows that, under the assumptions of Lemma \ref{l:propagateDefect}, we have information about the defect measures of sets moving forward under the flow.

\

\bpf[Proof of Corollary \ref{cor:forwards}]
Let $a\in C_c^\infty(T^*\Rea^d\cap \{|\xi|\leq \Xi\})$
By \eqref{eq:Hpb},
\beqs
\begin{aligned}
\partial_t \left( e^{Ct\langle \Xi\rangle^{m-1}}\int (a^2 \circ \varphi_t)\, d \mu\right) & =\int \partial_t (a^2 \circ \varphi_t)+( C\langle \Xi\rangle^{m-1}a^2)\circ\varphi_t \, d \mu\\
&\geq\int \partial_t (a^2 \circ \varphi_t)+( C\langle \xi\rangle^{m-1}a^2)\circ\varphi_t \, d \mu\\
&= \mu\big(H_{\Re \sigma_\h(\newP)} a^2+C\langle \xi\rangle^{m-1}a^2\big) \geq 0,
\end{aligned}
\eeqs
and thus 
\beq\label{eq:inequality}
e^{Ct\langle \Xi\rangle^{m-1}}\int a^2 \, d \mu \geq \int (a^2 \circ \varphi_{-t})\, d \mu.
\eeq
Let $1_B$ be the indicator function of $B\subset  T^*\Rea^d\cap\{|\xi|\leq \Xi\}$. By approximating $1_B$ by squares of smooth symbols, compactly supported symbols \eqref{eq:inequality} holds with $a^2= 1_B$. Since $1_B \circ \varphi_{-t}= 1_{\varphi_t(B)}$ the result \eqref{eq:forwards} follows. More precisely, we first let $B$ open and $K_n\Subset B$ compact with $K_n\uparrow B$ and choose $a_n\in C_c^\infty(T^*\mathbb{R}^d\cap \{|\xi|\leq \Xi\})$ with $a_n\equiv 1$ on $K_n$ and $\supp a_n\subset B$. 
The result for $B$ open follows by monotonicity of measure from below; the result for general $B$ follows by outer regularity of $\mu$.
\epf

\

We now review some recent results from \cite{GaLaSp:21} about defect measures when $u_{\h}$ satisfies the Helmholtz equation.
Let $f_\h \in L^2_{\rm comp}(\Rea^d)$ be such that $\|f_\h\|_{L^2(\Rea^d)}\leq C$.

We use Riemannian/Fermi normal coordinates $(x_1,x')$ in which $\Gamma_D$ is given by $\{x_1=0\}$ and $\Omega_+$ is $\{x_1>0\}$. 
The conormal
and cotangent variables are given by $(\xi_1,\xi')$.
Recall the definition of the hyperbolic set 
\beqs
\mc{H}_{\Gamma_D}:= \big\{ (x',\xi')\in T^*\Gamma_D \, :\, |\xi'|_{g} <1\}\subset T^*\Gamma_D,
\eeqs
(where the metric $g$ is that induced by $\Gamma_D$) and the definition of the gliding set
$$
\mc{G}:=\Big\{x_1=H_px_1=0,\quad H_p^2x_1<0, \quad |\xi|=1\Big\} \subset S_{\Gamma_D}^*\mathbb{R}^d.
$$
Let $\mc{N}\in \Psi_\h^{m}(\Gamma_D)$ and $\mc{D}\in \Psi_\h^{m+1}(\Gamma_D)$ have real-valued principal symbols satisfying 
\begin{equation}
\label{e:boundaryCond}
\begin{gathered}
|\sigma_\h(\mc{N})|^2\langle \xi'\rangle^{-2m} +|\sigma_\h(\mc{D})|^2\langle \xi'\rangle^{-2m-2}>c>0\text{ on }T^*\Gamma_D,\\ 
\sigma_\h(\mc{N})\sigma_\h(\mc{D})>0\text{ on }\overline{B^*\Gamma_D},
\end{gathered}
\end{equation}
where $B^*\Gamma_D:= \{ (x',\xi') : |\xi'|_g< 1\}$ and $|\xi'|_g$ denotes the norm of $\xi'$ in the metric, $g$, induced on $\Gamma_D$ from $\mathbb{R}^d$.
Let $u\in L^2_{\rm loc}(\Omega_+)$ be a solution to
$$
(-\h^2\Delta-1)u_\h=\h f_\h\quad\text{ in }\Omega_+, \qquad (\mc{N}h D_\nu -\mc{D})u_\h|_{\Gamma_D}=o(1).
$$
Later we restrict attention to specific $\mc{N}$ and $\mc{D}$, but we consider more-general operators here because we believe some of our intermediate results (specifically Lemma \ref{lem:hayfever}) are of independent interest; see \cite{GaMaSp:21-2}.

Suppose that $1^{\rm ext}_{\Omega_+}u_{\h}$ 
has defect measure $\mu$ and $1^{\rm ext}_{\Omega_+}u_{\h}$ and $f_h$ have joint defect measure $\mu^j$.
On $\Gamma_D$, let $\nu_j$ be the joint measure associated with the Dirichlet and Neumann traces and $\nu_n$ be the measure associated with the Neumann trace; see \cite[Theorem 2.3]{GaLaSp:21}. 
In what follows, we only use the fact that $\dot{n}^j \nu_n=\nu_j$ where $\dot{n}^j=\sigma_\h(\mc{N})/\sigma_\h(\mc{D})$ is bounded (see \cite[Lemma 2.14 and \S 4]{GaLaSp:21}).

With $u$ as above, let $\mu^{\rm in/out}$ be the positive measures  on $T^*\Gamma_D $, supported in the hyperbolic set $\mathcal{H}_{\Gamma_D}$, and defined in \cite[Lemma 2.9]{GaLaSp:21}/\cite[Proposition 1.7, Part (ii)]{Mi:00}.

In the following lemma, $^bT^*\Omega_+$ denotes the $b$-cotangent bundle to $\Omega_+$ and $\pi: T^*\Omega_+\rightarrow \,^bT^*\Omega_+$ is defined in local coordinates 
by  $\pi(x_1,x',\xi_1,\xi'):= (x_1,x',x_1\xi_1,\xi')$ (for more details about $^bT^*\Omega_+$, see, e.g.,~\cite[Section 18.3]{Ho:85}, \cite[Section 4B]{GaSpWu:20}).

\begin{lemma}
\label{l:bvpDefect}
With $u_\h$, $\mu$, $\mu^j$, $\mu^{\rm in}$, $\mu^{\rm out}$, and $\dot{n}^j$ as above, 
\begin{itemize}
\item[(i)] $\supp \mu \subset S^*\Omega_+$
\item[(ii)] For all $\chi \in C_c^\infty(\mathbb{R}^d\setminus \Omega_-)$, $\lim_{\h\to 0}\|\chi u_\h\|^2_{L^2}=\mu(|\chi|^2)$
\item[(iii)] For all $a\in C_c^\infty(^b\!T^*\Omega_+)$, 
$$
\pi_*\mu(a\circ\varphi_t)-\pi_*\mu(a)=\int_0^t \Big(-2\Im \pi_*\mu^j+\delta(x_1)\otimes (\mu^{\rm in}-\mu^{\rm out})+\frac{1}{2}\dfrac{\sigma_\h(\mc{N})}{\sigma_\h(\mc{D})}H_p^2x_1\mu1_{\mc{G}}\Big)(a\circ\varphi_s)\,ds
$$
(where the integral is understood as the integral of distributions acting on smooth functions).
\item[(iv)]
On $\mathcal{H}_{\Gamma_D}$,
$\mu^{\rm out}= \alpha\,\mu^{\rm in}$, where 
\beq\label{eq:alpha}
\alpha:=\left|\frac{\sqrt{1-|\xi'|^2_g}\,\sigma_\h(\mc{N}(x',\xi')-\sigma_\h(\mc{D})(x',\xi')}{\sqrt{1-|\xi'|_g^2}\,\sigma_\h(\mc{N})(x
',\xi')+\sigma_\h(\mc{D})(x',\xi')}\right|^2.
\eeq
\end{itemize}
\end{lemma}

\bpf[References for the proof]
Parts (i) and (ii) are proved in \cite[Lemma 4.2]{GaSpWu:20}. Part (iii) is proved in \cite[Theorem 2.15]{GaLaSp:21} 
(following \cite[Lemma 4.8]{GaSpWu:20}), 
and Part (iv) is proved in \cite[Lemmas 2.12 and 2.18]{GaLaSp:21} (following \cite[Proposition 1.10, Part (iii)]{Mi:00}).
\epf

\section{Parametrix for $(P_\theta-\lambda^2)$ via boundary complex absorption}\label{sec:parametrix}
We now find a parametrix for $(P_\theta-\lambda^2)$ using a complex absorbing potential on the boundary $\Gamma_D$. 
We then obtain by perturbation a parametrix for $(P_\theta-\lambda^2 -z1_{\Otr})$ for $z$ sufficiently small.

First, let 
$$
\mc{P}_\theta(\lambda):=\begin{pmatrix} P_\theta-\lambda^2\\\gamma_0^D\end{pmatrix}:H^{2}(\Omega_+)\to L^2(\Omega_+)\oplus H^{3/2}(\Gamma_D).
$$
Then let $E:H^{3/2}(\Gamma_D)\to H^2(\Omega_+)$ be an extension operator satisfying
$$
\gamma_0^D Eg=g,\qquad g\in H^\frac{3}{2}(\partial\Omega).
$$
Simple calculation then implies that
\begin{equation}
\label{e:invertP}
(\mc{P}_\theta(\lambda))^{-1}= \Big(\mc{R}_\theta(\lambda),\,E-\mc{R}_\theta(\lambda)(P_\theta-\lambda^2)E\Big),
\end{equation}
where $\mc{R}_\theta(\lambda):=(P_\theta-\lambda^2)^{-1}$ is the inverse of~\eqref{e:Fred1}.

\begin{lemma}\label{lem:Fred}
The operator $\mc{P}_\theta(\lambda)$ is Fredholm with index zero.
\end{lemma}
\begin{proof}
Recall that the map \eqref{e:Fred1} is Fredholm with index zero. First, note that if $\mc{P}_\theta(\lambda)u=0$, then $u\in H_0^1(\Omega_+)\cap H^2(\Omega_+)$ and in particular, $u\in \ker (P_\theta-\lambda^2)$. Therefore,
since $P_\theta-\lambda^2: H_0^1(\Omega_+)\cap H^2(\Omega_+)\rightarrow L^2(\Omega_+)$ is Fredholm,
 $\ker \mc{P}_\theta(\lambda)$ is finite dimensional. To see that the cokernel 
$L^2(\Omega_+)\oplus H^{3/2}(\Gamma_D)\big/\mc{P}_\theta(\lambda)H^2(\Omega_+)$
is finite dimensional, define the map
\begin{gather*}
\pi:L^2(\Omega_+)\oplus H^{3/2}(\Gamma_D)\Big/\mc{P}_\theta(\lambda)H^2(\Omega_+)\to L^2(\Omega_+)\Big/(P_\theta-\lambda^2)\big(H_0^1(\Omega_+)\cap H^2(\Omega_+)\big),\\ (f,g)+\mc{P}_\theta(\lambda)H^2(\Omega_+)\mapsto f-(P_\theta-\lambda^2)Eg +(P_\theta-\lambda^2)\big(H_0^1(\Omega_+)\cap H^2(\Omega_+)\big).
\end{gather*}
First, observe that this map is well defined since if $(f_1,g_1)+\mc{P}_\theta(\lambda)H^2(\Omega_+)=(f_2,g_2)+\mc{P}_\theta(\lambda)H^2(\Omega_+)$ then there is $u\in H^2(\Omega_+)$ such that 
$$
(f_1-f_2,g_1-g_2)=\big((P_\theta-\lambda^2)u,\gamma_D u\big).
$$ 
In particular, 
$$
(f_1-f_2)-(P_\theta-\lambda^2)E(g_1-g_2)= (P_\theta-\lambda^2)(u-E(g_1-g_2))\in (P_\theta-\lambda^2)\big(H_0^1(\Omega_+)\cap H^2(\Omega_+)\big),
$$
so $\pi(f_1,g_1)= \pi(f_2,g_2)$.

Now, suppose that $\pi(f,g)=0$. Then, there is $u\in H_0^1(\Omega_+)\cap H^2(\Omega_+)$ such that
$$
f-(P_\theta-\lambda^2)Eg=(P_\theta-\lambda^2)u.
$$
Therefore,
$$
(f,g)-\mc{P}_\theta(\lambda)Eg=(f-(P_\theta-\lambda^2)Eg,0)=((P_\theta-\lambda^2)u,0)\in \mc{P}_\theta(\lambda)H^2(\Omega_+),
$$
and $\pi$ is injective. 
For an injective operator, $\dim (\text{domain})\leq \dim (\text{range})\leq \dim (\text{codomain})$; therefore
$$\dim \Big(L^2(\Omega_+)\oplus H^{3/2}(\Gamma_D)\Big/\mc{P}_\theta(\lambda)H^2(\Omega_+\Big)
\leq \dim \Big(L^2(\Omega_+)\Big/(P_\theta-\lambda^2)\big(H_0^1(\Omega_+)\cap H^2(\Omega_+)\big)\Big)<\infty$$ 
Since $P_\theta-\lambda^2: H_0^1(\Omega_+)\cap H^2(\Omega_+)\rightarrow L^2(\Omega_+)$ is Fredholm, $\mc{P}_\theta(\lambda)$ is Fredholm.
 To see that $\mc{P}_\theta(\lambda)$ has index zero, 
recall that the index is constant in $\lambda$ by, e.g., \cite[Theorem C.5]{DyZw:19}, and
 observe that the formula~\eqref{e:invertP} implies that the inverse exists for some $\lambda$. 
\end{proof}

\

We now define our complex absorbing operator. 
Let $\psi\in C_c^\infty(\mathbb{R};[0,1])$ with $\psi\equiv 1$ on $[-b,b]$ and $\supp\psi\subset [-2b,2b]$. 
It will be convenient to have a specific notation for the Neumann trace with the standard derivative operator replaced by $D:= -\ri \h \partial$. We therefore let  
 $\gamma_{1,\h}^D :=- \ri \h \gamma_1^D$.
Let 
$$
\mc{P}_{\theta,Q}(\lambda):=\begin{pmatrix} P_\theta -\lambda^2\\ Q_b \gamma_{1,\h}^D  +\gamma_0^D\end{pmatrix}:H^2(\Omega_+)\to L^2(\Omega_+)\oplus H^{3/2}(\Gamma_D).
$$
where $Q_b\in \Psi_\h^{\comp}(\Gamma_D)$ with symbol 
\beq\label{eq:Q_bsymbol}
\sigma_\h(Q_b)=-\psi(|\xi'|_g).
\eeq
Note that 
$$
\mc{P}_{\theta,Q}(\lambda)=\mc{P}_\theta(\lambda)+\begin{pmatrix}0\\Q_b \gamma_{1,\h}^D \end{pmatrix}
$$
and hence $\mc{P}_{\theta,Q}(\lambda)$ is a compact perturbation of $\mc{P}_\theta(\lambda)$. Therefore, by Lemma \ref{lem:Fred}, $\mc{P}_Q(\lambda)$ is Fredholm with index zero.

\begin{lemma}
\label{l:absorbingInverse}
Let $Q_b$ be as above and $0<a<b$ and $C_1>0$. Then there exists $C>0$ such that for all $\lambda \in [a,b]+\ri[-C_1\h,C_1\h]$, 
\beq\label{eq:toohot1}
\|\gamma_{1,\h}^Du\|_{L^2(\Gamma_D)}+\N{u}_{H_{\h}^2(\Omega_+)}\leq C\h^{-1}\N{(P_\theta-\lambda^2)u}_{L^2(\Omega_+)}+C\N{(Q_b \gamma_{1,\h}^D +\gamma_0^D)u}_{H_{\h}^{3/2}(\Gamma_D)}.
\eeq
In particular, since $\mc{P}_{\theta,Q}(\lambda)$ is Fredholm with index zero,
\beqs
\mc{R}_{\theta,Q}(\lambda):=(\mc{P}_{\theta,Q}(\lambda))^{-1}
\eeqs
exists and satisfies
\beq\label{eq:highlighter1}
\|\gamma_{1,\h}^D\mc{R}_{\theta,Q}(\lambda)(f,g)\|_{L^2(\Gamma_D)}+\|\mc{R}_{\theta,Q}(\lambda)(f,g)\|_{H_{\h}^2(\Omega_+)}\leq C\Big(h^{-1}\|f\|_{L^2(\Omega_+)}+\|g\|_{H_{\h}^{3/2}(\Gamma_D)}\Big).
\eeq
\end{lemma}

Observe that the bound \eqref{eq:highlighter1} has the same $\h$-dependence as the standard non-trapping resolvent estimate.

Before proving Lemma \ref{l:absorbingInverse} we show how a parametrix for the operator $(P_\theta-\lambda^2 - z1_\Otr)$ can be expressed in terms of $\mc{R}_{\theta,Q}(\lambda)$.
Let 
$$
\mc{P}_{\theta,Q}(\lambda,z):=\begin{pmatrix} P_\theta-\lambda^2-z 1_\Otr \\Q_b \gamma_{1,\h}^D +\gamma_0^D\end{pmatrix}:H^2(\Omega_+)\to L^2(\Omega_+)\oplus H^{3/2}(\Gamma_D).
$$
By Lemma~\ref{l:absorbingInverse}, the bound \eqref{eq:highlighter1}, and inversion by Neumann series, for $|z|\leq \h/(2C)$ (where $C$ is the constant from Lemma~\ref{l:absorbingInverse})
\beqs
\mc{R}_{\theta,Q}(\lambda,z):=(\mc{P}_{\theta,Q}(\lambda,z))^{-1}
\eeqs
 exists and satisfies
\beq\label{eq:highlighter2}
\|\gamma_{1,\h}^D\mc{R}_{\theta,Q}(\lambda,z)(f,g)\|_{L^2(\Gamma_D)}+\|\mc{R}_{\theta,Q}(\lambda,z)(f,g)\|_{H^2_h(\Omega_+)}\leq 2C\Big(h^{-1}\|f\|_{L^2(\Omega_+)}+\|g\|_{H^{3/2}_h(\Gamma_D)}\Big).
\eeq

Next, let 
\beq\label{eq:Ptheta}
\mc{P}_\theta(\lambda,z):=\begin{pmatrix} P_\theta-\lambda^2-z 1_\Otr \\\gamma_0^D\end{pmatrix}:H^2(\Omega_+)\to L^2(\Omega_+)\oplus H_{\h}^{3/2}(\Gamma_D).
\eeq
If $\mc{R}_{\theta,Q}(\lambda,z)$ exists, then 
\beqs
\mc{P}_\theta(\lambda,z) = \big( I + K(\lambda,z)\big)\mc{P}_{\theta,Q}(\lambda,z),
\eeqs
where
\begin{equation}
\label{e:compactPerturb}
K(\lambda,z):=Q\mc{R}_{\theta,Q}(\lambda,z)\quad\tand\quad Q:=\begin{pmatrix} 0\\\ -Q_b \gamma_{1,\h}^D \end{pmatrix}.
\end{equation}
Since $K(\lambda,z):L^2(\Omega_+)\oplus H^{3/2}(\Gamma_D)\to L^2(\Omega_+)\oplus H^{3/2}(\Gamma_D)$ is compact,
$(I+K(\lambda,z))^{-1}$ is a meromorphic family of operators by \cite[Theorem C.8]{DyZw:19}.
Therefore, for $|z|\leq \h/(2C)$,
\beq\label{eq:Thursday3}
\mc{P}_{\theta}(\lambda,z)^{-1}=\mc{R}_{\theta,Q}(\lambda,z)\big(I+K(\lambda,z)\big)^{-1}.
\eeq
Let $R_\theta(\lambda,z)$ be the inverse of the map~\eqref{e:Fred2} with $V=-z1_{\Otr} $, i.e.
\beq\label{eq:Rtheta}
R_\theta(\lambda,z) := (P_\theta-\lambda^2 - z1_{\Otr})^{-1}.
\eeq
Then, for $|z|\leq \h/(2C)$,
\begin{equation}
\label{e:parametrix}
R_\theta(\lambda,z)=\mc{P}_\theta(\lambda,z)^{-1}\begin{pmatrix}I\\0\end{pmatrix}=\mc{R}_{\theta,Q}(\lambda,z)\big(I +K(\lambda,z)\big)^{-1}\begin{pmatrix}I\\0\end{pmatrix},
\end{equation}
which is the required parametrix.

\

\begin{proof}[Proof of Lemma \ref{l:absorbingInverse}]
Suppose that the estimate \eqref{eq:toohot1} fails with the left-hand side replaced by $\|u\|_{L^2(\Omega_+)}$, then there are $\h_n\to 0$, $\lambda_n\in [a,b]+\ri[-C\h,C\h]$, and $(\tilde{u}_n,\tilde{f}_n,\tilde{g}_n)\in H^2(\Omega_+)\oplus L^2(\Omega_+)\oplus H_{\h}^{3/2}(\Gamma_D)$ with 
$$
\big\|\tilde{f}_n\big\|_{L^2(\Omega_+)}+\N{\tilde{g}_n }_{H_{\h}^{3/2}(\Gamma_D)}=1,\qquad \N{\tilde{u}_n}_{L^2}=n,
$$
and with 
$$
\mc{P}_{\theta,Q}(\tilde{u}_n)=(\h_n \tilde{f}_n ,\tilde{g}_n ).
$$
In particular, renormalizing $u_n:=\tilde{u}_n/n$, $f_n:= \tilde{f}_n/n$, and $g_n:= \tilde{g}_n/n$,
\beqs
\N{f_n}_{L^2(\Omega_+)}= \h^{-1}\N{(P_\theta-\lambda_n^2)u_n}_{L^2(\Omega_+)}\leq \frac{1}{n}
\eeqs
and
\beqs
\N{g_n}_{L^2(\Gamma_D)}= \N{(Q_b \gamma_{1,\h}^D +\gamma_D)u_n}_{L^2(\Gamma_D)}\leq \frac{1}{n}.
\eeqs
Now, since $ 0<a\leq\Re \lambda_n\leq b$, we may rescale $\h_n$ to $\tilde{\h}:=\h_n/\Re \lambda_n$ and hence replace $\Re \lambda_n$ by $1$. Note that this rescaling does not cause any issues since $b^{-1} \h_n\leq \tilde{\h}_n\leq a^{-1}\h_n.$ Extracting a subsequence, we can assume that $1^{\rm{ext}}_{\Omega_+}u_n $ has defect measure $\mu$ (see e.g.~\cite[Theorem 5.2]{Zw:12}) and $\h_n^{-1}\Im \lambda_n\to \Im \beta_\infty$, and $ \Re \lambda _n=1$. Since $\|f_{\h_n}\|_{L^2} \tendo$, $\mu^j=0$.

Let $\chi,\chi_0 \in C_c^\infty(\mathbb{R}^d;[0,1])$ with $\supp \chi \Subset \Rea^d$ and $\chi,\chi_0 \equiv 1$ in a neighborhood of $\{\r \leq 2r_1\}$ and $\supp \chi_0\subset \{\chi \equiv 1\}$.  
We first show that 
\begin{equation}
\label{e:exterior}
\N{(1-\chi)u_n}_{L^2(\Omega_+)}=\cO(\h_n).
\end{equation}
 To do this, observe that, by \eqref{eq:symbolPtheta},
\beq\label{eq:elliptic_Friday2}
|\sigma_\h(P_\theta-\lambda_n^2)(x,\xi)|=\left|\frac{|\xi|^2}{(1+\ri\tan \theta)^2}-1\right|\geq c(|\xi|^2+1),\qquad \r(x,\xi)\geq 2r_1.
\eeq
Therefore, by ellipticity, for $W$ a neighborhood of $\supp \partial \chi$,
\beq\label{eq:elliptic_Friday}
\|u_n\|_{H_{\h}^2(W)}\leq C\big(\N{(P_\theta-\lambda^2)u_n}_{L^2(\Omega_+)}+\N{u_n}_{L^2(\Omega_+)}\big).
\eeq
Now, by \eqref{eq:elliptic_Friday2} and the definitions of $\chi$ and $\chi_0$,
$$
\Big|\sigma\Big(\Oph\big((1+|\xi|^2)^{-1}\big)(1-\chi_0)(P_\theta-\lambda_n^2)(1-\chi_0)-\ri \chi_0\Big)\Big|\geq c.
$$
Therefore, by~\cite[Theorem 4.29]{Zw:12},
\begin{align}\nonumber
&\|(1-\chi)u_n\|_{L^2(\Omega_+))}\\ \nonumber
&\qquad\leq C\N{\big[\Oph((1+|\xi|^2)^{-1})(1-\chi_0)(P_\theta-\lambda_n^2)(1-\chi_0)-\ri\chi_0\big](1-\chi)u_n}_{L^2(\mathbb{R}^d)}\\
&\qquad=C\N{\Oph((1+|\xi|^2)^{-1})(1-\chi_0)(P_\theta-\lambda_n^2)(1-\chi)u_n}_{L^2(\Rea^d)} .\label{eq:rain1}
\end{align}
But, 
\begin{align}\nonumber
&\N{\Oph((1+|\xi|^2)^{-1})(1-\chi_0)(P_\theta-\lambda_n^2)(1-\chi)u_n}_{L^2(\Rea^d)}, \\ \nonumber
&\qquad\leq C\N{(1-\chi)\h_nf_n}_{L^2(\Omega_+)}+\N{[P_\theta,\chi]u_n}_{H_{\h}^{-2}(\Omega_+)},\\
&\qquad\leq C\N{(1-\chi)\h_nf_n}_{L^2(\Omega_+)}+C\h_n \N{u_n}_{L^2(\Omega_+)}=\cO(\h_n).\label{eq:rain2}
\end{align}
where we have used that, by direct computation, $\|[P_\theta,\chi]\|_{H_{\h}^s(\Omega_+)\to H_{\h}^{s-1}(\Omega_+)}\leq C\h$ in the second inequality; \eqref{e:exterior} then follows from combining \eqref{eq:rain1} and \eqref{eq:rain2}.

We now show that $\mu(T^*\mathbb{R}^d)=1$. First, observe that
\beq\label{eq:reorder1}
(P_\theta-\lambda_n^2)\chi u_n= [P_\theta,\chi]u_n+o(\h_n)_{L^2}.
\eeq
Consequently, using \eqref{eq:elliptic_Friday} in \eqref{eq:reorder1} we find that 
\beqs
(P_\theta-\lambda_n^2)\chi u_n=\cO(\h_n)_{L^2}.
\eeqs
Since $(P_\theta -\lambda^2) = (-\h^2 \Delta -\lambda^2)$ on $\supp \chi$, we can now apply Lemma~\ref{l:bvpDefect} (with $u$ in that lemma replaced by $\chi u_n$ here) to find that
$$
\mu(\chi^2)=\lim_{\h\to 0}\N{\chi u_n}_{L^2(\Omega_+)}^2=\lim_{\h \to 0}\N{u_n}_{L^2(\Omega_+)}^2=1,
$$
where we have used \eqref{e:exterior} in the second equality. Moreover, 
$$
\mu(T^*\mathbb{R}^d)\leq \lim_{\h \to 0}\N{u_n}_{L^2(\Omega_+)}^2=1,
$$
so that in fact $\mu(T^*\mathbb{R}^d)=1$.

We now show that $\mu=0$ which is a contradiction. To do this, we start by observing that~\eqref{e:exterior} implies that $\mu(\{\r\geq 2r_1\})=0$. In fact, by  Lemma~\ref{l:ellipticDefect}, $\mu (\{\sigma_{\h}(P_\theta)\neq0\})=0$, and therefore, $\supp \mu \subset S^*_{\overline{\Omega_+}}\Rea^d\cap\{\r\leq  2r_1\} $.

 Now, Lemma~\ref{l:bvpDefect}, along with Lemma~\ref{l:propagateDefect} together with the fact that $\Im \sigma_\h( P_\theta)\leq 0$, allows us to propagate forward along the generalized bicharacteristic flow (in the sense of Corollary \ref{cor:forwards}), but not backward. In particular, since $\mu(\{\r\geq 2r_1\}=0)$, this implies that $\supp \mu\subset \Gamma_+$. Indeed, suppose that $A\subset S^*_{\overline{\Omega_+}}\mathbb{R}^d$ is compact and $A\cap \Gamma_+=\emptyset$. 
Then, by the definition of $\Gamma_+$~\eqref{e:trappedSets}, for each $\rho\in A$ there is $t_\rho>0$ such that ${\bf{r}}(\varphi_{-t_\rho}(\rho))> \max(2r_1,{\bf{r}}(\rho))$. Hence, by~\eqref{e:directEscape1} for $t\geq t_\rho$, ${\bf{r}}(\varphi_{-t}(\rho))>2r_1$ and by continuity of $\varphi_{-t_\rho}$, there is a neighborhood $U_\rho$ of $\rho$  such that $\varphi_{-t}(U_\rho)\subset \{{\bf{r}}> 2r_1\}$ for $t\geq t_\rho$. In particular, by compactness of $A$, there is $T>0$ such that $\varphi_{-T}(A)\subset \{{\bf{r}}>2r_1\}$. By compactness of $A$ in the $\xi$ variable and \eqref{eq:forwards}, there is $C>0$ such that $\mu(A)\leq \exp{(CT)}\mu(\varphi_{-T}(A))=0$. Now, by Lemma~\ref{l:trapping}, $\Gamma_+$ is closed and hence we may write $(\Gamma_+)^c=\cup_nA_n$ with $A_n$ compact. In particular, $\mu((\Gamma_+)^c)=0$ by monotonicity from below. 

Next, note that since $\Im \sigma_\h(P_\theta-\lambda^2)<0$ on $\{f_\theta\neq 0\}$,
$$
\supp \mu \subset \{f_\theta =0\}
$$
by Lemma \ref{l:ellipticDefect}.
In particular, by the definition of $f_\theta$,
$$
\supp \mu\subset \{\r < 2r_1\}.
$$
To complete the proof, we need to show that in fact $\mu(\Gamma_+)=0$. This is where the boundary term $Q_b$ is used.

We claim there are $T,c>0$ such that 
\begin{equation}
\label{e:backward}
\mu(\varphi_{-T}(A))\geq \re^c\mu(A)
\end{equation}
for all $A$.
 Once this is done, we have that $\mu\equiv 0$. To see this, observe that if $\mu(A)>0$, then by induction $\mu(\varphi_{-nT}(A))\geq \re^{nc}\mu(A)$. Taking $N> -(\log \mu(A))/c$, we have $\mu(\varphi_{-NT}(A))>1$, which is a contradiction to $\mu(T^*\Rea^d)=1$.
 
We now prove~\eqref{e:backward}. First, note that the statement is empty if $\mu(A)=0$. Therefore, we can assume that $\mu(A)>0$. 
Since $\supp \mu \subset \Gamma_+$, we assume that $A\subset \Gamma_+$; since $\Gamma_+$ is closed, we can assume that $A$ is compact. Now, by~\eqref{e:directEscape1}, \eqref{e:directEscape2}, and \eqref{e:trappedSets},
$$
\Gamma_+\cap \{r_0\leq \r\leq 2r_1\}\subset \bigcup_{t=0}^{{\sqrt{(2r_1)^2-r_0^2}}}\varphi_t(\Gamma_+\cap \{\r\leq  r_0\}).
$$
Therefore, increasing $T$ by ${\sqrt{(2r_1)^2-r_0^2}}$, we may assume that $A\subset \{\r< r_0\}\cap \Gamma_+$.

Letting $\mc{N}=Q_b$ and $\mc{D}=-1$ and recalling \eqref{eq:Q_bsymbol}, we see that $\mc{N}$ and $\mc{D}$ satisfy~\eqref{e:boundaryCond}. Therefore, the proof of~\eqref{e:backward} is completed by the next lemma.

\begin{lemma}\label{lem:hayfever}
Suppose that $\mc{N}$ and $\mc{D}$ are as in~\eqref{e:boundaryCond}, $\mu$ satisfies the conclusions of Parts (iii) and (iv) of Lemma \ref{l:bvpDefect} with $\mu^j=0$, and $A\subset \{\r<r_0\}\cap \Gamma_+$. Then there are $T,c>0$ such that~\eqref{e:backward} holds.
\end{lemma}

\begin{proof}
We claim that there are $\e_1,T>0$ such that for all $\rho\in \Gamma_+$ with $\r(\rho)<r_0$.
\begin{equation}
\label{e:rhoClaim}
\int_0^T \left(-\dfrac{1}{2}\dfrac{\sigma_\h(\mc{N})}{\sigma_\h(\mc{D})}H_p^2x_1 1_{\mc{G}}(\varphi_{-t}(\rho))+|H_px_1(\varphi_{-t}(\rho))|^{-1}\delta\big(x_1(\varphi_{-t}(\rho))\big)\log \alpha\big(\pi_{\Gamma_D}(\varphi_{-t}(\rho))\big)
\right)d t
\leq -\e_1.
\end{equation}
where $\pi_{\Gamma_D}:S_{\Gamma_D}^*{\mathbb{R}^d}\to T^*\Gamma_D$ is the orthogonal projection and $\alpha$ is given by \eqref{eq:alpha}.

Once~\eqref{e:rhoClaim} is proved, we claim that Lemma~\ref{l:bvpDefect} implies~\eqref{e:backward} with $(c,T)=(\e_1,T)$. Indeed, suppose that~\eqref{e:rhoClaim} holds and that $\mu(A)>0$, $A\subset \Gamma_+\cap \{\r<r_0\}$ and $A$ is closed. Then, let $0\leq a\in C_c^\infty(^bT^*\mathbb{R}^d\setminus \Omega_-)$ with $a\equiv 1$ on $A$ and 
\begin{equation}
\label{e:backFlow}
\bigcup_{t\in[0,-T]}\varphi_t(\supp a)\subset \left\{\r<\frac{r_0+r_1}{2}\right\}.
\end{equation}

Now, let $\chi\equiv 1$ on $\{\r\leq \frac{r_0+r_1}{2}\}$ with $\supp \chi \subset \{\r <r_1\}$. Then, 
$$
(-\h^2\Delta-1)\chi u= [-\h^2\Delta,\chi]u+\h f,\qquad \chi u|_{\Gamma_D}=0
$$
with $f=o(1)_{L^2}$ and hence by Lemma~\ref{l:bvpDefect}
$$
\pi_*\mu(\chi^2 (a\circ\varphi_t))-\pi_*\mu(\chi^2 a)=\int_0^t \Big(-4\big\langle \xi,\partial\chi\big\rangle\mu+\delta(x_1)\otimes (\mu^{\rm in}-\mu^{\rm out})+\dfrac{\sigma_\h(\mc{N})}{2\sigma_\h(\mc{D})}H_p^2x_1\mu1_{\mc{G}}\Big)(\chi^2( a\circ\varphi_s))\,ds.
$$
But, by~\eqref{e:backFlow}, $\chi^2\equiv 1$ on $\supp a\circ\varphi_t$ for $t\in [0,T]$. In particular, for $t\in[0,T]$
\begin{equation*}
\pi_*\mu(a\circ\varphi_t)-\pi_*\mu( a)=\int_0^t \Big(\delta(x_1)\otimes (\mu^{\rm in}-\mu^{\rm out})+\dfrac{\sigma_\h(\mc{N})}{2\sigma_\h(\mc{D})}H_p^2x_1\mu1_{\mc{G}}\Big)( a\circ\varphi_s)\,ds.
\end{equation*}
Finally, since $A$ is closed we may approximate $1_{A}$ by smooth, compactly supported functions to obtain
\begin{equation}
\label{e:toStudy}
\pi_*\mu(
{\varphi_{-t}(A)}
)-\pi_*\mu(A)=\int_0^t \Big(\delta(x_1)\otimes (\mu^{\rm in}-\mu^{\rm out})+\dfrac{\sigma_\h(\mc{N})}{2\sigma_\h(\mc{D})}H_p^2x_1\mu1_{\mc{G}}\Big)( 1_A\circ\varphi_s)\,ds.
\end{equation}
Now, to study~\eqref{e:toStudy}, we first assume that $A$ and $t$ are such that for all $\rho\in A$ and $s\in[0,2t]$, $\varphi_{-s}(\rho)$ does not lie in the glancing region ($H_px_1=0$) and each trajectory intersects $\Gamma_D$ exactly once and does so for $s\in (0,t)$. Shrinking the support of $a$ further if necessary, we can find $\Sigma\subset ^bT^*\mathbb{R}^d\setminus \Omega_-$ transverse to the vector field $H_p$ such that 
$$
F:[-t,t]\times\Sigma\,\ni(s,\rho)\mapsto \varphi_{-s}(\rho)\in ^bT_{\Gamma_D}^*\mathbb{R}^d
$$
are smooth coordinates and $\varphi_{-s}(A)$ is in the image of $F$ for all $s\in[0,t]$. Then,~\eqref{e:toStudy} reads
\begin{equation*}
\begin{aligned}
&\pi_*\mu(
{\varphi_{-t}(A)}
)-\pi_*\mu(A)\\
&=\int_0^t\Big(\delta(x_1)\otimes (\mu^{\rm in}-\mu^{\rm out})\Big)(1_A\circ \varphi_{t'})\,dt'\\
&=\int_0^t\int_{-t}^t\int_{\Sigma}\Big(|H_px_1|(s,\rho)\delta(s)\otimes (1_A(s-t',\rho))\Big)d(\mu^{\rm in}-\mu^{\rm out})(\rho)\,ds \,dt'\\
&=\int_{\Sigma}\int_0^{t} \Big(|H_px_1|(0,\rho)(\alpha^{-1}(\rho)-1)1_A(-t',\rho)\Big)d\mu^{\rm out}(\rho)dt'
\end{aligned}
\end{equation*}
Now, arguing as in~\cite[Lemma 2.16]{GaLaSp:21}, we obtain that $\pi_*\mu=  |H_px_1|\mu^{\rm out}1_{s<0}ds+|H_px_1|\mu^{\rm in}1_{s>0}ds$ and hence, 
$$
\pi_*\mu(A)=\int_\Sigma\int_0^t |H_px_1|(0,\rho)1_A(-t',\rho)d\mu^{\rm out}(\rho)dt'.
$$
Therefore
\begin{align*}
\pi_*\mu(\varphi_{-t}(A))&\geq \inf_{F^{-1}(A)}(\alpha(\rho))^{-1}\pi_*\mu(A)\\
&= \inf_{A}\re^{-\int_0^t(|H_px_1|(\varphi_{-t'}(\rho))^{-1}\delta(x_1(\varphi_{-t'}(\rho)))\log \alpha(\pi_{\Gamma_D}(\varphi_{-t'}(\rho)))dt'}\pi_*\mu(A),
\end{align*}
where this last equality comes from evaluating the integral using the fact that $F$ is well-defined (since each trajectory intersects $\Gamma_D$ exactly once).
\beqs
A\subset\Big\{ \varphi_{s}\big(\{x_1=H_px_1=0\}\big)\setminus \big\{H_px_1\neq 0,\,x_1=0\big\} \, :\, s\in[0,t]\Big\},
\eeqs
 so that, in particular, trajectories from $A$ do not intersect the hyperbolic set. In this case, \eqref{e:toStudy} implies that
\begin{equation}
\label{e:ac}
\partial_s \pi_*\mu(\varphi_{-s}(A))=\Big(\dfrac{\sigma_\h(\mc{N})}{2\sigma_\h(\mc{D})}H_p^2x_11_{\mc{G}}\pi_*\mu\Big)\big(\varphi_{-s}(A)\big),
\end{equation}
In particular, shrinking $A$ if necessary, we may choose $\Sigma \subset \{x_1=H_px_1=0\}$ transverse to $H_p$ 
and work in coordinates
$$
[0,t]\times \Sigma \ni(s,\rho)\mapsto \varphi_{-s}(\rho)\in \Big\{ \varphi_{-s}\big(\{x_1=H_px_1=0\}\big)\,:\, s\in[0,t]\Big\}.
$$
In these coordinates,~\eqref{e:ac} implies that $\pi_*\mu$ is absolutely continuous with respect to $t$ in the sense that there is a family of measures, $t\mapsto \nu_t$  on $\Sigma$ such that $\nu_t(\Sigma)\in L^1$ and $\mu=\nu_tdt$. Moreover,
$$
\int_B d\nu_s(\rho)=\int_B \exp\left(\int \dfrac{\sigma_\h(\mc{N})}{2\sigma_\h(\mc{D})}H_p^2x_11_{\mc{G}}(\varphi_{-s}(\rho))ds\right)d\nu(\rho).
$$
In particular, 
$$
\pi_*\mu(\varphi_{-t}(A))\geq \inf_{A} \exp\left(\int \dfrac{\sigma_\h(\mc{N})}{2\sigma_\h(\mc{D})}x_11_{\mc{G}}(\varphi_{-s}(\rho))ds\right)\pi_*\mu(A)
$$
Putting everything together, we have for all $A$ and $0\leq t\leq T$, 
\begin{align*}
&\pi_*\mu(\varphi_{-t}(A))\\
&\geq \inf_A \exp\Big(-\int_0^t(|H_px_1|(\varphi_{-t}(\rho))^{-1}\delta(x_1(\varphi_{-t}(\rho)))\log \alpha(\pi_{\Gamma_D}(\varphi_{t}(\rho)))-\dfrac{\sigma_\h(\mc{N})}{2\sigma_\h(\mc{D})}x_11_{\mc{G}}(\varphi_{-s}(\rho)))ds\Big)\pi_*\mu(A)\\
&\geq \re^{\e_1}\pi_*\mu(A)
\end{align*}
as claimed.

Therefore, it is enough to prove~\eqref{e:rhoClaim}. Seeking a contradiction, we assume that for every $\e_1>0$ and $T>0$ there is $\rho\in \Gamma_+$ with $\r(\rho)<r_0$ such that 
\beq\label{eq:contradictclaim}
\int_0^T \left(-\dfrac{\sigma_\h(\mc{N})}{2\sigma_\h(\mc{D})}H_p^2x_11_{\mc{G}}(\varphi_{-t}(\rho))+|H_px_1(\varphi_{-t}(\rho))|^{-1}\delta(x_1(\varphi_{-t}(\rho)))\log \alpha(\pi_{\Gamma_D}(\varphi_{-t}(\rho)))
\right) d t\geq -\e_1.
\eeq
 Note that since both terms are non-positive (since $\alpha \leq 1$ and $\sigma_\h(\mc{N})\sigma_\h(\mc{D})>0$), this implies that each term is $\geq -\e_1$. 
 
Now, if $\varphi_{-t}(\rho)\in\mc{G}$ for $t\in [t_1,t_2]$, then, since the flow in $\mc{G}$ is given by the flow of the vector field
$$
H_p^G:=H_p+\frac{H^2_p x_1}{H_{x_1}^2p}H_{x_1},\qquad p=|\xi|^2-1,
$$
(see \cite[Def.~24.3.6]{Ho:85}),
we obtain, using that $\sigma_\h(\mc{N})/\sigma_\h(\mc{D})>c>0$ on $\mc{G}$ (since $\mc{G} \subset \overline{B^* \Gamma_D}$).
\begin{align*}
\varphi_{-t_2}(\rho)&=\exp(-(t_2-t_1)H_{|\xi|^2}(\rho)) +\cO\left(\int_{t_1}^{t_2}H_p^2x_1(\varphi_{-t}(\rho))dt\right)\\
&=\exp(-(t_2-t_1)H_{|\xi|^2}(\rho)) +\cO\left(\int_{t_1}^{t_2}\dfrac{\sigma_\h(\mc{N})}{2\sigma_\h(\mc{D})}H_p^2x_1(\varphi_{-t}(\rho))dt\right),
\end{align*}
where both here and in the rest of this proof we write $a= b + \mc{O}(c)$ if $|a-b| \leq C c$ for some $C>0$ depending only on $\sigma_\h(\mc{N})$ and $\sigma_\h(\mc{D})$.
On the other hand, if $\varphi_{-t}(\rho)\notin\mc{G}$ for $t\in [t_1,t_2]$, and has exactly one intersection with $\Gamma_D$, then
$$
\varphi_{-t_2}(\rho)=\exp(-(t_2-t_1)H_{|\xi|^2}(\varphi_{-t_1}(\rho))+\cO\left(|t_2-t_1|2\sqrt{1-|\xi_r'|_g^2}\right).
$$
where $|\xi_r'|_g$ is measured at the point of reflection. All together, since $\sigma_\h(\mc{N})\sigma_\h(\mc{D})>c>0 $ on $|\xi'|_{g}\leq 1$, and thus there is $c>0$ such that
$$
\log \alpha= -4\sqrt{1-|\xi'|_g^2}\frac{\sigma_\h(\mc{N})}{\sigma_\h(\mc{D})}+\cO\big(1-|\xi'|_g^2\big)\leq  -c\sqrt{1-|\xi'|_g^2}\frac{\sigma_\h(\mc{N})}{\sigma_\h(\mc{D})}+\cO\big(1-|\xi'|_g^2\big), 
$$
we obtain from \eqref{eq:contradictclaim} that
$$
\varphi_{-T}(\rho)=\exp(-TH_{|\xi|^2}(\rho))+\cO(\e_1)
$$
Therefore, choosing $T\gg r_0$, and $\e_1$ small enough, we obtain
$$
{\rm dist}(\projx(\varphi_{-T}(\rho)),\projx(\rho))>3r_0
$$
which is a contradiction to $\rho\in \Gamma_+\cap \{\r \leq r_0\}$.
\end{proof}

We have therefore proved that
\beq\label{eq:ruler1}
\N{u}_{L^2(\Omega_+)}\leq C\h^{-1}\N{(P_\theta-\lambda^2)u}_{L^2(\Omega_+)}+C\N{(Q_b \gamma_{1,\h}^D +\gamma_0^D)u}_{H_{\h}^{3/2}(\Gamma_D)}.
\eeq
where here, and in the rest of the proof, $C$ denotes a constant, independent of $\h$, $\lambda$, and $z$, whose value may change from line to line.
To complete the proof of Lemma \ref{l:absorbingInverse}, we now need to obtain a bound on the $H^2_h$ norm of $u$, as opposed to just the $L^2$ norm in \eqref{eq:ruler1}. 
By a standard elliptic parametrix construction, for $\chi _1\in C^\infty(\overline{\Omega_+})$ supported away from $\Gamma_D$, we have 
\begin{align*}
\N{\chi_1 u}_{H_{\h}^2(\Omega_+)}&\leq C\N{(P_\theta-\lambda^2)u}_{L^2(\Omega_+)}+C\N{u}_{L^2(\Omega_+)}\\
&\leq  C\h^{-1}\N{(P_\theta-\lambda^2)u}_{L^2(\Omega_+)}+C\N{(Q_b \gamma_{1,\h}^D +\gamma_0^D)u}_{H_{\h}^{3/2}(\Gamma_D)},
\end{align*}
by \eqref{eq:ruler1}.
Finally, using the trace estimate from \cite[Corollary 4.2]{GaLaSp:21} we have for $\chi_2 \in C^\infty(\{x:\r\leq r_0\})$ with $\supp \chi_2 \Subset \Rea^d$, 
$$
\N{ \gamma_{1,\h}^D  u}_{L^2(\Gamma_D)}\leq C\N{\chi_2 u}_{L^2(\Omega_+)}+\N{(-\h^2\Delta-1) \chi_2 u}_{L^2(\Omega_+)}.
$$
Elliptic regularity for the Laplacian 
then implies that
\begin{align*}
\N{\chi_2 u}_{H_{\h}^2(\Omega_+)}&\leq C\N{(-\h^2\Delta-\lambda^2)\chi_2 u}_{L^2}+ C\N{\chi u}_{L^2}+C\N{\gamma_0^D u}_{H_{\h}^{3/2}(\Gamma_D)}\\
&\leq C\h^{-1}\N{(P_\theta-\lambda^2) u}_{L^2}+C\N{(Q_b \gamma_{1,\h}^D +\gamma_0^D)u}_{H_{\h}^{3/2}(\Gamma_D)},
\end{align*}
where we have used \eqref{eq:ruler1}. Combining the bounds on $\|\chi_1 u\|_{H^2_h(\Omega_+)}$, $\|\chi_2 u\|_{H^2_h(\Omega_+)}$, {and $\| \gamma_{1,\h}^D  u\|_{L^2(\Gamma_D)}$}, we obtain \eqref{eq:highlighter1}.
\end{proof}

\section{Proof of Lemma \ref{l:mainEstimate}}\label{sec:proofbounds}

With $R(\lambda,z)$ defined by 
\eqref{eq:Rdef}, $R_\theta(\lambda,z)$ defined by 
\eqref{eq:Rtheta}, and $\chi\in C^\infty$ with $\supp\chi \subset  \{x : \r\le r_1\}$ and $\supp \chi \Subset \Rea^d$, \eqref{e:resolveAgree} implies that 

\begin{equation}\label{e:scaledUnscaled}
\chi R_\theta(\lambda,z)\chi=\chi R (\lambda,z)\chi.
\end{equation}
Recalling \eqref{eq:twoRs}, we see that to prove the bounds \eqref{eq:bound1new}, \eqref{eq:bound2new} it is sufficient to bound 
$$
\|R_\theta(\lambda,z)\|_{L^2(\Otr)\to L^2(\Otr)}.
$$
We first focus on proving the bound for $\Im z>0$ \eqref{eq:bound2new}. By the definitions of $\mc{P}_\theta(\lambda,z)$ \eqref{eq:Ptheta} and $R_\theta(\lambda,z)$ \eqref{eq:Rtheta}, the bound \eqref{eq:bound2new} follows if we can prove the following.

\begin{lemma}
\label{l:upperHalf}
There exists $C>0$ such that 
if $\Re\lambda>0$, $\Im\lambda =0$, 
\beq\label{eq:Thursday4}
\|\mc{P}_\theta(\lambda,z)^{-1}\|_{L^2(\Otr)\otimes H^{3/2}_\h(\Gamma_D) \rightarrow L^2(\Otr)}\leq C \langle z\rangle (\Im z)^{-1} \quad\tfor \Im z>0.
\eeq
Moreover, there exists $\e>0$ small enough such that if $\Re\lambda>0$, $\Im\lambda =0$, 
\beq\label{eq:Thursday4a}
\|\mc{P}_\theta(\lambda,z)^{-1}\|_{L^2(\Omega_+)\otimes H^{3/2}_\h(\Gamma_D) \rightarrow H_{\h}^2(\Omega_+)}\leq C (\Im z)^{-1} \quad\tfor \Im z >0 \tand |z|\leq \e \h.
\eeq
\end{lemma}

To prove Lemma \ref{l:upperHalf}, we need the following result about the sign of the Dirichlet-to-Neumann map.
\begin{lemma}
For $\Re\lambda>0$, and $\Im \lambda\geq 0$ we have $\Im \mc{D}(\lambda/\h)\geq 0.$
\end{lemma}
\begin{proof}
Let $G(\lambda)$ be the meromorphic continuation from $\Im \lambda>0$ of the solution operator satisfying
\begin{equation*}
(-\h^2\Delta-\lambda^2)G(\lambda)g=0\,\,\text{in }\mathbb{R}^d\setminus \overline{\Omega_1},\qquad G(\lambda)g|_{\Gamma_{\tr}}=g,
\end{equation*}
and $G$ is $\lambda/\h$-outgoing; then $\mc{D}(\lambda/\h)=\gamma_1^{\tr} G(\lambda)$. 
Note that for $\Im \lambda>0$, $G(\lambda):H^{1/2}(\Gamma_{\tr})\to H^{1}(\mathbb{R}^d\setminus \Omega_1)$. Therefore, for $\Re\lambda>0$ and $\Im \lambda>0$, by integration by parts, 
\begin{align*}
0&=\big\langle (-\h^2\Delta-\lambda^2)G(\lambda)g,G(\lambda)g\big\rangle_{\mathbb{R}^d\setminus \Omega_1}\\
&=\|h\nabla G(\lambda)g\|_{L^2(\mathbb{R}^d\setminus \Omega_1)}^2- \lambda^2\|G(\lambda)g\|^2_{L^2(\mathbb{R}^d\setminus \Omega_1)}+\h^2\big\langle \mc{D}(\lambda/\h)g,g\big\rangle_{\Gamma_{\tr}} .
\end{align*}
Therefore, taking imaginary parts
$$
2\Re \lambda \Im \lambda \|G(\lambda)g\|_{L^2(\mathbb{R}^d\setminus \Omega_1)}=\h^2\Im\langle \mc{D}(\lambda/\h)g,g\rangle_{\Gamma_{\tr}}
$$
and in particular, for $\Re \lambda>0$, $\Im\lambda>0$
$$
0\leq\Im\langle \mc{D}(\lambda/\h)g,g\rangle_{\Gamma_{\tr}}
$$
Now, since the right hand side continues analytically from $\Im \lambda>0$ to $\Im\lambda =0$, we have 
$$
\Im \langle \mc{D}(\lambda/\h)g,g\rangle_{\Gamma_{\tr}} \geq 0
$$
for $\Re\lambda>0$ and $\Im \lambda =0$.
\end{proof}

\

\begin{proof}[Proof of Lemma \ref{l:upperHalf}]
Let $u\in H_{\loc}^2(\Omega_+)$. Then, let $v=u-E\gamma_D u\in H_{\loc}^2(\Omega_+)\cap H_{0,\loc}^1(\Omega_+)$.  By integration by parts,
\begin{align*}
-\Im \big\langle (P_\theta -\lambda^2 -z1_{\Otr})v,v\big\rangle_{\Omega_{\tr}}&=-\Im \big\langle (-\h^2 \Delta-\lambda^2 -z1_{\Otr})v,v\big\rangle_{\Omega_{\tr}}\\
&=(\Im z)\|v\|_{L^2(\Otr)}^2+ \h^2\Im \big\langle \mc{D}(\lambda/\h)v,v\big\rangle_{\Gamma_{\tr}}\geq (\Im z)\|v\|_{L^2(\Otr)}^2.
\end{align*}
Therefore, there exists $C, C_1, C_2>0$ such that for $\Im z>0$,
\begin{align*}
\|u\|_{L^2(\Omega_{\tr})}&\leq \|v\|_{L^2(\Omega_{\tr})}+\|E\gamma_0^D u\|_{L^2(\Omega_{\tr})}\\
&\leq (\Im z)^{-1}\|(-\h^2\Delta-\lambda^2-z1_{\Otr})v\|_{L^2(\Omega_{\tr})}+C_1\|\gamma_0^D u\|_{H_{\h}^{3/2}(\Gamma_D)}\\
&\leq  (\Im z)^{-1}\|(P_\theta-\lambda^2-z1_{\Otr})u\|_{L^2(\Otr)}+C_2 \langle z\rangle(\Im z)^{-1}\|E\gamma_0^D u\|_{H_{\h}^2(\Otr)}+C_1\|\gamma_0^D u\|_{H_{\h}^{3/2}(\Gamma_D)} \\
&\leq C\langle z\rangle(\Im z)^{-1}\N{\mc{P}_\theta(\lambda,z)}_{L^2(\Otr)\oplus H_{\h}^{3/2}(\Gamma_D)},
\end{align*}
by the definition of $\mc{P}_\theta(\lambda,z)$ \eqref{eq:Ptheta}.
Having obtained the bound \eqref{eq:Thursday4} on $\|u\|_{L^2(\Otr)}$, we now prove the bound \eqref{eq:Thursday4a} on $\|u\|_{H^2_\h(\Otr)}$.
Using, e.g., the trace estimate from \cite[Corollary 4.2]{GaLaSp:21} (in a similar way to the end of the proof of Lemma \ref{l:absorbingInverse}), we have 
\beq\label{eq:R4}
\|\gamma_{1,\h}^Du\|_{L^2(\Gamma_D)}
\leq  C\h^{-1}\|(-\h^2\Delta-\lambda^2-z1_{\Otr})u\|_{L^2(\Otr)}+C\langle z\rangle\|u\|_{L^2(\Omega_{\tr})}.
\eeq
Furthermore, by \eqref{eq:highlighter2} there exists $\e>0$ small enough such that for $\Im z>0 $ and $|z|\leq \e \h$, $(\mc{P}_{\theta,Q}(\lambda,z))^{-1}$ exists, and then, by  Lemma~\ref{l:absorbingInverse} and reducing $\e$ further if necessary,
\begin{align*}
\N{u}_{H_{\h}^2(\Omega_+)}\leq C\h^{-1}\N{(P_\theta -\lambda^2 -z1_{\Otr})u}_{L^2(\Omega_+)}+C\N{(Q_b\gamma_{1,\h}^D+\gamma_0^D)u}_{H_{\h}^{3/2}(\Gamma_D)}.
\end{align*}
By \eqref{eq:Q_bsymbol} and the Calderon-Vaillancourt theorem (see, e.g., \cite[Proposition E.24]{DyZw:19}, \cite[Theorem 13.13]{Zw:12}), $\|Q_b\|_{L^2(\Gamma_D)\rightarrow H^{3/2}_\h(\Gamma_D)}\leq C$. Using this along with \eqref{eq:R4}, the fact that $P_\theta=-\h^2\Delta$ on $\Otr$, and \eqref{eq:Thursday4}, we obtain
\begin{align*}
\N{u}_{H_{\h}^2(\Omega_+)} \leq C\big(\h^{-1}+\langle z\rangle^2(\Im z)^{-1}\big)\N{(P_\theta -\lambda^2 -z1_{\Otr})u}_{L^2(\Omega_+)}+C\N{\gamma_0^Du}_{H_{\h}^{3/2}(\Gamma_D)},
\end{align*}
which implies \eqref{eq:Thursday4a}; the proof is complete.
\end{proof}

\

Having proved the bound \eqref{eq:bound2new}, we now prove the bound \eqref{eq:bound1new}. From~\eqref{e:parametrix}, 
\beq\label{eq:Thursday5}
R_\theta(\lambda,z)=\mc{R}_{\theta,Q}(\lambda,z)(I +K(\lambda,z))^{-1}\begin{pmatrix}I\\0\end{pmatrix}
\eeq
where $K(\lambda,z)$ is defined by~\eqref{e:compactPerturb}.
Since we have the bound \eqref{eq:highlighter2} on $\mc{R}_{\theta,Q}(\lambda,z)$, to bound $R_\theta(\lambda,z)$ we only need to bound $(I +K(\lambda,z))^{-1}$.

Let $\mc{H}:=L^2(\Omega_+)\oplus H_{\h}^{3/2}(\Gamma_D)$.
Recalling the definition of trace class operators (see  \cite[Definition B.17]{DyZw:19}) and \cite[Equation B.4.7]{DyZw:19}, since $\mc{R}_{\theta,Q}(\lambda,z)$ exists for $|z|\leq \e \h$, $K(\lambda,z)$ defined by  \eqref{e:compactPerturb} is trace class for $|z|\leq \e \h$ with 
\begin{align*}
\|K(\lambda,z)\|_{\tc(\mc{H};\mc{H})}
&\leq \|Q_b\|_{\tc(L^2(\Gamma_D); H^{3/2}(\Gamma_D))}\|\gamma_{1,\h}^DR_{\theta,Q}(\lambda,z)\|_{\mc{H}\to L^2(\Gamma_D)},\\
&\leq C\big\|\langle hD\rangle^{3/2}Q_b\big\|_{\tc(L^2(\Gamma_D))}\|\gamma_{1,\h}^DR_{\theta,Q}(\lambda,z)\|_{\mc{H}\to L^2(\Gamma_D)}.
\end{align*}
Then, using similar reasoning to that in \cite[Page 434]{DyZw:19} to bound the norm of $\langle hD\rangle^{3/2}Q_b$ together with the bound \eqref{eq:highlighter2} on
$\gamma_{1,\h}^DR_{\theta,Q}(\lambda,z)$, we have
\beq\label{eq:Thursday5a}
\|K(\lambda,z)\|_{\tc(\mc{H};\mc{H})}\leq C \h^{1-d} \h^{-1}\leq C\h^{-d}.
\eeq
Furthermore, by \cite[Equation B.5.21]{DyZw:19} and \cite[Equation B.5.19]{DyZw:19}, 
\begin{align}\nonumber
\N{(I+K(\lambda,z))^{-1}}_{\mc{H}\rightarrow \mc{H}}&\leq \det \big(I+K(\lambda,z)\big)^{-1}\det \big(I+[K(\lambda,z)^*K(\lambda, z)]^{1/2}\big),\\
\nonumber
&\leq \det \big(I+K(\lambda,z)\big)^{-1} \exp\big(\|[K(\lambda,z)^*K(\lambda, z)]^{1/2}\|_{\tc(\mc{H})}\big),\\
&\leq \det \big(I+K(\lambda,z)\big)^{-1} \exp\big(\|[K(\lambda,z)\|_{\tc(\mc{H})}\big),\label{eq:Thursday1}
\end{align}
where we have used the definition of the trace class norm $\|\cdot\|_{\tc}$ in terms of singular values (see \cite[Equation B.4.2]{DyZw:19}) to write 
$$
\big\|[K(\lambda,z)^*K(\lambda, z)]^{1/2}\big\|_{\tc(\mc{H})}= \N{K(\lambda, z)}_{\tc(\mc{H})}.
$$
Using \eqref{eq:Thursday5a}  in \eqref{eq:Thursday1}, we find that
\begin{equation}
\label{e:estimate1}
\N{(I+K(\lambda,z))^{-1}}_{L^2\rightarrow L^2}\leq \det \big(I+K(\lambda,z)\big)^{-1}
\exp(C\h^{-d})\quad \tfor |z|\leq \e \h.
\end{equation}
To estimate $ \det (I+K(\lambda,z))^{-1}$ we use the same idea used to prove the bound \eqref{eq:bound1}, namely the following complex-analysis result

\begin{lemma}\mythmname{\cite[Equation D.1.13]{DyZw:19}}\label{lem:complex}
Let $\Omega_0 \Subset \Omega_1 \Subset \mathbb{C}$, let $f$ be holomorphic in a neighbourhood of $\Omega_1$ with zeros $z_j, j=1,2,\ldots$, and let $z_0\in \Omega_1$.
There exists $C=C(\Omega_0,\Omega_1,z_0)$ such that for any $\delta>0$ sufficiently small
\beqs
\log |f(z)| \geq -C \log\big(\delta^{-1}\big)
 \left(\max_{z\in\Omega_1} \log |f(z)| - \log |f(z_0)|\right) \quad\tfor z \in \Omega_0\setminus \bigcup_j B(z_j,\delta).
\eeqs
\end{lemma}

Applying this result with $f(z) = \det (I+K(\lambda,z))$, we see that to get an upper bound on $\log \det (I+K(\lambda,z))^{-1}$ we only need a 
 lower bound on $\det (I+K(\lambda,z_0))$ for some $|z_0|\leq \e \h$ and an upper bound on $\det (I+K(\lambda,z))$ for all $|z|\leq \e\h$.

To obtain the upper bound for all $|z|\leq \e\h$, we again use \cite[Equation B.5.19]{DyZw:19} and \eqref{eq:Thursday5a} to obtain
\begin{equation}
\label{e:estimate2}
|\det (I+K(\lambda, z))|\leq \exp(\|K(\lambda,z)\|_{\tc})\leq \exp(C\h^{-d})\quad \tfor |z|\leq \e \h.
\end{equation}
To obtain the lower bound for some $|z_0|\leq \e \h$, we first observe that, from \eqref{eq:Thursday3},
$$
\big(I+K(\lambda,z)\big)^{-1}=\mc{P}_{\theta,Q}(\lambda,z)\mc{P}_{\theta}(\lambda,z)^{-1}=I-Q\mc{P}_\theta(\lambda,z)^{-1}.
$$
so that 
$$
\big|\det \big(I +K(\lambda,z)\big)\big|^{-1}=\big|\det \big(I-Q\mc{P}_\theta(\lambda,z)^{-1}\big)\big|.
$$
Since $Q\mc{P}_\theta(\lambda,z)$ is trace class, we use \cite[Equation B.5.19]{DyZw:19}, \cite[Equation B.4.7]{DyZw:19}, \eqref{eq:Thursday5a}, and \eqref{eq:Thursday4a} to obtain
\begin{equation}
\label{e:lower}
\log \big|\det (I +K(\lambda,z_0))\big|^{-1}\leq \N{Q}_{\tc(H_{\h}^2(\Omega_+);\mc{H})}\N{\mc{P}_{\theta}(\lambda,z_0)^{-1}}_{\mc{H}\rightarrow H_{\h}^2(\Omega_+)}\leq C\h^{-d} \tfor z_0=\ri\e \h.
\end{equation}
Therefore, combining Lemma \ref{lem:complex}, \eqref{e:estimate2}, and \eqref{e:lower}, we have 
\beqs
\log \big|\det \big(I+K(\lambda,z )\big)^{-1}\big|\leq C\h^{-d}\log \delta^{-1},\qquad z\in B(0,\e_1 \h)\Big\backslash \bigcup_{z_j}B(z_j,\delta)
\eeqs
where $z_j$ are the poles of $(I+K(\lambda,z))^{-1}.$ Therefore, combining this last bound with~\eqref{eq:Thursday5}, \eqref{e:estimate1}, and \eqref{eq:highlighter2}, we have
\beqs
\N{R_{\theta}(\lambda,z)}_{L^2(\Omega_+)\rightarrow L^2(\Omega_+)}\leq \exp\Big(C\h^{-d}\log \delta^{-1}\Big)\quad\tfor z\in B(0,\e_1 \h)\Big\backslash \bigcup_{z_j}B(z_j,\delta).
\eeqs
where $z_j$ are the poles of $\mc{R}_{\theta}(\lambda,z)$.
The bound \eqref{eq:bound1new} and the fact that $z_j$ are the poles of $R_{\Otr}(\lambda,z)$ then follow from 
the relation \eqref{e:scaledUnscaled} and Lemma~\ref{l:inverseForm}.

\section{Proofs of Theorems \ref{thm:main1h} and \ref{thm:main2h}}
\label{sec:mainproof}

\subsection{Proof of Theorem \ref{thm:main1h}}
With Lemma \ref{l:mainEstimate} in hand, this proof is very similar to
\cite[Proof of Theorem 7.6]{DyZw:19}, except that now we work in the complex $z$ plane as opposed to the complex $\lambda$ plane. 
In addition, in this proof, the roles of $\eps_0$ and $\eps$ are swapped compared to \cite[Proof of Theorem 7.6]{DyZw:19}.

Let 
\beq\label{eq:eps0}
\eps_0(\h):= \h^{-\alpha} \eps(\h),
\eeq
with $\alpha>3(d+1)/2$ (we see later where this requirement comes from).
The lower bound \eqref{eq:lowerboundBurq} then implies that, given $\h_0$, there exists $C'$
(depending on $\h_0$ and $\alpha$) such that 
\beq\label{eq:log}
\log\left(\frac{2}{\eps_0(\h)}\right)\leq \frac{C'}{\h} \quad\tfa 0<\h\leq \h_0.
\eeq

Seeking a contradiction, we assume that when $\h=\h_j$ there are no eigenvalues in  $B(0,\eps_0(\h_j))$
(the exponential lower bound on $\eps_0(\h)$ leading to \eqref{eq:log} therefore limits how small this ball can be).
Our goal is to show that this assumption implies that 
\beq\label{eq:contradict}
\N{R(1,0)}_{L^2(\Otr)\rightarrow L^2(\Otr)} < \frac{1}{2} \big(\eps(\h_j)\big)^{-1}.
\eeq
Indeed, since $\supp u_\ell\Subset \Omega_1$, 
\beq\label{eq:contradict-1}
R(1,0) (-\h_j^2 \Delta -1) u_\ell = u_\ell.
\eeq
Then, by taking the norm of \eqref{eq:contradict-1} and using \eqref{eq:contradict}, we obtain that $\N{u_\ell}_{L^2(\Otr)}< 1/2$, which  contradicts $\N{u_\ell}_{L^2(\Otr)}=1$.
We prove \eqref{eq:contradict} by using Theorem \ref{thm:scmp} where $\Omega(\h)$ is a box (to be specified below) in $B(0,\eps_0(\h)/2)$ 
with Lemma \ref{l:mainEstimate} providing the bounds  \eqref{eq:expbound} and \eqref{eq:absbound}. 

We first use the bound \eqref{eq:bound1new} from Lemma \ref{l:mainEstimate}. 
This bound is valid for $z\in B(0,\eps_1 \h)$ and away from the poles. 
The definition of $\eps_0(\h)$ \eqref{eq:eps0} and the upper bound in \eqref{eq:epslowerbound} implies that $B(0,\eps_0(\h)/2)\subset B(0,\eps_1 \h)$ for $\h$ sufficiently small.
We then choose $\delta$ in \eqref{eq:bound1new} to equal $\eps_0(\h)/2$ {and use \eqref{eq:log}}
 so that, for all $\h_j$ sufficiently small,
\beq\label{eq:expbound2}
\N{R(1,z)}_{L^2(\Otr)\rightarrow L^2(\Otr)} \leq  \exp\Big({C_1 C'} \h_j^{-(d+1)}\Big) \quad\tfa z \in B(0, \eps_0(\h_j)/2),
\eeq
and thus for all $z\in \Omega(\h_j)$ (since $\Omega(\h_j) \subset  B(0, \eps_0(\h_j)/2)$).
We now let 
\beqs
Q(z,\h):= R_{\Otr}(1,z), \quad L:=d+1,\quad\tand C:=\max\big\{C_1 C'\,,\, C_2\, c\big\},
\eeqs
where $c = c(\h_0)$ is chosen large enough such that $\langle z\rangle \leq c$ for all 
$z\in B(0,\eps_0(\h)/2)$ and $\h \leq h_0$; these choices ensure that the right-hand sides of the bounds 
\eqref{eq:expbound2} and \eqref{eq:bound2new} are bounded by the right-hand sides of \eqref{eq:expbound} and \eqref{eq:absbound} respectively.
We then let 
\beqs
{w=0,}\quad 2\beta(\h)= \frac14 \eps_0(\h), \quad\tand\quad \delta(\h) = M \eps(\h)
\eeqs
with $M$ chosen (sufficiently large) later in the proof.
For the assumptions of Theorem \ref{thm:scmp} to hold at $\h=\h_j$, we need that (i) the box $\Omega(\h_j)$ defined by \eqref{eq:box} is inside $B(0,\eps_0(\h_j)/2)$ (so that the bound \eqref{eq:expbound} follows from \eqref{eq:expbound2})
and (ii) the second inequality in \eqref{eq:restrict1} is satisfied. The first requirement is ensured if 
\beqs
\delta(\h_j) \h_j^{-(d+1)} \ll \frac12 \eps_0(\h_j), \quad \text{ that is } \quad M\eps(\h_j) \h_j^{-(d+1)} \ll \frac12\h_j^{-\alpha}\eps(\h_j),
\eeqs
which is satisfied if $\h_j$ is sufficiently small since $\alpha >d+1$.
The second requirement is
\beqs
\frac18 h^{-2\alpha} \eps(\h)^2 \geq C \h^{-3(d+1)} M \eps(\h)^2;
\eeqs
given $M$, this inequality is satisfied when $\h$ is sufficiently small since $\alpha> 3(d+1)/2$.

Therefore, 
the assumptions of Theorem \ref{thm:scmp} are all satisfied at $\h=\h_j$ (for $\h_j$ sufficiently small), and the result is that the bound \eqref{eq:scmp1} holds for all $z\in [-\beta(\h_j),\beta(\h_j)]$, and thus, in particular, at $z=0$. Therefore, for all $\h_j$ sufficiently small,
\beqs
\N{R(1,0)}_{L^2(\Otr)\rightarrow L^2(\Otr)} \leq \frac{C}{M\eps(\h_j)} \exp(1+C).
\eeqs
We now choose
\beqs
M:= 2C\exp(1+C),
\eeqs
and obtain \eqref{eq:contradict}, i.e.~the desired contradiction to there being no eigenvalues in $B(0,\eps_0(\h_j))$.

\subsection{Proof of Theorem \ref{thm:main2h}}

We first recall the following lemma proved in \cite[Lemma 4]{St:99}; see also \cite[Lemma AII.20]{La:93}.
\ble\label{lem:linindep}
Let $f_1,\ldots, f_N$ be $N$ vectors in a Hilbert space $\mc{H}$ with 
\beqs
\big| \langle f_i, f_j \rangle_{\mc{H}} - \delta_{ij}\big| \leq \eps \quad\tfa\,\, i,j=1,\ldots,N.
\eeqs
If $\eps< N^{-1}$, then $f_1,\ldots, f_N$ are linearly independent.
\ele

We use Lemma \ref{lem:linindep} both in the proof of Theorem \ref{thm:main2h} below, and in the proof of the following preparatory result.

\ble\label{lem:preparation}
Let $m(\h_j)$ and $\eps(\h)$ be as in Theorem \ref{thm:main2h} (so that, in particular $\eps(\h)\ll \h^{(5d+3)/2}$ as $\h \tendo$).
Then there exists $C>0$ (independent of $\h_j$) such that
\beq\label{eq:Weyl1}
m(\h_j) \leq C \h_j^{-d}.
\eeq
\ele

\bpf
First observe that it is sufficient to prove the result for sufficiently small $\h_j$ (equivalently, sufficiently large $j$).
Let $P(\h_j) = -\h_j^2 \Delta$ with zero Dirichlet boundary conditions on $\Gamma_D$ and $\Gamma_{\tr}$. $P(\h_j)$ is therefore self-adjoint with discrete spectrum and, since $\supp u_{j,\ell} \subset \mc{K}\Subset \Omega_1$, 
\beqs
\N{\big(P(\h_j) - E_{j,\ell}\big)u_{j,\ell}}_{L^2(\Otr)} = \eps(\h_j) \quad \tfa j,\ell.
\eeqs

Let $\mu>c>0$, let $\Pi(\h_j)$ be the orthogonal projection on to the eigenspaces corresponding to all eigenvalues of $P(\h_j)$ in $[a_0-\mu, b_0+ \mu]$,
and let $M(\h_j)$ be the number of these eigenvalues (counting multiplicities). 
By the Weyl law (with no remainder term) 
on manifolds with boundary (see e.g.~\cite[Theorem 17.5.3]{Ho:85}),
\beqs
M(\h_j) \leq C \h_j^{-d}.
\eeqs
Furthermore, $\rank \Pi(\h_j)\leq M(\h_j)$ and thus to prove the result \eqref{eq:Weyl1} it is sufficient to prove that $m(\h_j)\leq \rank \Pi(\h_j)$. 
To keep expressions compact, we now write $P$ and $\Pi$ instead of $P(\h_j)$ and $\Pi(\h_j)$.

Since $\Pi$ commutes with $(P-E_{j,\ell})^{-1}$, and $(P-E_{j,\ell})$ is invertible on $(I-\Pi)L^2$,
\beq\label{eq:combine1}
\big(I-\Pi
\big) u_{j,\ell} = \big(P
- E_{j,\ell}\big)^{-1} \big(I- \Pi
\big) \big(P
- E_{j,\ell}\big) u_{j,\ell}.
\eeq
Since $P$ is self-adjoint, the spectral theorem (see, e.g., \cite[Theorem B.8]{DyZw:19}) implies that
\beq\label{eq:combine2}
\big\| \big(P
- E_{j,\ell}\big)^{-1} \big(I- \Pi
\big)\big\|_{L^2(\Otr)\rightarrow L^2(\Otr)} \leq \frac{1}{\mu}.
\eeq
Therefore, combining \eqref{eq:combine1} and \eqref{eq:combine2}, we have
\beqs
\N{\big(I-\Pi
\big) u_{j,\ell} }_{L^2(\Otr)\rightarrow L^2(\Otr)}\leq \frac{\eps(\h_j)}{\mu}
\eeqs
(compare to \cite[Equation 32.2]{La:93} and the first displayed equation in \cite[\S3]{St:99}).
Then, for $\ell_1, \ell_2 \in \{1,\ldots,m(\h_j)\}$,
\begin{align}\nonumber
\big| \langle\Pi u_{j, \ell_1}, \Pi u_{j,\ell_2}\rangle_{L^2(\Otr)} - \delta_{\ell_1 \ell_2}\big|
&\leq \big|  \langle u_{j,\ell_1}, u_{j,\ell_2}\rangle_{L^2(\Otr)} - \delta_{\ell_1 \ell_2}\big| \\ \nonumber
&\qquad+ 
\big| \langle u_{j,\ell_1}, (I-\Pi) u_{j,\ell_2} \rangle_{L^2(\Otr)} \big|
+\big| \langle(I-\Pi) u_{j,\ell_1},\Pi u_{j,\ell_2} \rangle_{L^2(\Otr)} \big|,\\
& \leq \h_j^{-2}\eps(\h_j) + \frac{2}{\mu} \eps(\h_j)\label{eq:orthog},\\
& {\ll \h_j^{(5d-1)/2} \quad\tas j\tendi},\nonumber
\end{align}
where we have used that $\N{\Pi}_{L^2(\Otr)\rightarrow L^2(\Otr)}\leq 1$ since $\Pi$ is orthogonal.
By Lemma \ref{lem:linindep}, any subset of $\{\Pi u_{j,\ell}\}_{\ell=1}^{m(\h_j)}$ with cardinality $\ll \h_j^{-(5d-1)/2}$ is linearly independent.
Seeking a contradiction assume that \eqref{eq:Weyl1} does not hold, i.e.~for all $C>0$ there exists $j$ such that $m(\h_j) > C \h_j^{-d}$.
Choose a subset of $\{\Pi u_{j,\ell}\}_{\ell=1}^{m(\h_j)}$ with cardinality $\lfloor C \h_j^{-d}+1 \rfloor$. By the above argument, this subset is linearly independent, and thus $\lfloor C \h_j^{-d}+1 \rfloor \leq \rank \Pi(\h_j)=  M(\h_j)\leq C \h_j^{-d}$ which is the required contradiction. 
\epf

\

\bpf[Proof of Theorem \ref{thm:main2h}]
The proof is similar to that of the corresponding ``quasimodes to resonances'' result \cite[Theorem 1]{St:99} (see also \cite[\S7.7, Exercise 1]{DyZw:19}), except that we use the semiclassical maximum principle in the $z$ plane  (as in the proof of Theorem \ref{thm:main1h}), and now we also work in an interval in $\lambda$ (as opposed to at $\lambda=1$ in the proof of Theorem \ref{thm:main1h}). To keep the expressions compact, we write $\h$ instead of $\h_j$ and write functions of the index $j$ as functions of $\hbar$; in particular, we drop the subscript $j$ on $\h_j, E_{j,\ell}$, and $u_{j,\ell}$.

Let 
\beqs
\mc{Z}:= \mc{Z}\big(\e_1(\h)\,,\,\e_0(\h)\,,\,a(\h),\,b(\h)\,;\,\h\big),
\eeqs
where $\mc{Z}(\e_1,\e_0,a,b;\h)$ is  defined by \eqref{eq:mathcalZ}, $\e_0(\h)$ is as in the statement of the theorem, and $\e_1(\h)\ll \h$ will be fixed later.
We assume throughout that  $|\mathcal{Z}|<\infty$, since otherwise the proof is trivial. Let $\Pi(\h)$ denote the orthogonal projection onto
$$
\bigcup_{p\in \mathcal{Z}}\Pi_{z_p}(L^2(\Otr)),
$$
where $\Pi_{z_p}$ is defined in \eqref{eq:projection1}.  Let 
$\widetilde{\mathcal{Z}}(\lambda)$ be the set of distinct values of $z_p(\h,\lambda)$ such that $p\in \mathcal{Z}$. 
{(While $\mc{Z}$ is independent of $\lambda$, $\widetilde{\mc{Z}}$ 
depends on $\lambda$ since 
the poles of $z\mapsto R_\Otr(z,\lambda)$ 
depend on $\lambda$.)}
Note that for $z_p\neq z_q$, $\rank (\Pi_{z_p}+\Pi_{z_q})=\rank \Pi_{z_p}+\rank\Pi_{z_q}$;
therefore
\beqs
\rank \Pi(\h) = \sum_{z_p\in \widetilde{\mathcal{Z}}(\lambda)} \rank \Pi_{z_p(\h,\lambda)} = \sum_{z_p\in\widetilde{\mathcal{Z}}(\lambda)} m_R\big(z_p(\h,\lambda)\big)=|\mathcal{Z}|,
\eeqs
where $m_R(z_0)$ is defined in \eqref{eq:projection1}. To prove the theorem, therefore, it is sufficient to show that 
$m(\h)\leq \rank \Pi(\h)$. 

Seeking a contradiction, we assume that $\rank \Pi(\h)< m(\h)$. 
By Lemma \ref{lem:singular}, near $z_p$, the singular part of $R_\Otr(\lambda ,z)$ is in the range of $\Pi_{z_p}(\h,\lambda)$, and therefore $z\mapsto (I-\Pi(\h))R_\Otr(\lambda ,z)$ is holomorphic on 
\beqs
\Omega(\h):= \big( -2\eps_1(\h), \, 2\eps_1(\h)\big) - \ri \big(0,  2\eps_0(\h)\big)
\eeqs
for all $\lambda^2\in[a(\h),b(\h)]$.
Let $\widetilde{\Omega}(\h)\subset \Omega(\h)$ be defined by
\beqs
\widetilde{\Omega}(\h):= \big( -\eps_1(\h), \, \eps_1(\h)\big) - \ri \big(0,  \eps_0(\h)\big).
\eeqs
Our goal is to apply the semiclassical maximum principle (Theorem \ref{thm:scmp}) in subsets of $\widetilde{\Omega}(\h)$ with $Q(z,\h) = (I-\Pi(\h)) R_\Otr(\lambda,z)$.

By Lemma \ref{l:mainEstimate}, the fact that $\max(\eps_0,\eps_1)\ll \h$, and the fact that $\Pi(\h)$ is orthogonal (and so $\|I-\Pi(\h)\|_{L^2(\Otr)\rightarrow L^2(\Otr)}\leq 1$),
\beq\label{eq:analytic1}
\big\|\big(I-\Pi(\h)\big)R_{\Otr}(\lambda,z)\big\|_{L^2(\Otr)\to L^2(\Otr)}\leq \exp\Big(C_1\h^{-d}\log \delta^{-1}\Big)\quad\tfor z\in \widetilde{\Omega}(\h)
\Big\backslash \bigcup_{m}B(z_m(\h,\lambda),\delta)
\eeq
{and for $\lambda^2 \in[a(\h),b(\h)]$}, 
where the $z_m(\h,\lambda)$ are the poles of $R_\Otr(\lambda,z)$ such that $B(z_m(\h,\lambda),\delta)\cap \widetilde{\Omega}(\h) \neq \emptyset$. 
If $\delta> \min\{ \eps_0(\h), \eps_1(\h)\}$, then these $z_m(\h,\lambda)$ might include poles that are not equal to $z_p(\h,\lambda)$ for some $p\in\mathcal{Z}$, but we restrict $\delta$ so that this is not the case. 
Indeed, we now choose $\delta>0$ so that the bound in \eqref{eq:analytic1} holds for all $z\in \widetilde{\Omega}(\h)$ {and for all $\lambda^2 \in [a(\h),b(\h)]$}.

If $\delta$ and $z_m$ are such that $B(z_m,\delta)\Subset \Omega(\h)$, then the bound in \eqref{eq:analytic1} holds on $\partial B(z_m,\delta)$, and then, since $z\mapsto (I-\Pi(\h))R_{\Otr}(\lambda ,z)$ is holomorphic in $\Omega(\h)$, the maximum principle implies that the bound in \eqref{eq:analytic1} holds in $B(z_m,\delta)$.
We now restrict $\delta$ so that there cannot be a connected union of $B(z_m,\delta)$ that intersects both $\widetilde{\Omega}(\h)$ and $\partial \Omega(\h)$. Once this is ruled out, the maximum principle and the fact that $z\mapsto (I-\Pi(\h))R_{\Otr}(\lambda,z)$ is holomorphic in $\Omega(\h)$ imply that the bound in \eqref{eq:analytic1} holds in $\widetilde{\Omega}(\h)$.
Since we have assumed that $\rank \Pi(\h)<m(\h)$, and $m(\h)\leq C\h^{-d}$ by~\eqref{eq:Weyl1}, there exist a maximum of $C\h^{-d}$ of balls of radius $\delta$. In particular, the maximum distance between any two points in such a connected union is bounded by $2C\delta\h^{-d}$ and hence, a connected union intersecting both $\partial\Omega(\h)$ and $\tilde{\Omega}(\h)$ is ruled out if 
\beq\label{eq:noconnected}
2C\delta \h^{-d} < \min\big\{\eps_0(\h ),\eps_1(\h )\big\}.
\eeq
We now assume that $\eps_0(\h )\leq \eps_1(\h )$ and set 
\beqs
\delta:= \frac{\eps_0(\h ) \, \h ^d}{4 C},
\eeqs
so that \eqref{eq:noconnected} holds.
The lower bound on $\eps_0(\h)$ in \eqref{eq:conditioneps0} and the lower bound on $\eps(\h)$ \eqref{eq:lowerboundBurq} imply that, given $\h_0$, there exists $C'$
(depending on $\h_0$) such that 
\beqs
\log \delta^{-1}\leq \frac{C'}{\h} \quad\tfa 0<\h\leq \h_0.
\eeqs
Therefore, the end result is that, if $\h$ is sufficiently small, 
\beqs
\big\|\big(I-\Pi(\h)\big)R_{\Otr}(\lambda,z)\big\|_{L^2(\Otr)\to L^2(\Otr)}\leq \exp\Big(C\h^{-d-1}\Big)\tfor z\in \widetilde{\Omega}(\h)\tand \lambda^2\in [a(\h),b(\h)],
\eeqs
where {$C:= \max\{C_1 C', c \,C_2\}$, 
where, as in the proof of Theorem \ref{thm:main1h}, $c = c(\h_0)$ is chosen large enough such that $\langle z\rangle \leq c$ for all 
$z\in \widetilde{\Omega}(\h)$ and $\h \leq h_0$.}

We apply the semiclassical maximum principle (Theorem \ref{thm:scmp}) with
\beqs
w=0,\quad \beta(\h)= \eps_1(\h), \quad\delta(\h) =\h ^{d+1}\eps_0(\h ),\quad \tand\quad L=d+1,
\eeqs
and we now fix $\eps_1(\h )$ as
\beqs
\eps_1(\h ):= \frac{\h ^{(d+1)/2} \eps_0(\h )}{C};
\eeqs
observe that this definition of $\eps(\h )$ satisfies both the second requirement in \eqref{eq:restrict1} and our previous assumption that $\eps_0(\h )\leq \eps_1(\h )$.
The result of Theorem \ref{thm:scmp} is that 
\beq\label{eq:boundlambdarange}
\begin{gathered} 
\big\|\big(I-\Pi(\h)\big)R_{\Otr}(\lambda ,z)\big\|_{L^2(\Otr)\to L^2(\Otr)}\leq  C \exp(C+1)\frac{\h ^{-(d+1)}}{\eps_0(\h )} \\ \tfor z \in \big[-\eps_1(\h ), \eps_1(\h )\big]\quad \tand\quad \lambda^2\in [a(\h),b(\h)].
\end{gathered}
\eeq
The definitions of $E_{\ell}$ and $u_{\ell}$ imply that
\beqs
\big(I-\Pi(\h)\big)R_\Otr\big(\sqrt{E_\ell}, 0\big)\big(-\h ^2 \Delta - E_{\ell}\big) u_{\ell} = \big(I-\Pi(\h)\big) u_{\ell},
\eeqs
for $\ell=1,\ldots,m(\h)$. Since $E_{\ell} \in [a(\h), b(\h)]$ for all $\ell$, {the fact that the bound \eqref{eq:boundlambdarange} holds for all $\lambda^2\in [a(\h),b(\h)]$} implies that 
\beqs
\big\|(I-\Pi(\h)) u_{\ell}\big\|_{L^2(\Otr)\to L^2(\Otr)}\leq C \exp(C+1)\h ^{-(d+1)}\frac{\eps(\h )}{\eps_0(\h )}
\eeqs
for $\ell=1,\ldots, m(\h )$.
Therefore 
\beqs
\Big| \big\langle \Pi(\h) u_{\ell_1}, \Pi(\h)  u_{\ell_2}\big\rangle_{L^2(\Otr)} - \delta_{\ell_1 \ell_2}\Big|
\leq \eps(\h ) + 2C \exp(C+1)\h ^{-(d+1)}\frac{\eps(\h )}{\eps_0(\h )}
\eeqs
(compare to \eqref{eq:orthog}, but note that now the projection $\Pi$ is different). 
Using the inequality \eqref{eq:qualitysmall} and the second inequality in \eqref{eq:conditioneps0}, we have 
\beqs
\Big| \big\langle \Pi(\h) u_{\ell_1}, \Pi(\h)  u_{\ell_2}\big\rangle_{L^2(\Otr)} - \delta_{\ell_1 \ell_2}\Big| 
\ll \h^{d}
\quad\text{ and thus } \quad 
\Big| \big\langle \Pi(\h) u_{\ell_1}, \Pi(\h)  u_{\ell_2}\big\rangle_{L^2(\Otr)} - \delta_{\ell_1 \ell_2}\Big| 
\leq \frac{\h^{d}}{C},
\eeqs
where $C$ is the constant in \eqref{eq:Weyl1}. 
By \eqref{eq:Weyl1} and Lemma \ref{lem:linindep}, $\{\Pi(h) u_{\ell}\}_{\ell=1}^{m(\h_j)}$ are linearly independent, and thus $\rank \Pi(\h)\geq m(\h)$, which is the desired contradiction to the assumption that $\rank \Pi(\h)< m(\h)$.
\epf

\section*{Acknowledgements}

EAS  gratefully acknowledges discussions with Alex Barnett (Flatiron Institute) that started his interest in eigenvalues of discretisations of the Helmholtz equation under strong trapping.
JG thanks Maciej Zworski (UC Berkeley) for bringing to his attention the paper \cite{St:00}.
PM thanks Pierre Jolivet (IRIT, CNRS) for his help with the software FreeFEM.
The authors thank the referees for their careful reading of the paper and constructive comments. This research made use of the Balena High Performance Computing (HPC) Service at the University of Bath.
PM and EAS were supported by EPSRC grant EP/R005591/1.

\appendix 

\section{From eigenvalues to quasimodes}\label{a:eToQ}

\begin{lemma}[From eigenvalues to quasimodes in $\h$ notation]\label{lem:evaluestoquasimodes}
Suppose that there exist $z=\cO(\h^\infty)$ and $u$ satisfying~\eqref{e:semiEigenvalue} with $\|u\|_{L^2(\Otr)}=1$. 
Let $\chi\in C_c^\infty(\Omega_1)$ with $\chi \equiv 1$ in a neighborhood of $\projx(K)$. 
Then $\chi u$ is a quasimode (in the sense of Definition \ref{def:quasimodesh}) of quality $\e(\h)=\cO(\h^\infty)$ satisfying
$$
\|u-\chi u\|_{H_{\h}^2(\Otr)}=\cO(\h^\infty).
$$
\end{lemma}
\begin{proof}
The proof is similar to the proof of the ``resonances to quasimodes'' result of \cite[Theorem 1]{St:00}, except that we avoid using results about $\DtN$ for strictly convex obstacles that are used in \cite{St:00} and instead use a commutator argument.

First observe that 
$$(-\h^2\Delta-1-z)u=0\qquad\text{ in }\Otr,$$
so that
$$u= \indicatorR R(1,0)\indicatorE \,z\, u.$$
Therefore,
$$u= \indicatorR R_\theta(1,0)\indicatorE \,z\,u$$
by \eqref{e:resolveAgree} and the definition of $R_\theta(\lambda,z)$ \eqref{eq:Rtheta}.
Let 
\beq\label{eq:defv}
v=R_\theta(1,0)\indicatorE z\,u,
\eeq 
and observe that $v= u$ on $\Otr$.

We now claim that, since $z=\cO(\h^\infty)$ and $\Otr\Subset \mathbb{R}^d$,  $\WFh(v)\subset \Gamma_+$ (defined by \eqref{e:trappedSets}). 
By the definition of the wavefront set \cite[Definition E.36]{DyZw:19}, this is equivalent to $A v= \cO(\h^\infty)$ for all $A$ with $\WFh(A)\subset (\Gamma^+)^c$.
This then follows by noting that $(P_\theta-1)v=\cO(\h^\infty)_{L^2_{\comp}}$ and applying~\cite[Theorem E.47]{DyZw:19},~\cite[Section 24.4]{Ho:85},~\cite[Theorem 8.1]{Va:08}
\footnote{Strictly speaking~\cite[Theorem E.47]{DyZw:19} is used away from the boundary and~\cite[Theorem 8.1]{Va:08} is written for the time dependent problem, but the semiclassical version can be easily recovered by applying the time dependent results to $\re^{\ri t/\h}v(x)$. It is then necessary to use the arguments in~\cite[Section 24.4]{Ho:85} to obtain the `diffractive improvement' i.e. that singularities hitting a diffractive point follow only the flow of $H_p$ rather than sticking to the boundary. A careful examination of \cite[Lemma 24.4.7]{Ho:85} shows that the norm on the error term on $(P_\theta-1)v$ is correct.}
(with, in the notation of \cite[Theorem E.47]{DyZw:19}, $B_1=I$, $B={\bf P}= P_\theta-1$),
together with the facts that $\sigma_\h(\Im (P_\theta-1))\leq 0$ and that $P_\theta-1$ is elliptic on $\{\r\geq 2r_1\}$ (so that if $(x_0,\xi_0)\in \WFh(A)$ then there exists $T\geq 0$ such that $\varphi_{-T}(x_0,\xi_0) \in \operatorname{ell}_\h (P_\theta-1)$).

Now let $\chi\in C_c^\infty(\Omega_1)$ with $\chi \equiv 1$ in a neighborhood of $\projx(K)$. We claim that $\chi v=\chi u$ is a quasimode with quality $\e(\h)=\cO(\h^\infty)$.
To prove this, since
\beq\label{eq:A0}
\N{u-\chi u}_{H^2_\h(\Otr)} = \N{(1-\chi)v}_{H^2_\h(\Otr)} = \N{(1-\chi)v}_{H^2_\h(\Otr\setminus \{\chi\equiv 1\})},
\eeq
it is sufficient to prove that $v$ is $\cO(\h^\infty)_{H_{\h,\loc}^2}$ outside a compact set.

{Our first step is to prove that,
with $r_0<a<b<r_1$,} for $\h$ sufficiently small,
\begin{equation}
\label{e:elliptic2}
\|v\|_{L^2(\r>a)}\leq C\h^{-1}\|(P_\theta-1)v\|_{L^2(\Omega_+)}+{C}\|v\|_{L^2(a<\r<b)},
\end{equation}
{where here, and in the rest of the proof, $C$ denotes a constant, independent of $\h$ and $z$, whose value may change from line to line.}
To prove \eqref{e:elliptic2}, first observe that, since $P_\theta-1$ is elliptic on $\r\geq 2r_1$, {by \cite[Theorem E.33]{DyZw:19} (more precisely its proof together with the calculus from~\cite[Chapter 4]{Zw:12}}),
$$
\|v\|_{L^2(\r>3r_1)}\leq C\|(P_\theta-1)v\|_{L^2(\Omega_+)}+C_Nh^N\|v\|_{L^2(\r>2r_1)}
$$
and hence
\begin{equation}
\label{e:elliptic1}
\|v\|_{L^2(\r>3r_1)}\leq C\|(P_\theta-1)v\|_{L^2(\Omega_+)}+C_Nh^N\|v\|_{L^2(2r_1<\r<4r_1)}.
\end{equation}
Next, observe that there exists $T>0$ such that for all $\rho\in \Gamma_+\cap \{\frac{a+b}{2}<\r<4r_1\}$, there exists $0\leq t\leq T$ such that $a<\r(\varphi_{-t}(\rho))<b.$ In particular, using~\cite[Theorem E.47]{DyZw:19} again, we have 
$$
\|v\|_{L^2(\frac{a+b}{2}<\r<4r_1)}\leq C\h^{-1}\|(P_\theta-1)v\|_{L^2(\Omega_+)}+\|v\|_{L^2(a<\r<b)} +C_Nh^N\|v\|_{L^2(\r>a)}.
$$
{Using this and \eqref{e:elliptic1} in 
\beqs
\|v\|_{L^2(\r > a)} \leq \|v\|_{L^2(a<\r < b)} +\|v\|_{L^2(\frac{a+b}{2}<\r<4r_1)} +\|v\|_{L^2(\r >3r_1)},
\eeqs
we obtain \eqref{e:elliptic2} for $\h$ sufficiently small.}

The next part of the proof involves using a commutator argument to control (up to $h^\infty$ errors) $\|v\|_{L^2(a<\r < b)}$ by the norm on a slightly bigger region and with a gain of $\h$ (see \eqref{eq:A1} below).
Let $\psi \in C_c^\infty (-r_1,r_1)$ with $\psi \equiv 1$ on $\{|x|\leq r_0\}$,  $x\psi'(x)\leq 0$, and $x\psi'(x)<0$ on $a\leq |x|\leq b$.   Then,
\begin{align*}
&2\h^{-1}\Im \big\langle (-\h^2\Delta-1)v,\psi(\r)v\big\rangle_{L^2(\Omega_+)}\\
&\quad\qquad=-\ri\h^{-1}\Big(\big\langle (-\h^2\Delta-1)v,\psi(\r)v\big\rangle_{L^2(\Omega_+)}-\big\langle \psi(\r)v,(-\h^2\Delta-1)v\big\rangle_{L^2(\Omega_+)}\Big)\\
&\quad\qquad=\ri\h^{-1}\big\langle [-\h^2\Delta, \psi(\r)]v,v\big\rangle_{L^2(\Omega_+)}\\
&\quad\qquad=\big\langle  \big(2\psi'(\r)\h D_r-\ri\h[\Delta(\psi(\r))]\big)v,v\big\rangle_{L^2(\Omega_+)}
\end{align*}
By the definition of $\Gamma_+$ \eqref{e:trappedSets}, $\sigma_\h(\psi'(\r)hD_r)=\psi'(\r)\langle \xi,\tfrac{x}{|x|}\rangle<-c<0$ on  $\Gamma_+\cap \{a\leq \r\leq b\}$. Therefore, 
since $\WFh(v)\subset \Gamma_+$, for $\psi_1\in C_c^\infty(r_0<\r<r_1)$ with $\psi_1\equiv 1$ in a neighbourhood of $\supp \partial \psi(\r)$
$$
2\h^{-1}\Im \big\langle (-\h^2\Delta-1)v	,\psi(\r)v\big\rangle_{L^2(\Omega_+)}
\leq -c\|v\|^2_{L^2(a<\r <b)}+C\h\|\psi_1v\|_{L^2(\Omega_+)}^2 +C_N\h^N\|v\|^2_{L^2(\Omega_+)},
$$  
by the microlocal Garding inequality \cite[Proposition E.34]{DyZw:19} (with $A= - \psi'(\r) \h D_r -c$ and $B$ supported 
in $\langle \xi, x/|x|\rangle < \epsilon$, i.e., away from $\Gamma^+$, and in $\{r_0\leq \r\leq r_1\}$).
Therefore, by Young's inequality,
\beq\label{eq:A1}
\|v\|^2_{L^2(a<\r<b)}\leq C\h^{-N-2}\|(-\h^2\Delta-1)v\|_{L^2(\r<r_1)}^2+C\h\|\psi_1 v\|^2_{L^2(\Omega_+)}+C_N\h^N\|v\|_{L^2(\Omega_+)}^2
\eeq
We now use the propagation estimate again to control (up to $\h^\infty$ errors) $\|\psi_1 v\|^2_{L^2(\Omega_+)}$ by $\|v\|^2_{L^2(a<\r<b)}$.
Suppose that $\rho\in\r^{-1}(\{  \supp \psi_1\})\cap \Gamma_+$. Then, there exists $|t|\leq \sqrt{r_1^2-r_0^2}$ such that $\varphi_t(\rho)\in \{a<\r<b\}$. 
Therefore, by standard propagation estimates ~\cite[Theorem E.47]{DyZw:19}, again using that $\WFh(v)\subset \Gamma_+$, we have 
\beq\label{eq:A2}
\|\psi_1 v\|^2_{L^2(\Omega_+)}\leq C\h^{-1}\|(-\h^2\Delta-1)v\|_{L^2(\r \leq r_1)}^2+C\|v\|^2_{L^2(a<\r<b)}
+C_N\h^N\|v\|_{L^2(\Omega_+)}^2.
\eeq

We next use the propagation estimate again to control $\| v\|_{L^2(\{\r \leq r_1\}\setminus \{\chi\equiv 1\})}$ by 
$\|v\|^2_{L^2(a<\r<b)}$. To do this, we need that
there exists $T>0$ such that for all $\rho\in S^*_{\Omega_+\setminus\{\chi\equiv 1\}}\Omega_+$ with $\r(\rho)\leq r_1$ there is $|t|\leq T$ with $a<\r(\varphi_t(\rho))<b$.  Suppose not; then there exist $\rho_n\in S^*_{\Omega_+\setminus\{\chi\equiv 1\}}\Omega_+$ with $\r(\rho_{n})\leq r_1$ and $T_n \to \infty$ {such that}
$$
\bigcup_{|t|\leq T_n}\varphi_t(\rho_{{n}})\cap \{a<\r<b\}=\emptyset.
$$
By~\eqref{e:directEscape1}, we have $\r(\rho_n)\leq r_0$ and also $\r(\varphi_{\pm T_n}(\rho_n))\leq r_0$.
In particular, we may assume that $\rho_n\to \rho\in \{\r\leq r_0\}\setminus K$ 
(since $\projx(K)\Subset \{\chi \equiv 1\}$)
and $\varphi_{\pm T_n}(\rho_n)\to \rho_\pm$. Then, by Lemma~\ref{l:trapping}, $\rho\in \Gamma_+\cap \Gamma_-=K$, which is a contradiction. 
Applying the propagation estimate {(using the existence of the uniform time $T$)}, we have
\begin{equation}
\label{e:propagateInOut}
\| v\|^2_{L^2(\{\r \leq r_1\}\setminus \{\chi\equiv 1\})}\leq C\h^{-1}\|(-\h^2\Delta-1)v\|_{L^2(\r \leq r_1)}^2+C\|v\|^2_{L^2(a<\r<b)}+C_N\h^N\| v\|_{L^2(\Omega_+)}^2.
\end{equation}

Finally, we control $\|v\|_{L^2(\Omega_+\setminus \Otr)}$. For this, note that, $v=u1_{\Otr}+v1_{(\Otr)^c}$ and by~\eqref{e:elliptic2} and~\eqref{e:propagateInOut} we have
\begin{equation}
\label{e:exteriorb}
\|v\|_{L^2(\Omega_+\setminus \Otr)}\leq C\h^{-1}\|(P_\theta-1)v\|_{L^2(\Omega_+)}+C\|v\|_{L^2(a<\r<b)}+C_N\h^N\|u\|_{L^2(\Otr)}.
\end{equation}

Now, using \eqref{eq:A2} in \eqref{eq:A1}, 
and then using the definition of $v$ \eqref{eq:defv} and that $v=u$ on $\Otr$, 
we have
\begin{align*}
\|v\|^2_{L^2(a<\r<b)}&\leq C\h^{-N-2}\|(-\h^2\Delta-1)v\|_{L^2(\r \leq r_1)}^2+C_{N}\h^N\|v\|_{L^2(\Omega_+)}^2\\
&=C\h^{-N-2}\|(P_\theta-1)v\|_{L^2(\Omega_+)}^2+C_{N}\h^N\|u\|_{L^2(\Otr)}^2+C_{N}\h^N\|v\|_{L^2(\Omega_+\setminus \Otr)}^2.
\end{align*}
Then, using \eqref{e:exteriorb},
\begin{align*}
\|v\|^2_{L^2(a<\r<b)}&\leq C_{N}\h^N\|u\|_{L^2(\Otr)}^2+ C_{N}\h^N\|v\|_{L^2(a<\r<b)},
\end{align*}
and, taking $\h$ small enough, we obtain
$$
\|v\|_{L^2(a<\r<b)}\leq C_{N}\h^N\|u\|_{L^2(\Otr)}\leq C_{N}\h^N,
$$
since $\|u\|_{L^2(\Otr)}=1$.
Therefore, using~\eqref{e:propagateInOut},~\eqref{e:exteriorb},  the definition of $v$ \eqref{eq:defv}, and the fact that $z= \cO(\h^\infty)$, we have
$$
\| \psi(\r)v\|^2_{L^2(\Omega_+\setminus \{\chi\equiv 1\})}=\cO(\h^\infty).
$$
so that, since
$\WFh(v)\subset S^*\Rea^d$ (which is fibre compact),
$$
\| \psi(\r)v\|^2_{H_\h^2(\Omega_+\setminus \{\chi\equiv 1\})}=\cO(\h^\infty);
$$
the result then follows from \eqref{eq:A0}.
\end{proof}

\section{Details of how the eigenvalues/eigenfunctions were computed in \S\ref{sec:num}}\label{sec:DtN}

When discretising the sesquilinear form \(a(\cdot, \cdot)\) defined by~\eqref{eq:sesqui}, 
we need to calculate the Dirichlet-to-Neumann map $\DtN(k)$. Instead of approximating $\DtN(k)$ using either a perfectly-matched layer (PML) or an absorbing boundary condition, we use boundary integral operators to find $\DtN(k)$ ``exactly'' (i.e.,~up to the discretisation of these integral operators).

Recall that the single-layer potential on $\Gtr$ is defined for $\varphi \in L^1(\Gamma)$ by
\begin{align*}
  \mathcal{S}_k \varphi (x) := \int_{\Gtr} \Phi_k (x,y) \varphi (y)\, ds (y) \quad \tfa x \in \mathbb{R}^d \setminus \Gtr,
\end{align*}
where, in 2-d, $\Phi_k(x,y):= \ri H_0^{(1)}(k |x-y|)/4$, where \(H^{(1)}_0\) is the order zero Hankel function of the first kind. 
The single-layer and adjoint-double-layer operators are then defined, respectively, by 
\(S_k := \gamma^{\tr}_0 \mathcal{S}_k\) and \( D'_k := \gamma^{\tr}_1 \mathcal{S}_k -I/2 \), where the traces are taken from inside $\Otr$.
With these definitions, for values of $k$ for which $S_k: H^{-1/2}(\Gamma)\rightarrow H^{1/2}(\Gamma)$ is invertible,
\beq\label{eq:DtN}
\DtN(k)= \left(-\frac{1}{2}I+D'_k\right)S_k^{-1};
\eeq
see, e.g., \cite[Page 136]{ChGrLaSp:12}.

To avoid the operator product in \eqref{eq:DtN}, we introduce the auxiliary variable \(\varphi_\ell = S_k^{-1} (\gamma_0^{\tr}(u_\ell)) \in H^{-1/2}(\Gtr)\). The eigenvalue problem~\eqref{eq:eigenfunction} can therefore be rewritten as: find $u_\ell \in H^1_{0,D}(\Otr)$ and $\varphi_\ell \in H^{-1/2}(\Gtr)$ such that
\begin{align}\nonumber
\big(\nabla u_\ell, \nabla v\big)_{L^2(\Otr)} - k^2 \big(u_\ell ,v\big)_{L^2(\Otr)}  
    - \left\langle \left(-\frac{1}{2}I+D'_k\right) \varphi_\ell,\gamma_0^{\tr} v\right\rangle_{\Gtr} =&\,\, \mu_\ell \big(u_\ell,v\big)_{L^2(\Otr)},\\
\tand\qquad  \big\langle \gamma_0^{\tr} u_\ell, \psi \big\rangle_{\Gtr} - \big\langle S_k \varphi_\ell, \psi \big\rangle_{\Gtr} =& \,\,0,
\label{eq:coupled}
\end{align}
for all \(v \in H^1_{0,D}(\Otr)\) and \(\psi \in H^{-1/2}(\Gtr)\). We note that this formulation is the transpose of the Johnson--N\'ed\'elec FEM-BEM coupling \cite{JoNe:80} applied to the eigenvalue problem \eqref{eq:eigenfunction}; see, e.g., \cite[Equation 9]{GaHsSa:12}.

We use continuous piecewise-linear basis functions to discretise \eqref{eq:coupled}
, and obtain the following generalised eigenvalue problem
\begin{equation}\label{eq:discretised_eigenproblem}
  \widetilde{\mathbf{A}}  \mathbf{u}_\ell
  =
  \begin{pmatrix}
    \mathbf{A}_k &  \dfrac{1}{2}(\mathbf{M}^{\tr})^T - \mathbf{D}'_k \\
    \mathbf{M}^{\tr} & -\mathbf{S}_k
  \end{pmatrix}
  \mathbf{u}_\ell
  =
  \mu_\ell 
  \begin{pmatrix}
    \mathbf{M} &0 \\
    0 & 0 
  \end{pmatrix}
  \mathbf{u}_\ell=: \mu_\ell \mathbf{B} \mathbf{u}_\ell,
\end{equation}
where $\mathbf{M}$ is the mass matrix on $\Otr$, $\mathbf{S}_k$ is a discretisation of the single-layer operator, and $\mathbf{A}_k$ is the Galerkin matrix corresponding to the discretisation of $(\nabla u_\ell, \nabla v)-k^2 (u_\ell, v)$. The matrices $\mathbf{M}^{\tr}$ and $\mathbf{D}'_k$ are defined by 
\begin{align*}
  (\mathbf{M}^{\tr})_{i,j} = \langle \gamma_0^{\tr} v_j, \psi_i \big\rangle_{\Gtr} \quad \textrm{and}\quad(\mathbf{D}'_k)_{j,i} = \left\langle D'_k \psi_i,\gamma_0^{\tr} v_j\right\rangle_{\Gtr},\quad \textrm{for }i=1\dots M\textrm{ and }j=1\dots N,
\end{align*}
where \(v_j\) and \(\psi_i\) are, respectively, continuous piecewise-linear basis functions of the Galerkin discretisations \(V_h(\Otr)\subset H^1_{0,D}(\Otr)\) and \(V_h(\Gtr)\subset H^{-1/2}(\Gtr)\); the dimensions of these spaces are denoted \(N\) and \(M\) respectively.

To build the matrices in \eqref{eq:discretised_eigenproblem} and solve this problem, we use PETSc~\cite{Balay1997,Balay2019,Balay2020} and the eigensolver SLEPc~\cite{slepc-manual,slepc-toms} via the software FreeFEM~\cite{Hecht2012}. Since we are interested in the eigenvalues near the origin, we use the shift-and-invert technique, i.e, we compute the largest eigenvalues of the problem \((\widetilde{\mathbf{A}})^{-1} \mathbf{B} \mathbf{u}_\ell = \nu_\ell \mathbf{u}_\ell\), and then set \(\mu_\ell=1/\nu_\ell\). To obtain the action of $(\widetilde{\mathbf{A}})^{-1}$, we use SuperLU~\cite{lidemmel03} to compute the LU factorisation of \(\widetilde{\mathbf{A}}\).

\footnotesize{
\newcommand{\etalchar}[1]{$^{#1}$}

}

\end{document}